\documentclass{amsart}

\usepackage{fullpage}
\usepackage{lmodern}
\usepackage{amsmath}
\usepackage{amsfonts}
\usepackage{amssymb}
\usepackage{graphicx}
\usepackage{mathrsfs}
\usepackage{amsthm}
\usepackage{enumerate}
\usepackage{leftidx}
\usepackage{hyperref}
\usepackage{extarrows}
\usepackage[all]{xy}
\usepackage{stmaryrd}
\usepackage{wasysym}
\usepackage[OT2,T1]{fontenc}
\usepackage{epstopdf}
\usepackage {hyperref}
\DeclareMathOperator{\ad}{ad}

\DeclareMathOperator{\diag}{diag}

\DeclareMathOperator{\coh}{\mathbf {H}}

\DeclareMathOperator{\Ind}{Ind}
\DeclareMathOperator{\Int}{Int}

\DeclareMathOperator{\Res}{Res}

\DeclareMathOperator{\id}{id}

\DeclareMathOperator{\GL}{GL}

\DeclareMathOperator{\SO}{SO}

\DeclareMathOperator{\UU}{U}

\DeclareMathOperator{\GSpin}{GSpin}
\DeclareMathOperator{\Sp}{Sp}

\newcommand{\RZ}{\mathcal{N}}
\newcommand{\RZO}{\mathrm{RZ}}
\DeclareMathOperator{\Spf}{Spf}

\DeclareMathOperator{\inv}{inv}

\newcommand{\abs}[1]{\left\vert#1\right\vert}
\newcommand{\set}[1]{\left\{#1\right\}}
\newcommand{\sett}[1]{\left(#1\right)}

\newcommand{\To}{\longrightarrow}
\newcommand{\isom}{\overset{\sim}{\To}}

\newcommand{\ZZ}{\mathbb{Z}}
\newcommand{\QQ}{\mathbb{Q}}

\newcommand{\GG}{\mathbb{G}}

\newcommand{\TT}{\mathcal{T}}

\newcommand{\rr}{\mathfrak{r}}
\newcommand{\MMM}{\mathscr M}

\newcommand{\ignore}[1]{}
\DeclareSymbolFont{cyrletters}{OT2}{wncyr}{m}{n}
\DeclareMathSymbol{\Sha}{\mathalpha}{cyrletters}{"58}

\usepackage{enumitem}

\setlist[enumerate]{leftmargin=*}
\setlist[itemize]{leftmargin=*}

\RequirePackage{xspace}
\RequirePackage{etoolbox}
\RequirePackage{varwidth}
\RequirePackage{enumitem}
\RequirePackage{tensor}
\RequirePackage{mathtools}
\RequirePackage{longtable}
\RequirePackage{multirow}

\def\ge{\geqslant}
\def\le{\leqslant}

\def\o{\omega}

\def\s{\sigma}

\def\i{^{-1}}

\def\<{\langle}
\def\>{\rangle}

\newcommand{\BF}{\ensuremath{\mathbb {F}}\xspace}
\newcommand{{\BG}}{\ensuremath{\mathbb {G}}\xspace}

\newcommand{{\BK}}{\ensuremath{\mathbb {K}}\xspace}

\newcommand{\BS}{\ensuremath{\mathbb {S}}\xspace}

\newcommand{\BZ}{\ensuremath{\mathbb {Z}}\xspace}

\newcommand{\Ad}{{\mathrm{Ad}}}
\DeclareMathOperator{\Orb}{Orb}

\DeclareMathOperator{\tr}{tr}

\DeclareMathOperator{\supp}{supp}

\newtheorem{theorem}[subsubsection]{Theorem}
\newtheorem{proposition}[subsubsection]{Proposition}
\newtheorem{lemma}[subsubsection]{Lemma}

\newtheorem{corollary}[subsubsection]{Corollary}

\theoremstyle{definition}
\newtheorem{definition}[subsubsection]{Definition}

\newtheorem{remark}[subsubsection]{Remark}

\numberwithin{equation}{subsection}

\setitemize[0]{leftmargin=*,itemsep=\the\smallskipamount}
\setenumerate[0]{leftmargin=*,itemsep=\the\smallskipamount}

\renewcommand{\to}{%
   \ifbool{@display}{\longrightarrow}{\rightarrow}%
   }
\let\shortmapsto\mapsto
\renewcommand{\mapsto}{%
   \ifbool{@display}{\longmapsto}{\shortmapsto}%
   }
\newlength{\olen}
\newlength{\ulen}
\newlength{\xlen}
\newcommand{\xra}[2][]{%
   \ifbool{@display}%
      {\settowidth{\olen}{$\overset{#2}{\longrightarrow}$}%
       \settowidth{\ulen}{$\underset{#1}{\longrightarrow}$}%
       \settowidth{\xlen}{$\xrightarrow[#1]{#2}$}%
       \ifdimgreater{\olen}{\xlen}%
          {\underset{#1}{\overset{#2}{\longrightarrow}}}%
          {\ifdimgreater{\ulen}{\xlen}%
             {\underset{#1}{\overset{#2}{\longrightarrow}}}
             {\xrightarrow[#1]{#2}}}}%
      {\xrightarrow[#1]{#2}}
   }
\makeatother
\newcommand{\xyra}[2][]{%
   \settowidth{\xlen}{$\xrightarrow[#1]{#2}$}%
   \ifbool{@display}%
      {\settowidth{\olen}{$\overset{#2}{\longrightarrow}$}%
       \settowidth{\ulen}{$\underset{#1}{\longrightarrow}$}%
       \ifdimgreater{\olen}{\xlen}%
          {\mathrel{\xymatrix@M=.12ex@C=3.2ex{\ar[r]^-{#2}_-{#1} &}}}%
          {\ifdimgreater{\ulen}{\xlen}%
             {\mathrel{\xymatrix@M=.12ex@C=3.2ex{\ar[r]^-{#2}_-{#1} &}}}
             {\mathrel{\xymatrix@M=.12ex@C=\the\xlen{\ar[r]^-{#2}_-{#1} &}}}}}%
      {\mathrel{\xymatrix@M=.12ex@C=\the\xlen{\ar[r]^-{#2}_-{#1} &}}}%
   }
\makeatletter
\newcommand{\xla}[2][]{%
   \ifbool{@display}%
      {\settowidth{\olen}{$\overset{#2}{\longleftarrow}$}%
       \settowidth{\ulen}{$\underset{#1}{\longleftarrow}$}%
       \settowidth{\xlen}{$\xleftarrow[#1]{#2}$}%
       \ifdimgreater{\olen}{\xlen}%
          {\underset{#1}{\overset{#2}{\longleftarrow}}}%
          {\ifdimgreater{\ulen}{\xlen}%
             {\underset{#1}{\overset{#2}{\longleftarrow}}}
             {\xleftarrow[#1]{#2}}}}%
      {\xleftarrow[#1]{#2}}
   }
\newcommand{\isoarrow}{%
   \ifbool{@display}{\overset{\sim}{\longrightarrow}}{\xrightarrow\sim}%
   }
   
\begin{document}
\title{Fine Deligne--Lusztig varieties and Arithmetic Fundamental Lemmas}
\author[Xuhua He]{Xuhua He}
\email{xuhuahe@math.cuhk.edu.hk}
\address{Lady Shaw Building,
	The Chinese University of Hong Kong,
	Shatin, N.T., Hong Kong}
\author[Chao Li]{Chao Li}\email{chaoli@math.columbia.edu} 
\address{Department of Mathematics, Columbia University, 2990 Broadway,
	New York, NY 10027, USA}

\author[Yihang Zhu]{Yihang Zhu}\email{yihang@math.columbia.edu}
\address{Department of Mathematics, Columbia University, 2990 Broadway,
	New York, NY 10027, USA}


\begin{abstract}
We prove a character formula for some closed fine Deligne--Lusztig varieties. We apply it to compute fixed points for fine Deligne--Lusztig varieties arising from the basic loci of Shimura varieties of Coxeter type. As an application, we prove an arithmetic intersection formula for certain diagonal cycles on unitary and GSpin Rapoport--Zink spaces arising from the arithmetic Gan--Gross--Prasad conjectures. In particular, we prove the arithmetic fundamental lemma in the minuscule case, without assumptions on the residual characteristic.
\end{abstract}

\maketitle
\tableofcontents

\section{Introduction}

\subsection{The AFL conjecture}

The \emph{arithmetic Gan--Gross--Prasad (AGGP) conjectures} generalize the celebrated Gross--Zagier formula to higher dimensional Shimura varieties of orthogonal or unitary type (\cite[\S 27]{Gan2012}, \cite[\S 3.2]{Zhang2012}, \cite{Rapoport2017a}). The \emph{arithmetic fundamental lemma conjecture} (AFL) arises from Zhang's relative trace formula approach towards the AGGP conjecture for the group $\UU(1, n-2)\times \UU(1, n-1), n\geq 2$. It relates a derivative of orbital integrals on symmetric spaces to an arithmetic intersection number of cycles on unitary Rapoport--Zink spaces,
\begin{equation}
\label{eq:AFL}
\omega(\gamma) \cdot \partial \Orb(\gamma, \mathbf{1}_{S_n(\mathcal{O}_{F})})=-\Int(g)\cdot \log q.
\end{equation}
For the precise definitions of the quantities appearing in the identity, see \cite[\S 1]{Rapoport2017}. The left-hand side of (\ref{eq:AFL}) is known as the \emph{analytic side} and the right-hand side is known as the \emph{arithmetic-geometric} side.

Let us briefly recall the definition of the arithmetic-geometric side. Let $p$ be an odd prime. Let $F$ be a finite extension of $\mathbb{Q}_p$ with residue field $\mathbb{F}_q$ and a uniformizer $\pi$. Let $E$ be an unramified quadratic extension of $F$. Let $\breve E$ be the completion of the maximal unramified extension of $E$. Let $k = \overline{ \BF}_q$. For any integer $n\geq 1$, the \emph{unitary Rapoport--Zink space} $\mathcal{N}_n$ is the formal scheme over $S=\Spf \mathcal{O}_{\breve E}$ parameterizing deformations up to quasi-isogeny of height 0 of unitary $\mathcal O_F$-modules of signature $(1,n-1)$. Fix an integer $n\geq 2$. There is a natural closed immersion $\delta: \mathcal{N}_{n-1}\rightarrow \mathcal{N}_n$. Denote by $\Delta\subset \mathcal{N}_{n-1}\times_S\mathcal{N}_{n}$ the image of $(\id,\delta): \mathcal{N}_{n-1}\rightarrow \mathcal{N}_{n-1}\times_S\mathcal{N}_n$.

Let $C_{n-1}$ be a non-split Hermitian space of dimension $n-1$, for the quadratic extension $E/F$. Here non-split means that the discriminant has odd valuation. Define a non-split Hermitian space of dimension $n$ by $C_{n}: =C_{n-1} \oplus E u$, where the direct sum is orthogonal and $u$ has norm $1$. The unitary group $J(F):=\UU(C_n)(F)$ acts on $\mathcal{N}_n$ in a natural way. Let $g\in J(F)$. The arithmetic-geometric side of the AFL conjecture (\ref{eq:AFL}) concerns the arithmetic intersection number of the diagonal cycle $\Delta$ and its translate by $\id\times g$, defined as (see \cite[\S 2.2]{Zhang2012}) $$\Int(g):=\chi(\mathcal{N}_{n-1}\times_S\mathcal{N}_n, \mathcal{O}_\Delta \otimes^\mathbb{L}\mathcal{O}_{(\id \times g)\Delta}).$$ When $\Delta$ and $(\id\times g)\Delta$ intersect properly, namely when the formal scheme
\begin{equation}
\label{eq:intersection}
\Delta \cap(\id\times g)\Delta \cong \delta(\mathcal{N}_{n-1})\cap \mathcal{N}^g_n
\end{equation}
is an Artinian scheme (where $\mathcal{N}^g_n$ denotes the fixed points of $g$), the arithmetic intersection number $\Int(g)$ is simply the $\mathcal{O}_{\breve E}$-length of the Artinian scheme (\ref{eq:intersection}) (see \cite[Proposition 4.2 (iii)]{RTZ}).

Recall that $g\in J(F)$ is called \emph{regular semi-simple} if $$L(g):=\mathcal{O}_E\cdot u+\mathcal{O}_E\cdot g u+\cdots+ \mathcal{O}_E\cdot g^{n-1} u$$ is a full-rank $\mathcal{O}_E$-lattice in $C_n$. In this case, the \emph{invariant} of $g$ is the unique sequence of integers $$\inv(g):=(r_1\ge r_2\ge \ldots\ge r_n)$$ characterized by the condition that there exists a basis $\{e_i\}$ of the lattice $L(g)$ such that $\{\pi^{-r_i}e_i\}$ is a basis of the dual lattice $L(g)^\vee$. It turns out that the ``bigger'' $\inv(g)$ is, the more difficult it is to compute the intersection. With this in mind, recall that a regular semi-simple element $g$ is called \emph{minuscule} if $r_1=1$ and $r_n\ge0$.

\subsection{The AFL in the minuscule case}

In the minuscule case, the analytic side is relatively straightforward to evaluate. One of our main results is an explicit formula for the arithmetic-geometric side $\Int(g)$ when $g$ is minuscule, which allows us to establish new cases of the AFL conjecture.

\begin{theorem}[Corollary \ref{cor:AFL}]\label{thm:AFL}
	The arithmetic fundamental lemma holds when $g$ is minuscule.
\end{theorem}

\begin{remark}
	When $F=\mathbb{Q}_p$ and $p>\frac{n+1}{2}$, this theorem was first proved by Rapoport--Terstiege--Zhang \cite{RTZ} (see also a simplified proof in \cite{AFL}). The same methods together with \cite{Cho2018} should prove the theorem for any $p$-adic field $F$ with the size of its residue field $q>\frac{n+1}{2}$. However, potential global applications to the AGGP conjectures require the truth of AFL at \emph{all} unramified places, thus it is desirable to remove the assumption that $q>\frac{n+1}{2}$. Our proof is different from \cite{RTZ} and treats all local fields $F$ (with odd residue characteristic, in order to define the Rapoport--Zink spaces) uniformly.
\end{remark}

\begin{remark}
	After this work was done, Zhang \cite{Zhang2019} has recently announced a proof of the arithmetic fundamental lemma when $F=\mathbb{Q}_p$ and $p>n$ (without assuming that $g$ is minuscule). 
\end{remark}

To state the explicit formula for $\Int(g)$, assume $g$ is minuscule and $\RZ_n^g \neq \emptyset$. Then it can be shown that $g$ stabilizes both $L(g)^\vee$ and $L(g)$, and acts as an unitary operator on $\mathbb V:=L(g)^\vee/L(g)$, which has a natural structure of a Hermitian space over $\BF_{q^2}$. Let $\bar g\in \UU(\mathbb V)(\BF_q)$ be the induced element.

For any monic polynomial $Q \in  \BF_{q^2}[\lambda]$ with $Q(0)\neq 0$, we define its \emph{reciprocal polynomial} $Q^*$ by replacing each root $x\in k^{\times}$ of $Q$ with $x^{-q}$ (with multiplicities). We say $Q$ is \emph{self-reciprocal} if $Q=Q^*$.

Let $f\in \mathbb{F}_{q^2}[\lambda]$ be the characteristic polynomial of $\bar g$. Then $f$ is self-reciprocal. For any monic irreducible  $Q\in\BF_{q^2} [\lambda]$, we denote the multiplicity of $Q$ in $f$ by $m_Q$. 

\begin{theorem}[Theorem \ref{thm:app to AFL}]\label{thm:AFLformula}
	Assume $g$ is minuscule and $\Int(g)\neq 0$. Then there is a unique monic irreducible self-reciprocal $Q_0\in \BF_{q^2}[\lambda]$ such that $m_{Q_0}$ is odd. We have $$\Int(g)=\frac{m_{Q_0}+1}{2}\cdot \deg Q_0\cdot \prod_{\{ Q, Q^* \}}(1+m_Q).$$ Here the product is over pairs $\set{Q, Q^*}$ of monic irreducible non-self-reciprocal polynomials in $\BF_{q^2}[\lambda]$ with non-zero constant terms.
\end{theorem}

Theorem \ref{thm:AFL} then follows immediately from Theorem \ref{thm:AFLformula} and the explicit formula for the analytic side given in  \cite[Proposition 8.2]{RTZ}.

\begin{remark}
	Theorem \ref{thm:AFLformula} is also used to prove the minuscule case of Liu's arithmetic fundamental lemma for Fourier--Jacobi cycles, see \cite[Appendix E]{Liu2018}.
\end{remark}

\begin{remark}\label{rem:orthogonalAFL} In Theorem \ref{thm:app to RZO} we also establish an analogous arithmetic intersection formula for GSpin Rapoport--Zink spaces arising from the AGGP conjectures for orthogonal groups. This provides a new proof of the main result of \cite{RZO}, and also removes the assumption that $p\ge \frac{n+1}{2}$ in \textit{loc.~cit.} 
\end{remark}

\subsection{Computing the arithmetic intersection}
The starting point of the proof of Theorem \ref{thm:AFLformula} is the observation made in \cite[Proposition 4.1.2]{AFL} that, in the minuscule case, the formal scheme (\ref{eq:intersection}) can be identified with the fixed point scheme $\mathcal{V}^{\bar g}$ of an explicitly given smooth projective variety $\mathcal{V}$ over $k$, under a finite-order automorphism $\bar g$. It also turns out that $\mathcal V^{\bar g}$ is an Artinian scheme. Hence $\Int(g)$ is given by the $k$-length of $\mathcal V^{\bar g}$.

In order to compute the $k$-length of $\mathcal V^{\bar g}$, there are two apparent approaches. One approach, taken in \cite{AFL}, is to explicitly study all the local equations. The other approach, which we take in the current paper, is to compute it using the Lefschetz trace formula. Thus we obtain 
\begin{align}\label{eq:int(g)}\Int (g) = \tr\left(\bar g \mid \coh^*(\mathcal{V})\right),
\end{align} where $\coh^*(\mathcal{V})$ denotes the \'etale ${\QQ}_{\ell}$-cohomology of $\mathcal V$, for a fixed prime $\ell \neq p$.

To compute the right hand side of (\ref{eq:int(g)}), we utilize the fact that the variety $\mathcal{V}$ is the closure of a generalized Deligne--Lusztig variety in a partial flag variety of the unitary group $\mathbb G=\UU(\mathbb V)$ over $\mathbb{F}_q$. To be precise, let $G:=\mathbb G_k$, and let $\s$ be the Frobenius automorphism of $k$ over $\BF_q$. Then $\mathcal{V}$ is the closure inside $G/P$ of the generalized Deligne--Lusztig variety $$X_P(w):=\{ h P \in G/P: h^{-1}\sigma(h)\in P w P\},$$ for a certain standard parabolic subgroup $P \subset G$ and a certain $w\in W_P \backslash W / W_P$. Here $W$ denotes the Weyl group of $G$ and $W_P$ denotes the parabolic subgroup of $W$ corresponding to $P$. The automorphism $\bar g$ of $\mathcal V$ is given by the natural action of the group element $\bar g \in \mathbb G(\BF_q)$. 

Vollaard \cite[Theorem 2.15]{Vollaard2010} constructed a nice stratification
\begin{equation}\label{eq:Vol}
\mathcal{V}=\bigsqcup_i X_i 
\end{equation} of $\mathcal V$ into finitely many locally closed strata $X_i$, where each $X_i$ is the image in $G/P$ of a generalized Deligne--Lusztig variety in $G/P_i$ for a different parabolic subgroup $P_i \subset G$. This stratification is remarkable because it is different from the naive decomposition 
$$\mathcal V = \overline {X_P(w)} = \bigsqcup _{w' \in W_P \backslash W /  W_P, w'\leq w} X_P(w').$$
In fact, the stratification (\ref{eq:Vol}) is a special example of \emph{stratification into fine Deligne--Lusztig varieties}, which will be discussed in the next subsection \S \ref{sec:char-form-fine}. Now each $X_i$ turns out to be a fine Deligne--Lusztig variety in $G/P$, and can be related via parabolic induction to a classical Deligne--Lusztig variety in the full flag variety of a Levi subgroup of $G$. In this way, the computation of the right hand side of (\ref{eq:int(g)}) reduces to computing the characters on the cohomology with compact support $\coh ^*_c (X_i)$ for each $X_i$, and eventually reduces to the classical Deligne--Lusztig character formula in \cite{DL}.

We thus place the problem of computing the right hand side of (\ref{eq:int(g)}) into the more general framework of developing a character formula for fine Deligne--Lusztig varieties and their closures.

\subsection{A character formula for fine Deligne--Lusztig varieties}\label{sec:char-form-fine} Let $\mathbb F_q$ be a finite field. Let $k = \overline {\BF}_q$, and let $\sigma$ be the Frobenius automorphism of $k$ over $\BF_q$. Let $\mathbb G$ be a connected reductive group over $\mathbb{F}_q$. Let $G = \mathbb G_{k}$, and let $W$ be the Weyl group of $G$. Let $J$ be a subset of the simple reflections in $W$. Let $W_J$ be the subgroup of $W$ generated by $J$, and let $P_J$ be the corresponding standard parabolic subgroup of $G$. Let ${}^JW$ be the set of minimal length coset representatives of $W_J\backslash W$. For $w\in {}^JW$, we have the associated  \emph{fine Deligne--Lusztig variety} $$X_{J, w}=\{g P_J \in G/P_J; g \i \s(g) \in P_J \cdot_\s B w B\},$$ where $\cdot_\s$ is the $\s$-conjugation action. When $J = \emptyset$, $X_{\emptyset , w}$ recovers the classical Deligne--Lusztig variety $X_w$ inside the full flag variety of $G$, associated to $w$.

In Definition \ref{defn:unbranched}, we will introduce the notion of a \emph{$\sigma$-unbranched datum} $(J, \mathscr{L})$, where $J$ is a set of simple reflections in $W$, and $\mathscr{L}$ is a sub-diagram of the Dynkin diagram of $G$ satisfying certain axioms with respect to $J$. Associated to such $(J, \mathscr{L})$, we will construct canonically a finite sequence of elements $w_i\in {}^JW$, such that we have the following simple closure relation (see Corollary \ref{cor:closure})
\begin{equation}
\label{eq:fineDLstrata}
\overline{X_{J,w_1}}=\bigsqcup_{i} X_{J,w_i}.  
\end{equation}
The above stratification subsumes (\ref{eq:Vol}) as a special case.
Moreover, for each $i$ we will construct a rational parabolic subgroup $\mathbb P_i\subset \mathbb G$, and a projection to a reductive group $\mathbb P_i\rightarrow \mathbb G_i$ over $\BF_q$, such that $w_i$ can be naturally viewed as an element of the Weyl group $W_i$ of $G_i: =\mathbb G_{i,k}$. We show that each fine Deligne--Lusztig variety $X_{J,w_i}$ is related via parabolic induction to the classical Deligne--Lusztig variety $X_{w_i}^{\mathbb G_i}$ in the full flag variety of $G_i$ associated to $w_i$ (see Proposition \ref{prop:inductive}): $$X_{J, w_i} \cong \mathbb  G(\BF_q) \times^{\mathbb P_i(\BF_q)} X^{\mathbb G_i}_{w_i} .$$ 

For each $i$, we fix a $\sigma$-stable  maximal torus $T_i\subset G_i$ of type $w_i$. Now we are ready to state our main character formula.

\begin{theorem}[Theorem \ref{thm:char formula}]\label{thm:mainfixedpoint}
	Assume $(J, \mathscr{L})$ is a $\sigma$-unbranched datum. Let $w_i, \mathbb P_i, \mathbb G_i, T_i$ be as above. Let $g\in \mathbb G(\mathbb{F}_q)$ be a regular element. Then
	\begin{align*}
	\tr(g \mid \coh ^*(\overline{X_{J,w_1}})) & =\sum_{i} \tr(g \mid \coh ^*_c(X_{J,w_i})) \\ \nonumber & =\sum_i \sum_{\gamma\in \Gamma_i}\#\mathcal{M}_i^{g, \gamma}  \cdot \frac{ |\mathbb G_{i,\gamma} (\BF_q)| } {| \mathbb G_{i,\gamma}^0 (\BF_q)|}  \cdot \big| T_i \cap ({}^{\mathbb G_i(\BF_q)} \gamma_i )\big | .    \end{align*}
	Here we have
	\begin{itemize}
		\item $\Gamma_i$ is a complete set of representatives of elements in $T_i(\mathbb{F}_q)$ modulo $\mathbb G_i(\mathbb{F}_q)$-conjugacy.
		\item $\mathcal{M}_i^{g}:=\{r\in \mathbb G(\mathbb{F}_q)/\mathbb P_i(\BF_q);  r^{-1}gr\in \mathbb P_i(\mathbb{F}_q)\}$, and $\mathcal{M}_i^{g,\gamma}\subset \mathcal{M}_i^{g}$ consists of those $r\in \mathcal{M}_i^g$ such that the semi-simple part of the projection of $r^{-1}gr$ to $\mathbb G_i$ is $\mathbb G_i(\mathbb{F}_q)$-conjugate to $\gamma$.
		\item ${}^{\mathbb G_i(\BF_q)} \gamma_i$ is the $\mathbb G_i(\BF_q)$-conjugacy class of $\gamma_i$.
	\end{itemize}
\end{theorem}
\subsection{Four families of fine Deligne--Lusztig varieties}\label{subsec:four families in into}
In \S \ref{sec:some-char-form}, we apply Theorem \ref{thm:mainfixedpoint} to fine Deligne--Lusztig varieties that arise from the basic loci of Shimura varieties of Coxeter type \cite{GH}. There are four infinite families of such fine Deligne--Lusztig varieties, where the $\BF_q$-groups $\mathbb G$ are respectively the even non-split special orthogonal group, the odd special orthogonal group, the symplectic group, and the odd unitary group.

In all these cases, we obtain an explicit formula for $\tr (g \mid \coh ^* (\overline {X_{J,w_1}}))$, for $g \in \mathbb G(\BF_q)$ whose image under the standard representation is regular. Our formula is in terms of the characteristic polynomial of $g$, subsuming the formula in Theorem \ref{thm:AFLformula} as a special case. See Theorems \ref{thm:RZO case}, \ref{thm:odd ortho}, \ref{thm:symp}, \ref{thm:unitary}. The odd unitary cases and the even non-split special orthogonal cases are relevant to the AGGP conjectures for unitary and orthogonal groups respectively, and our formulas lead to the arithmetic intersection formulas in Theorem \ref{thm:AFLformula} and Remark \ref{rem:orthogonalAFL}.  

\subsection{Further remarks on Theorem \ref{thm:AFLformula}} Arguably the most difficult part of Theorem \ref{thm:AFLformula} is to compute the intersection multiplicity at each point of intersection in (\ref{eq:intersection}). The computation in \cite{RTZ} uses Zink's theory of windows and displays to compute the local equations of (\ref{eq:intersection}). It requires explicitly writing down the window of the universal deformation of $p$-divisible groups. The assumption $p> \frac{n+1}{2}$ made in \textit{loc.~cit.~}ensures that the ideal of local equations is admissible (see the last paragraph of \cite[p.~1661]{RTZ}), which is crucial in order to construct the frames for the relevant windows needed in Zink's theory.

As mentioned above, the starting point of the simplified proof in \cite{AFL} is that the intersection (\ref{eq:intersection}) can be identified with $\mathcal{V}^{\bar g}$, and thus a deformation-theoretic problem of $p$-divisible groups is transformed to a purely algebro-geometric problem over $k$. When $p>\frac{n+1}{2}$, the computation of $\mathcal{V}^{\bar g}$ is further reduced in \cite{AFL} to a more elementary fixed point problem of a linear transformation on a projective space. However, when $p\le\frac{n+1}{2}$ the multiplicities remain mysterious.

Our proof of Theorem \ref{thm:AFLformula} shares the same starting point as \cite{AFL}. The new observation is the inductive structure of fine Deligne--Lusztig varieties, which allows us to exploit the full power of the classical character formula of Deligne--Lusztig. Our approach circumvents the need to analyze the local structure of (\ref{eq:intersection}), and gives the desired formula without the extra assumption on $p$. 

Finally, we remark that in the computation in \cite{RTZ} or \cite{AFL}, the number $\frac{m_{Q_0} +1}{2}$ in Theorem \ref{thm:AFLformula} appears as the common intersection multiplicity at each point of intersection. In our current computation, we obtain a different geometric interpretation of this number, as the \emph{number of the strata} $X_i$ whose $\coh^*_c$ contribute non-trivially to the trace (\ref{eq:int(g)}). (In the proofs of Theorem \ref{thm:RZO case} and Theorem \ref{thm:unitary}, this number appears as $\abs{\mathscr I}$.) As a simple illustration of this phenomenon, consider the automorphism $f(x) = x+1$ of order $p$ on $\mathbb P^1$ over $k$. The only fixed point is $\infty$, which has multiplicity $2$. On the other hand, we have an $f$-stable stratification $\mathbb{P}^1=\mathbb{A}^1 \sqcup \{\infty\}$, which gives $$\tr(f \mid  \coh ^*(\mathbb P^1)) = \tr (f \mid \coh^*_c(\mathbb A^1)) + \tr (f \mid \coh ^*(\{\infty\})). $$ Note that $\tr (f \mid \coh^*_c(\mathbb A^1)) = \tr (f \mid \coh ^*(\{\infty\})) =1$. Thus the multiplicity $2$ also appears as the number of contributing strata. 

\subsection{Organization of the paper}In \S \ref{sec:fine DL}, we introduce the notion of a $\s$-unbranched datum, and study the closure relation and inductive structure for the fine Deligne--Lusztig varieties associated to a $\s$-unbranched datum, culminating in the proof of the general character formula Theorem \ref{thm:mainfixedpoint} (Theorem \ref{thm:char formula}). In \S \ref{sec:basic loci}, we recall the four infinite families of fine Deligne--Lusztig varieties arising from basic loci of Coxeter type in Shimura varieties. In each case we identify the unique $\s$-unbranched datum. In \S \ref{sec:some-char-form}, we apply the general character formula to each of the four families in \S \ref{sec:basic loci}, obtaining explicit character formulas in terms of characteristic polynomials (Theorems \ref{thm:RZO case}, \ref{thm:odd ortho}, \ref{thm:symp}, \ref{thm:unitary}). In \S \ref{sec:application}, we apply the results in \S \ref{sec:some-char-form} to  obtain the arithmetic intersection formulas in Theorem \ref{thm:AFLformula} and Remark \ref{rem:orthogonalAFL} (Theorem \ref{thm:app to AFL} and Theorem \ref{thm:app to RZO}).  
\subsection{Notations and conventions} Let $k $ be an algebraically closed field. For a smooth scheme $X$ over $k$, we denote by $\coh^*(X)$ and $\coh^* _c(X)$ the \'etale $\QQ_{\ell}$-cohomology and the \'etale $\QQ_{\ell}$-cohomology with compact support respectively, for a fixed prime $\ell$ which is invertible in $k$.

For any linear algebraic group $G$ over $k$, we identify $G$ with its $k$-points. If a subfield $k_0$ of $ k$ and a $k_0$-form $\mathbb G$ of $G$ are given in the context, we often abuse notation to write $G(k_0)$ for $\mathbb G(k_0)$.

By convention, a \emph{quadratic space} means a finite-dimensional vector space over a field equipped with a non-degenerate quadratic form. Since we will never consider characteristic $2$ fields, we shall specify the quadratic form by specifying its associated bi-linear  pairing. Thus the quadratic form is recovered from the bilinear pairing $[\cdot, \cdot ]$ as $x\mapsto [x,x]/2$. Similarly, \emph{Hermitian forms} and \emph{symplectic forms} are all understood to be non-degenerate.

For any field $F$, we denote by $F[\lambda] ^{\mathrm{monic}}$ the set of monic polynomials in the polynomial ring $F[\lambda]$.

\subsection{Acknowledgments} X.~H.~was partially supported by the NSF grant DMS-1801352. C.~L.~was partially supported by an AMS travel grant for ICM 2018 and the NSF grant DMS-1802269. Y.~Z.~was partially supported by the NSF grant DMS-1802292. We would like to thank the Hausdorff Center for Mathematics for the hospitality, during the Conference on the Occasion of Michael Rapoport's 70th Birthday. We would also like to thank the referees for careful reading and useful comments.

\section{Fine Deligne--Lusztig varieties}\label{sec:fine DL}
\subsection{Basic setting and notations} Fix an odd prime $p$, and let $q$ be a power of $p$.  Let $k=\overline{ \BF}_q$ and $\s$ be the Frobenius automorphism of $k$ over $\BF_q$. 

Let $\mathbb G$ be a connected reductive group over $\BF_q$, and let $G = \mathbb G_k$. We fix a $\s$-stable Borel subgroup $B$ of $G$, with a Levi decomposition $B = TU$ which is also $\s$-stable. Let $W$ be the \emph{canonical Weyl group} of $G$ equipped with the canonical action of the Frobenius $\s$, as in \cite[\S 1.1]{DL}. Then using the pair $(T,B)$ we identify $W$ with $N_G(T)/T$, and the identification is $\s$-equivariant.

Let $\BS$ be the set of simple reflections in $W$. For any $J \subset \BS$, let $P_J \supset B$ be the standard parabolic subgroup of $G$ associated to $J$, and let $L_J$ be the standard Levi subgroup of $P_J$. Denote by $W_J$ the subgroup of $W$ generated by $J$ (called a parabolic subgroup of $W$). Thus $W_J$ is the Weyl group of $L_J$.

For $w \in W$, we denote by $\supp(w)$ the support of $w$, i.e., the set of simple reflections that occur in some (or equivalently, any) reduced expression of $w$. We define $$\supp_\s(w): =\bigcup_{i \in \BZ} \s^i(\supp(w)).$$ 

We recall the notion of Coxeter elements following \cite[7.3]{Sp}. For each $\s$-orbit in $\BS$, we pick a simple reflection. Let $c$ be the product of these simple reflections in any given order. We call such $c$ a {\it $\s$-twisted Coxeter element} of $W$. More generally, for a $\s$-stable subset $\Sigma\subset \BS$, we may consider $\s$-twisted Coxeter elements of the parabolic subgroup $W_{\Sigma }$. If $c$ is such an element, then $\supp_{\s} (c) = \Sigma$, and $\supp (c)$ is a complete set of representatives of the $\s$-orbits in $\Sigma$.

\subsection{Classical Deligne--Lusztig varieties}
For $w\in W$, the (classical) Deligne--Lusztig variety $X_w$ in the full flag variety $G/B$ is defined by $$X_w=\{g B \in G/B; g \i \s(g) \in B w B\}.$$
These Deligne--Lusztig varieties give a partition of the full flag variety $$G/B=\bigsqcup_{w \in W} X_w.$$
The closure relation is given by the Bruhat order $\leq $ of the Weyl group, i.e. for any $w \in W$, $$\overline{X_w}=\bigsqcup_{w' \le w} X_{w'}.$$

\subsection{Fine Deligne--Lusztig varieties} 
Let $J \subset \BS.$ Let $G/P_J$ be the partial flag variety of type $J$. In 1977, Lusztig introduced a partition of $G/P_J$ into fine Deligne--Lusztig varieties. 

We follow the approach in \cite[\S 3]{He-par}. Let ${}^J W$ be the set of minimal length coset representatives of $W_J \backslash W$. For any $w \in {}^J W$, we set $$X_{J, w}=\{g P_J \in G/P_J; g \i \s(g) \in P_J \cdot_\s B w B\},$$ where $\cdot_\s$ is the $\s$-conjugation action, i.e., $x\cdot_\s y : = x y \s(x)\i$. When $J = \emptyset$, we have $X_{\emptyset , w} = X_w$. 

Then we have a partition $$G/P_J=\bigsqcup_{w \in {}^J W} X_{J, w}$$
into locally closed sub-varieties. 

The partial order $\le_{J, \s}$ on ${}^J W$ is introduced in \cite[Proposition 3.8]{He-piece} (see also \cite[4.7]{He-min}). For $w, w' \in {}^J W$, we write $$w \le_{J, \s} w'$$ if $u w \s(u) \i \le w'$ for some $u \in W_J$. By \cite[Proposition 3.13]{He-piece} and \cite[Corollary 4.6]{He-min}, $\le_{J, \s}$ is a partial order on ${}^J W$. Now we have 

\begin{theorem}\cite[Theorem 3.1]{He-par} \label{thm:closure}
	For $w \in {}^J W$, 
	\[\pushQED{\qed}
	\overline{X_{J, w}} =\bigsqcup_{w' \in {}^J W; w' \le_{J, \s} w} X_{J, w'}. \qedhere \popQED \]
\end{theorem}

\subsection{The $\s$-unbranched datum} 
\label{subsec:setting}
We would like to single out certain cases where the right hand side of Theorem \ref{thm:closure} has a relatively simple description.
\begin{definition}\label{defn:unbranched}
	We say that a subset $J \subset \BS$ is \emph{$\s$-unbranched} if the following conditions hold.
	\begin{enumerate}
		\item The set $\BS -J$ is contained in one $\s$-orbit in $\BS$.
		\item There exists a sub-diagram $\mathscr L$ of the Dynkin diagram of $(\mathbb G, W, \BS)$ satisfying the following conditions.
		\begin{itemize}
			\item The diagram $\mathscr L$ is connected and without branching;
			\item The nodes of $\mathscr L$ form a complete set of representatives of the $\sigma$-orbits in $\BS$.  
			\item One (and hence exactly one) end-node of $\mathscr L$ is in $\BS - J$. 
	\end{itemize} 	\end{enumerate}
	We call a pair $(J, \mathscr L)$ as above a \emph{$\s$-unbranched datum} for $\mathbb G$. When we would like to emphasize the group $\mathbb G$, we write $(\mathbb G, J ,\mathscr L)$.
\end{definition}
\subsubsection{}
From now on we assume the existence of a $\s$-unbranched subset $J\subset \BS$, and fix a $\s$-unbranched datum $(J,\mathscr L)$ once and for all. Let $a$ be the number of nodes in $\mathscr L$. By assumption $\mathscr L$ is connected and without branching, with exactly one end-node in $\BS - J$. Hence we may canonically list the consecutive nodes in $\mathscr L$ as 
\begin{align}\label{eq:order of nodes}
\rr_1, \rr_2,\cdots, \rr_{a} \in \BS, 
\end{align}
with $\rr_{a} \in \BS - J$. Write $i_{\max} = a+1.$ 

For each $1 \leq  i \leq i_{\max}$, define
$$w_i : = \rr_a \rr_{a-1} \cdots \rr_{i}.$$ Here by convention
$w_{i_{\max}}: = 1.$
We also define 
\begin{align*}
\Sigma^{\flat}_i  & : = \supp_{\s} w_i = \bigcup _{j=i}^a \text{the $\s$-orbit of }\rr_{j}, \\
\Sigma _i & : = \begin{cases}
\text{the $\s$-orbit of }\rr_{i-1}, & \text{if }2\leq i \leq i_{\max},\\
\emptyset, &  \text{if }i =1,
\end{cases}\\
\Sigma ^{\sharp} _i & : = \BS - (\Sigma _i ^{\flat } \cup \Sigma _i ).
\end{align*}

\begin{lemma}\label{lem:disconn}
	For all $3 \leq i \leq a$ and $m \in \BZ$, the sets $\set{\s ^m (\rr_{i-2}), \s ^m (\rr_{i-3}),\cdots , \s ^m (\rr_1) }$ and $\set{\rr_a, \rr_{a-1},\cdots, \rr_i}$ are disconnected from each other.
\end{lemma}
\begin{proof}
	Firstly, we observe that these two sets do not share any common element, because of the second condition in Definition \ref{defn:unbranched} (2). Now suppose that the two sets are connected. Then there exist integers $j$ and $l$, satisfying $$ 1 \leq l \leq i -2 < i \leq j \leq a ,$$ such that $\rr_j$ is connected with $\s^m (\rr _l)$. Choose $n \in \mathbb N$ such that $\s ^{nm} (\rr _j) = \rr_j$. Then in the list 
	$$ \rr_j, \s ^m \rr_l , \s ^m \rr_{l+1} ,\cdots, \s ^m \rr _j , \s^{2m} \rr_l, \cdots, \s ^{2m} \rr_j, \cdots, \s^{nm} \rr_l, \cdots, \s ^{nm} \rr_j,$$ each member is connected with (and unequal to) its predecessor, and the last member is equal to the first member. Since the Dynkin diagram does not contain loops, there must be a member in the list which equals the second member following it.
	Hence one of the following three situations must happen: 
	\begin{enumerate}
		\item There exist integers $\alpha,\beta$, with $l \leq \beta \leq j-2$, such that $\s^{\alpha m} \rr_{\beta} = \s^{\alpha m} \rr_{\beta+2}$. 
		\item There exists an integer $\alpha$, such that $\s^{\alpha m} \rr_{j-1} = \s^{(\alpha+1) m} \rr_l$. 
		\item There exists an integer $\alpha$, such that $\s^{\alpha m} \rr_{j} = \s^{(\alpha+1) m} \rr_{l+1}$. 
	\end{enumerate}
	Since $j-l \geq 2$, each of these three situations contradicts with the second condition in in Definition \ref{defn:unbranched} (2).
\end{proof}
\begin{lemma}\label{lem:about Sigma}
	For each $1 \leq i \leq i_{\max}$, we have $$\BS = \Sigma _i ^{\flat} \sqcup \Sigma _i \sqcup \Sigma _i ^{\sharp}.$$ The sets $\Sigma _i ^{\flat}, \Sigma _i , \Sigma _i ^{\sharp}$ are all $\sigma$-stable. Moreover $\Sigma _i^{\flat}$ is disconnected from $\Sigma _i ^{\sharp}$.
\end{lemma}
\begin{proof}
	The first assertion holds because $\rr_1,\cdots, \rr_a $ lie in distinct $\sigma$-orbits in $\BS$. The second assertion follows easily from the definition. The third assertion follows from Lemma \ref{lem:disconn}.
\end{proof}

Note that each $w_i$ is $\s$-twisted Coxeter in $W_ {\Sigma_i ^{\flat}}$, and $W_ {\Sigma_1 ^{\flat}} = W_{\BS} = W$. We further have the following result.

\begin{lemma}\label{lem:w_i}
	For each $1 \leq i \leq i_{\max}$, we have $w_i \in {}^J W$. Moreover $$\{ w \in {}^JW ; w \leq _{J,\s} w_1 \} = \{w_1, w_2,\cdots, w_{i_{\max}}\}. $$ 
\end{lemma}
\begin{proof}
	Since $\mathscr L$ is connected and since $\rr_a \in \BS-J$, we have $w_i \in {}^J W$. By definition, $w_i \le w_1$ for any $i$. 
	
	On the other hand, let $w \in {}^J W$ with $w \le_{J, \s} w_1$. Then by \cite[Proposition 3.8]{He-piece}, there exists $u \in W_J$ with $\ell(w \s(u) \i)=\ell(w)-\ell(u)$ and $u w \s(u) \i \le w_1$. Then we have $w \s(u) \i \in {}^J W$ and $w \s(u) \i=w_i$ for some $1 \le i \le i_{\max}$. Then $u w_i \le w_1$. Since $u \in W_J$ and $w_i \in {}^J W$, we have $\ell(u w_i)=\ell(u)+\ell(w_i)$. Note that $\rr_{i-1} w_i \nleq w_1$, so we have $u \le \rr_{i-2} \rr_{i-3} \cdots \rr_1$. By Lemma \ref{lem:disconn}, the sets $\{\s(\rr_{i-2}), \s(\rr_{i-3}), \ldots, \s( \rr_1)\}$ and $\{ \rr_a, \rr_{a-1}, \ldots, \rr_i\}$ are disconnected from each other. Hence $w=w_i \s(u)=\s(u) w_i$. Since $w \in {}^J W$, we have $\s(u)=1$ and hence $w=w_i$. 
\end{proof}

By the above lemma, the fine Deligne--Lusztig variety $X_{J,w_i}$ is defined for each $1\leq i \leq i_{\max}$. 
\begin{corollary}\label{cor:closure}
	We have $$\overline{X_{J, w_1}}=\bigsqcup_{1 \leq i \leq i_{\max}} X_{J, w_i}.$$
\end{corollary}
\begin{proof}
	This follows from Theorem \ref{thm:closure} and Lemma \ref{lem:w_i}. 
\end{proof}

Given $g \in G^{\mathrm{reg}} \cap G(\BF_q)$, our goal in this section is to compute 
$$\tr(g, J ,\mathscr L) : = \tr (g 
\mid  \mathbf H^* (\overline {X_{J,w_1}})). $$  

\begin{corollary}\label{cor:decomp of trace}
	For $g \in G^{\mathrm{reg}} \cap G(\BF_q)$, we have
	$$\tr (g, J, \mathscr L ) = \sum _{i=1} ^{i_{\max}} \tr (g \mid \coh ^*_c (X_{J,w_i})). $$  
\end{corollary} 
\begin{proof}
	This follows from Corollary \ref{cor:closure}.
\end{proof}
\subsection{Parabolic induction}\label{subsec:parabolic induction} We keep the setting of \S \ref{subsec:setting}. Fix $1 \leq i \leq i_{\max}$. Denote 
$$P_i : = P_{\Sigma_i ^{\flat} \sqcup \Sigma _i ^{\sharp}}, \quad L_i: = L_{\Sigma_i ^{\flat} \sqcup \Sigma _i ^{\sharp}}, \quad G_i^{\ad} : = (L_{\Sigma_i ^{\flat}})^{\ad}, \quad H_i ^{\ad} : =  (L_{\Sigma_i ^{\sharp}})^{\ad}. $$
Since $\Sigma_i ^{\flat}$ is disconnected from $\Sigma_i ^{\sharp}$ (see Lemma \ref{lem:about Sigma}), we have a canonical isomorphism 
$$ L_i ^{\ad} \cong G_i^{\ad} \times H_i^{\ad}. $$
Let $L_i \to L _i^{\natural}$ be the central isogeny with the smallest kernel such that $L_i^{\natural}$ is the direct product of the inverse images in $L_i^{\natural}$ of $G_i^{\ad}$ and $H_i^{\ad}$. We denote by $G_i$ (resp.~$H_i$) the inverse image of $G_i^{\ad}$ (resp.~$H_i^{\ad}$) in $L_i ^{\natural}$. Then $G_i^{\ad}$ (resp.~$H_i^{\ad}$) is indeed the adjoint group of $G_i$ (resp.~$H_i$), so the notation is compatible.

Thus we have $L_i ^{\natural} = G_i \times H_i.$ Moreover, since $\Sigma_i ^{\flat}, \Sigma_i ^{\sharp}$ are $\s$-stable, the groups $P_i, L_i, L_i^{\natural}, G_i, H_i$, as well as the central isogeny $L_i \to L_i ^{\natural}$ and the decomposition $L_i ^{\natural} = G_i \times H_i$, are all defined over $\BF_q$. When we would like to emphasize the reductive groups over $\BF_q$ underlying $P_i, L_i$, etc., we shall write $\mathbb P_i,\mathbb L_i$, etc. We let $\pi_i$ denote the projection $P_i \to L_i \to L_i ^{\natural} \to G_i$, and let $\pi'_i$ denote the projection $P_i \to L_i \to L_i ^{\natural} \to H_i$. 

Let $W_i : = W_{\Sigma _i^{\flat}}$. 
Then $W_i$ is identified with the Weyl group of $G_i$, inside which $w_i$ is a $\sigma$-twisted Coxeter element. Let $X_{w_i}^{\mathbb G _i}$ be the classical Deligne--Lusztig variety associated to the element $w_i\in W_i$ in the full flag variety of $G_i$. Then we have a natural action of $G_i(\BF_q)$ on $X_{w_i}^{\mathbb G_i}$. Define the action of the group $P_i(\BF_q)$ on $G(\BF_q) \times X_{w_i}^{\mathbb G_i}$ by $$p \cdot (g, x)=(g p\i, \pi_i(p) \cdot x).$$ Let $G(\BF_q) \times^{P_i(\BF_q)} X^{\mathbb G_i}_{w_i}$ be the quotient space. As a $k$-variety this  is just a finite disjoint union of isomorphic copies of $X^{\mathbb G_i}_{w_i}$. 

\begin{proposition}\label{prop:inductive}
	For each $1 \leq i \leq i_{\max}$, we have a $G(\BF_q)$-equivariant isomorphism \begin{align*}
	G(\BF_q) \times^{P_i(\BF_q)} X^{\mathbb G_i}_{w_i} & \isom  X_{J, w_i}\\  (g, g' (G_i \cap B)) & \mapsto g g' P_J.
	\end{align*} 
\end{proposition}
\begin{proof} We fix $1 \leq i \leq i_{\max}$. We claim that $\Sigma ^ \sharp_i$ is the maximal subset of $J$ that is stable under $\Ad(w_i) \circ \s$. In fact, by definition $\Sigma ^{\sharp}_i$ is a $\s$-stable subset of $J$ (see Lemma \ref{lem:about Sigma}). Since $\Sigma _i ^{\sharp}$ is disconnected from $\Sigma _i ^{\flat}$ by Lemma \ref{lem:about Sigma}, $\Sigma _i ^{\sharp}$ is also stable under $\Ad (w_i)$. Now let $K$ be an arbitrary $\Ad (w_i)\circ \s $-stable subset of $J$. We need to show that $K\subset \Sigma_i^\sharp$. We first show that $K \cap \Sigma_i = \emptyset$. If $i=1$, then $\Sigma _i = \emptyset$ by definition. If $i=i_{\max}$, then $\Sigma_i$ is the $\s$-orbit of $\rr_a$, and $K$ is $\s$-stable (as $w_{i_{\max}} =1$ by convention). In this case, since $\rr_a \notin J$, we must have $K \cap \Sigma
	_i = \emptyset$. Now let $2 \leq i \leq a$. Then $\Ad (w_i) \rr_{i-1}\notin \BS$, and for any $\rr\in \Sigma_i-\set{\rr_{i-1}}$, we have either $\Ad(w_i) \rr \notin \BS$, or $\Ad(w_i) \rr = \rr \notin J$. Hence for all $\rr \in \Sigma_i$ we have $\Ad(w_i) \rr \notin J$, and so $K \cap \Sigma_i = \emptyset$. Thus we have shown that $K\cap \Sigma_i = \emptyset$ in all cases.
	
	Similarly, for any integer $j$ with $i \leq j \leq a$, the following holds. On one hand either $\Ad (w_i) \rr_j = \rr_{j-1}$ or $\Ad (w_i) \rr_j \notin \BS$, and on the other hand, for any $\rr \neq \rr_j$ that is in the $\s$-orbit of $\rr_j$, either $\Ad (w_i) \rr \notin \BS$ or $\Ad (w_i) \rr = \rr \notin J$. Moreover, we have $\Ad (w_i) \rr_i\notin \BS$ if $i<a$, and we have $\Ad(w_a) \rr_a=\rr_a \notin J$. Using this and by induction on $j$, we see that $K$ does not contain any element in the $\s$-orbit of $\rr_j$, for any $j \geq i$. Therefore $K \cap \Sigma _i ^{\flat}  =\emptyset$. We already saw $K \cap \Sigma _i = \emptyset$, so $K \subset \Sigma _i ^{\sharp}.$ This proves our claim that $\Sigma ^ \sharp_i$ is the maximal subset of $J$ that is stable under $\Ad(w_i) \circ \s$.

	By the above claim and by \cite[4.2(d)]{Lu07} (see also \cite[\S 3]{He-par}), the projection map $G/P_{\Sigma _i ^{\sharp}}\to G/P_J$ induces an isomorphism $$X_{\Sigma^\sharp_i, w_i} \isom  X_{J, w_i}.$$ 
	
	Note that $P_{\Sigma^\sharp_i} \cdot_\sigma B w_i B \subset P_i$. Thus $g P_{\Sigma^\sharp_i} \in X_{\Sigma^\sharp_i, w_i}$ implies that $g \i \s(g) \in P_i$. By Lang's theorem, $g \i \s(g) \in P_i$ is equivalent to $g \in G(\BF_q) P_i$. The projection map $G/P_{\Sigma^\sharp_i} \to G/P_i$ induces an isomorphism $$X_{\Sigma^\sharp_i, w_i} \isom  G(\BF_q) \times^{P_i(\BF_q)} X',$$ where $X'$ is the sub-variety of $P_i/P_{\Sigma^\sharp_i}$ given by $$X'=\{p P_{\Sigma^\sharp_i}\in P_i/P_{\Sigma^\sharp_i}; ~ p \i \s(p) \in P_{\Sigma^\sharp_i} \cdot_{\s} B w_i B\} . $$ 
	
	Recall that $\pi_i$ denotes the projection $P_i \to L_i \to L_i ^{\natural} \to G_i$.
	Note that $$P_i/P_{\Sigma^\sharp_i}  \cong L_i/(L_i \cap P_{\Sigma^\sharp_i}) \cong L_i^{\natural}/( \pi _i (B)\times H_i)  \cong G_i / \pi_i (B),$$ where $G_i/\pi _i (B)$ is the full flag variety of $G_i$. Under this isomorphism, the sub-variety $X'$ of $P_i/P_{\Sigma^\sharp_i}$ is identified to $X^{\mathbb G_i}_{w_i}$. The proposition is proved. 
\end{proof}

\begin{corollary}\label{cor:indcution of trace}For each $1\leq i\leq i_{\max}$, we have an isomorphism of virtual $G(\BF_q)$-representations $$\coh^*_c(X_{J, w_i}) \cong \Ind^{G(\BF_q)}_{P_i(\BF_q)} \coh ^*_c(X^{\mathbb G_i}_{w_i}),$$ where $P_i(\BF_q)$ acts on $X^{\mathbb G_i}_{w_i}$ via the projection $\pi_i : P_i(\BF_q) \to G_i (\BF_q)$.
\end{corollary}
\begin{proof}
	This follows immediately from Proposition \ref{prop:inductive}.
\end{proof}
\subsection{Review of regular elements} We recall the definition of regular elements and some standard facts. Let $G$ be a reductive group over $k$. \begin{definition}
	An element $g\in G$ is called \textit{regular}, if the centralizer $G_g$ of $g$ in $G$ has dimension equal to the rank of $G$. The set of regular elements is denoted by $G ^{\mathrm
		{reg}}$.
\end{definition}
If $G$ is semi-simple, the above definition is the same as \cite{Steinberg-reg}. In general, one easily checks that $g\in G$ is regular in the above sense if and only if the image of $g$ in $G^{\ad}$ is regular. Thus we can easily transport the results from \cite{Steinberg-reg}, which only discusses semi-simple groups, to reductive groups. 
\begin{theorem}\label{thm:steiberg crit}
	An element $g\in G$ is regular if and only if there are only finitely many Borel subgroups of $G$ that contain $g$.  
\end{theorem}
\begin{proof} This follows from \cite[Theorem 1.1]{Steinberg-reg} applied to $G^{\ad}$.  
\end{proof}

\begin{proposition}\label{prop:implication}
	Assume $G'$ is a reductive group over $k$ that contains $G$ as a closed subgroup. Then ${G'} ^{\mathrm{reg}} \cap G \subset G ^{\mathrm{reg}}$.
\end{proposition}
\begin{proof}
	Fix a Borel subgroup $B' \subset G'$ that contains $B$. By Theorem \ref{thm:steiberg crit}, it suffices to show that the natural map between flag varieties $G/B \to G'/ B'$ is finite-to-one (at the level of $k$-points). For this, it suffices to show that $B$ is of finite index in $B' \cap G$. Note that the identity component $(B'\cap G)^0$ of $B'\cap G$ is a connected solvable closed subgroup of $G$ which contains $B$. Hence $(B'\cap G)^0 = B$. But we know that $(B'\cap G)^0 $ has finite index in $B'\cap G$ because the latter is a linear algebraic group over $k$. 
\end{proof}
\begin{proposition}\label{prop:projection of reg}
	Let $P=P_J$ be a standard parabolic subgroup of $G$, with standard Levi subgroup $L= L_J$. The projection $P \to L$ maps $P\cap G^{\mathrm{reg}}$ into $L^{\mathrm{reg}}$.
\end{proposition}
\begin{proof}
	The projection $P\to L$ induces a bijection from the set of Borel subgroups of $G$ contained in $P$ to the set of Borel subgroups of $L$. Thus the proposition follows from Theorem \ref{thm:steiberg crit}.
\end{proof}

The following proposition is well known and elementary to verify.
\begin{proposition}\label{prop:GL reg}
	Let $V$ be a finite dimensional $k$-vector space. An element $g \in \GL(V)$ is regular if and only if each eigenspace of $g$ is one dimensional. \qed 
\end{proposition}

\subsection{The character formula on a classical Deligne--Lusztig variety}
Let $g \in G(\BF_q)$ and let $g=su$ be the Jordan decomposition of $g$. Assume $g$ is regular in $G$. Let $w \in W$. Let $(T_w,B_w)$ be the pair associated to $w$ as in \cite[Lemma 1.13]{DL}. Namely, $T_w$ is a $\sigma$-stable maximal torus of $G$, and $B_w$ is a Borel subgroup of $G$ containing $T_w$ such that $B_w$ and $\sigma (B_w)$ have relative position $w$. The pair $(T_w,B_w)$ is well defined up to $G(\BF_q)$-conjugation, but we fix a representative. 
We denote by ${}^G s$ the conjugacy class in $G(k)$ of $s$, and denote by ${}^{G(\BF_q)} s$ the conjugacy class in $G(\BF_q)$ of $s$.
\begin{proposition}\label{prop:rough trace}
	In the above setting, we have 
	\begin{align}\label{eq:rough trace}
	\tr(g \mid  \coh _c^*(X_w)) =  \frac{|G_s (\BF_q)| }{| G_s^0 (\BF_q)| } \cdot \big | T_w \cap {}^{G(\BF_q)} s \big |.
	\end{align}
\end{proposition}
\begin{proof}
	By \cite[Theorem 4.2]{DL}, we have $$
	\tr(g \mid \coh _c^*(X_w))=\frac{1}{|G_s^0 (\BF_q)|} \sum_{g' \in G(\BF_q); g' T_w(g') \i \subset G_s^0} Q_{g' T_w (g') \i, G_s^0}(u),$$
	where $Q_{g' T_w (g') \i, G_s^0}$ is the Green function. Since $g $ is regular in $G$, we know that $u$ is regular in $G_s^0$. Hence by \cite[Theorem 9.16]{DL}, we have $Q_{g' T_w (g') \i, G_s^0}(u)=1$ for every $g'$ that appears in the above summation. Therefore we have
	$$\tr(g \mid  \coh ^*_c (X_w))=   \frac{1}{| G_s^0 (\BF_q)| }\# \{ g ' \in G(\BF_q) ; g'T_w (g') \i \subset G_s^0 \} . $$
	Now for $g' \in G(\BF_q)$, the condition $g' T_w (g')\i \subset G_s ^0$ is equivalent to the condition $s \in g' T_w (g')\i$, which is equivalent to the condition $(g')\i s g' \in T_w \cap {}^{G(\BF_q)} s$. Therefore we have 
	$$\# \{ g ' \in G(\BF_q) ; g'T_w (g') \i \subset G_s^0 \} = |G_s (\BF_q)| \cdot \big | T_w \cap {}^{G(\BF_q)} s \big |  $$
	by the orbit-stabilizer relation. The proposition follows. 
\end{proof} 
\begin{definition}\label{defn:TT}
	For each $\gamma \in T_w (\BF_q)$, define $$ \TT (w,\gamma) : =  \frac{|G_{\gamma} (\BF_q)| }{| G_{\gamma}^0 (\BF_q)| } \cdot \big| T_w \cap {}^{G(\BF_q)} \gamma \big | .$$ Since $T_w$ is well defined up to $G(\BF_q)$-conjugation, the above definition indeed only depends on $w$ and $\gamma$.
\end{definition}
\begin{corollary}\label{cor:summary of classical}
	Let $g \in G(\BF_q) \cap G^{\mathrm{reg}}$ and $w \in W$. Let $g = su$ be the Jordan decomposition. We have 
	$$ \tr (g \mid \coh^*_c(X_w)) = \begin{cases}
	0 , & \text{if }T_w \cap {}^{G(\BF_q)} s  = \emptyset ,\\
	\TT (w,\gamma), &\text{if }T_w \cap {}^{G(\BF_q)} s  \neq \emptyset.
	\end{cases} $$
	In the second case, $\gamma$ is any element of $T_w \cap {}^{G(\BF_q)} s$.  \end{corollary}
\begin{proof}
	This follows from Proposition \ref{prop:rough trace}, by noting that the right hand side of (\ref{eq:rough trace}) only depends on the $G(\BF_q)$-conjugacy class of $s$. 
\end{proof}
\subsubsection{}\label{subsubsec:conn setting} Let $w\in W$ and $\gamma \in T_w (\BF_q)$. We will give a more explicit formula for $\TT(w,\gamma)$, under the assumption that $G_{\gamma}$ is connected. For example, if $G^{\mathrm{der}}$ is simply connected, then our assumption is always satisfied, by a result of Steinberg \cite[Corollary 8.5]{Steinberg-End} (cf.~\cite[p.~788]{kottwitzrational} or \cite[Theorem 3.5.6]{Carter}).

Assume $G_{\gamma}$ is connected. We canonically identify $W$ with  $N_G(T_w)/T_w$ via the pair $(T_w, B_w)$ fixed before. Then the Weyl group of $G_{\gamma}$ is a canonical subgroup $W(\gamma)$ of $W$, generated by the reflections associated to roots $\alpha$ in $\Phi (T_w, G)$ such that $\alpha(\gamma) =1$ (see \cite[Theorem 3.5.4]{Carter}). Denote by $F_w$ the automorphism $\Ad (w) \circ \sigma$ of $W$. Then $W(\gamma)$ is stable under $F_w$, as $\gamma$ is an $\BF_q$-point of $T_w$. 
\begin{lemma}\label{lem:compute TT}
	In the setting of \S \ref{subsubsec:conn setting}, we have $$\TT(w,\gamma) = \# \{ {}^x \gamma ;  x \in W, {}^x \gamma \in G(\BF_q)\} = \# (W/W(\gamma)) ^{F_w} . $$
\end{lemma}
\begin{proof}
	Since $G_{\gamma}$ is connected, it follows from the Lang--Steinberg theorem that $H^1 (\BF_q, G_{\gamma}) =0 $, and so ${}^{G(\BF_q)} \gamma =  {}^G\gamma  \cap G(\BF_q)$. Therefore 
	$$ \TT (w,\gamma) = \abs{T_w(\BF_q) \cap {}^G \gamma }. $$ 
	Now assume $h\in G$ satisfies $h \gamma h\i \in T_w$. Then $h\i T_w h \subset G_{\gamma}$. Sine $h\i T_w h$ and $T_w$ are two maximal tori of $G_{\gamma}$, there exists $c \in G_{\gamma}$ such that $h\i T_w h = c T_w c\i$. Then we have 
	$$h \gamma h\i = (hc) \gamma (hc)\i, \quad hc \in N_G(T_w).$$ 
	The above analysis shows that, 
	$$ \abs{T_w(\BF_q) \cap {}^G \gamma } = \# \{ {}^x \gamma ;  x \in W, {}^x \gamma \in G(\BF_q)\} .$$ This proves the first equality in the lemma. To prove the second equality, note that $$  \# \{ {}^x \gamma ;  x \in W, {}^x \gamma \in G(\BF_q)\} = \# (W/W_{\gamma}) ^{F_w},$$ where $W_\gamma$ is the stabilizer of $\gamma$ in $W$. Since $G_{\gamma}$ is connected, we have $W_{\gamma} = W(\gamma)$, see \cite[Theorem 3.5.3]{Carter}. \end{proof}  
\subsection{Combining the results}\label{subsec:combining results}
Keep the setting of \S \ref{subsec:setting}. For each $1\leq i \leq i _{\max}$, fix a $\sigma$-stable maximal torus $T_i$ in $G_i$ of type $w_i$.
Fix $\Gamma_i \subset T_i (\BF_q)$ to be a complete set of representatives of elements in $T_i (\BF_q)$ modulo $G_i(\BF_q)$-conjugacy. Fix $g \in G(\BF_q)$. For each $1\leq i \leq i_{\max}$ and each $\gamma \in \Gamma_i$, define $$ \widetilde {\mathcal M}_i ^g: = \set{r \in G (\BF_q); r^{-1} g r \in P_i (\BF_q)} , $$$$
\widetilde{\mathcal M}_i ^{g, \gamma}: = \{ r\in \widetilde{\mathcal M}_i ^g ;  (\pi_i(r^{-1}gr))_s \in {}^{G_i(\BF_q)}\gamma \} .$$ 
Here $(\pi_i (r^{-1} g r)) _s$ denotes the semi-simple part of $\pi_i (r^{-1} g r) \in G_i (\BF_q)$ in the Jordan decomposition. Note that $\widetilde{\mathcal M}_i ^g$ and $\widetilde{\mathcal M}_i ^{g, \gamma}$, if non-empty, are stable under right multiplication by $P_i(\BF_q)$. We denote $$\mathcal M_i ^g : = \widetilde{\mathcal M}_i^g / P_i (\BF_q) , \qquad \mathcal M_i ^{g,\gamma} : = \widetilde{\mathcal M}_i^{g,\gamma} / P_i (\BF_q).$$ For $\gamma \in \Gamma_i \subset T_i(\BF_q)$, we also define $\TT (w_i, \gamma)$ as in Definition \ref{defn:TT}, with respect to $\mathbb G_i$ and $w_i \in W_i$.
\begin{theorem}\label{thm:char formula} Fix $g \in G(\BF_q) \cap G^{\mathrm{reg}}$. Then
	$$\tr (g , J, \mathscr L) = \sum _{i=1} ^{i_{\max}}\sum _{\gamma \in \Gamma_i} \# \mathcal M_i ^{g, \gamma} \cdot \TT (w_i, \gamma).   $$  
\end{theorem}
\begin{proof}
	By Corollary \ref{cor:decomp of trace} and Corollary \ref{cor:indcution of trace}, we have
	\begin{align}\label{eq:step 1}
	\tr(g, J ,\mathscr L) = \sum _{i=1} ^{i_{\max}} \abs{P_i (\BF_q)} ^{-1}\sum_{r \in \widetilde{\mathcal M}_i ^{g} } \tr \bigg(\pi _i ( r \i g r) \mid \coh ^*_c (X^{\mathbb G_i} _{w_i}) \bigg) .
	\end{align} 
	Fix $1\leq i \leq i_{\max}$. For any $r \in \widetilde{\mathcal M}_i ^{g}$, it follows from Proposition \ref{prop:projection of reg} that the image of $r\i g r$ under $P_i \to L_i$ is regular in $L_i$. It easily follows that $\pi_i (r \i g r)$ is regular in $G_i$. We may hence apply Corollary \ref{cor:summary of classical} to get 
	\begin{align}\label{eq:step 2}
	\sum_{r \in \widetilde{\mathcal M}_i ^{g} } \tr \bigg(\pi _i ( r \i g r) \mid \coh ^*_c (X^{\mathbb G_i} _{w_i}) \bigg)  =  \sum _{\gamma \in \Gamma_i } \sum_{r \in \widetilde{\mathcal M}_i ^{g, \gamma} } \TT(w_i,\gamma) =\sum_{\gamma \in \Gamma _i} \# \widetilde{\mathcal M}_i ^{g, \gamma}  \cdot \TT(w_i,\gamma) .
	\end{align}
	Combining (\ref{eq:step 1}) and (\ref{eq:step 2}), we obtain 
	\[\pushQED{\qed} \tr (g, J,\mathscr L) = \sum _{i=1} ^{i_{\max}} \sum _{\gamma \in \Gamma _i} \abs{P_i (\BF_q)}^{-1} \# \widetilde{\mathcal M}_i ^{g, \gamma}  \cdot \TT(w_i,\gamma) = \sum _{i=1} ^{i_{\max}}\sum _{\gamma \in \Gamma_i} \# \mathcal M_i ^{g, \gamma} \cdot \TT (w_i, \gamma).  \qedhere  \]
\end{proof}

\section{Basic loci of Shimura varieties of Coxeter type}
\label{sec:basic loci}
The notion of basic loci of Coxeter type in Shimura varieties is introduced in \cite{GH}. The basic loci in these cases can be decomposed into a finite union of Ekedahl--Oort strata indexed by the set $\text{EO}^K_{\s, \text{cox}}$ defined in \cite[\S 5.1]{GH}, and each Ekedahl--Oort stratum is a union of classical Deligne--Lusztig varieties of Coxeter type. We have the following classification theorem.

\begin{table}[h]
	\centering\caption{$\sigma$-unbranched data}\label{tab:coxeter}
	\begin{tabular}{|c|c l |}
		\hline
		\text{Enhanced Tits datum} & \text{$\s$-unbranched datum $(\mathbb G, J, \mathscr L = \sett{\rr_1,\rr_2, \cdots, \rr_a})$} &
		\\
		\hline
		$(A_n, \o^\vee_1, \BS)$ & $(\text{trivial group}, \emptyset, \emptyset)$ &
		\\
		\hline
		$(B_n, \o^\vee_1, \BS)$ & $({}^2 D_n, \BS-\{s_{n-1} \}, \sett{s_1,\cdots, s_{n-1}})$ & \\
		\hline
		$(B_n, \o^\vee_1, \tilde \BS-\{n\})$ & $(B_{n-1}, \BS-\{s_{n-1}\}, \sett{s_1,\cdots, s_{n-1}} )$ &
		\\
		\hline
		$(B$-$C_n, \o^\vee_1, \BS)$ & $({}^2 D_n, \BS-\{s_{n-1} \}, \sett{s_1,\cdots, s_{n-1}} )$ &
		\\
		\hline
		$(B$-$C_n, \o^\vee_1, \tilde \BS-\{n\})$ & $(B_{n-1}, \BS-\{s_{n-1}\}, \sett{s_1,\cdots, s_{n-1}} )$ &
		\\
		\hline
		$(C$-$B_n, \o^\vee_1, \BS)$ & $(B_{n}, \BS-\{s_n\}, \sett{s_1,\cdots, s_{n}}  )$
		&\\
		\hline
		$(C$-$BC_n, \o^\vee_1, \BS)$ & $(B_{n}, \BS-\{s_n\} ,  \sett{s_1,\cdots, s_{n}} )$
		&\\
		\hline
		$(C$-$BC_n, \o^\vee_1, \tilde \BS-\{n\})$ & $(C_{n}, \BS-\{s_n\},  \sett{s_1,\cdots, s_{n}})$
		&\\
		\hline
		$(D_n, \o^\vee_1, \BS)$ & $({}^2 D_{n-1}, \BS-\{s_{n-2}\},  \sett{s_1,\cdots, s_{n-2}})$
		&\\
		\hline
		$({}^2 A'_n, \o^\vee_1, \BS)$  &  $({}^2 A_{2 m}, \BS-\{s_m\}, \sett{s_1,\cdots, s_m}), ~ m : = \lfloor{\frac{n-1}{2}}\rfloor$
		&\\
		\hline
		$({}^2 B_n, \o^\vee_1, \tilde \BS-\{n\})$ & $(B_n, \BS-\{s_n\}, \sett{s_1,\cdots, s_{n}})$
		&\\
		\hline
		$({}^2 B$-$C_n, \o^\vee_1, \tilde \BS-\{n\})$ & $(C_n, \BS-\{s_n\}, \sett{s_1,\cdots, s_{n}} )$
		&\\
		\hline
		$({}^2 D_n, \o^\vee_1, \BS)$ & $({}^2 D_n, \BS-\{s_{n-1}\}, \sett{s_1,\cdots, s_{n-1}})$    
		&\\
		\hline
		$(A_3, \o^\vee_2, \BS)$ & $({}^2 (A_1 \times A_1), \{s_1\}, \sett{s_2})$
		& * \\
		\hline
		$({}^2 A'_3, \o^\vee_2, \BS)$ & $({}^2 A_3, \{s_2, s_3\}, \sett{s_2, s_1}) $
		& * \\
		\hline
		$(C_2, \o^\vee_2, \BS)$  & $({}^2 (A_1 \times A_1), \{s_1\}, \sett{s_2}) $
		& * \\
		\hline
		$(C_2, \o^\vee_2, \tilde \BS-\{1\})$ & $(A_1, \emptyset, (s_1))$
		&\\
		\hline
		$({}^2 C_2, \o^\vee_2, \tilde \BS-\{1\})$ & $(B_2, \{s_1\}, (s_1,s_2))$
		&\\
		\hline
		$({}^2 C$-$B_2, \o^\vee_1, \tilde \BS-\{1\})$ & $(B_2, \{s_2\}, (s_2,s_1))$
		& * \\
		\hline
	\end{tabular}
\end{table}

\begin{theorem}\cite[Theorem A]{GH}\label{classification}
	The irreducible enhanced Tits data of Coxeter type for $\s$-stable maximal $K$ are classified in the first column of Table~\ref{tab:coxeter}. 
\end{theorem}

We list in the second column of Table~\ref{tab:coxeter} the associated $\s$-unbranched data. In each case, let $w$ be the maximal element in $\text{EO}^K_{\s, \text{cox}}$ computed in \cite[\S 6]{GH}. Then the reductive group $\mathbb G$ over $\BF_q$ is the reductive quotient of the parahoric subgroup associated to $\supp_\s(w)$, and we have $J=K \cap \supp_\s(w)$. In each case it turns out that $J$ is $\s$-unbranched, and that there is a unique $\s$-unbranched datum of the form $(J,\mathscr L)$. In table Table~\ref{tab:coxeter} we record the type of $\mathbb G$, the set $J$, and the nodes $(\rr_1,\cdots, \rr_a)$ of the unique $\mathscr L$ in the order as in (\ref{eq:order of nodes}). We let $s_i \in \BS$ denote the $i$-th node, according to Bourbaki's numbering \cite{bourbaki}. In all except the four cases marked with $*$, we have $\rr_i = s_i$ for all $1\leq i \leq a$.

Consequently, the associated fine Deligne--Lusztig varieties come in four infinite families: \begin{enumerate}
	\item \label{item:even ortho} $\mathbb G$ is the non-split even special orthogonal group $\SO_{2n}$, $J=\BS-\{s_{n-1}\}$, $\mathscr L = (s_1, \cdots, s_{n-1})$.
	\item \label{item:odd ortho} $\mathbb G$ is the odd special orthogonal group $\SO_{2n+1}$, $J=\BS-\{s_n\}$, $\mathscr L = (s_1, \cdots, s_{n})$.
	\item \label{item:symp} $\mathbb G$ is the symplectic group $\Sp_{2n}$, $J=\BS-\{s_n\}$, $\mathscr L = (s_1, \cdots, s_n)$.
	\item \label{item:unitary} $\mathbb G$ is the odd unitary group $\UU_{2n+1}$, $J=\BS-\{s_n\}$, $\mathscr L = (s_1, \cdots, s_n)$.  
\end{enumerate}

\section{Explicit character formulas}\label{sec:some-char-form}

In this section, we use Theorem \ref{thm:char formula} to compute $\tr (g, J, \mathscr L)$ for the four infinite families specified at the end of \S \ref{sec:basic loci}. We shall only consider $g\in G(\BF_q)$ whose image in $\GL_N$ under the standard representation is regular. This is a stronger hypothesis than requiring $g$ to be regular in $G$, except for the unitary case. However, for the known arithmetic applications this is enough (see \S \ref{sec:application}). We first need some preparations in \S \ref{subsec:rec} and \S \ref{subsec:GL reg}. 

\subsection{Reciprocal of polynomials}\label{subsec:rec}
We shall work with the base field $\BF_q$, but we shall consider polynomials $f(\lambda)$ in $\BF_q [\lambda]$ or $\BF_{q^2} [\lambda]$. These will appear as characteristic polynomials of elements in orthogonal or symplectic groups over $\BF_q$, or unitary groups of $\BF_{q^2}/\BF_q$-Hermitian spaces. Recall that $\s$ is the Frobenius automorphism of $k  = \overline{ \BF}_q$ over $\BF_q$. For $x\in k$, we write $x^{\s}$ for the image of $x$ under $\s$, i.e., $x^{\s} : = x^q$.

\begin{definition}
	For a polynomial $f\in \BF_{q^2}[\lambda]$ with $f(0) \neq 0$, we define its \emph{reciprocal polynomial} as $$f^*(\lambda):=(f(0)^{\s})^{-1}\cdot \lambda^{\deg f}\cdot f(1/\lambda)^\s \in \BF_{q^2}[\lambda].$$ We call $f\in \BF_{q^2}[\lambda]$ \emph{self-reciprocal}, if $f(0) \neq 0$ and $f=f^*$. (In particular, self-reciprocal polynomials are monic.) These definitions restrict to polynomials in $\BF_q [\lambda]$.
\end{definition}
\begin{remark}
	If $f(\lambda) \in \BF_{q^2} [\lambda]$ is monic and has factorization $f(\lambda) = \prod _j (\lambda -\lambda _j)$ with each $\lambda_j \in k^{\times}$, we have $f^*(\lambda) = \prod _j (\lambda - (\lambda_j ^{\s})^{\i})$. If in addition $f(\lambda) \in \BF_q [\lambda]$, then we also have $f^*(\lambda) = \prod _j (\lambda  - \lambda_j ^{-1}). $
\end{remark}
\begin{definition}
	We denote by $\mathsf {Irr}^{\times}$ the set of monic irreducible polynomials in $\BF_q [\lambda]$ with non-zero constant terms. We let $\mathsf {SR} \subset \mathsf {Irr}^{\times}$ be the subset of self-reciprocal irreducible polynomials, and let $\mathsf{NSR}: = (\mathsf {Irr}^{\times}  - \mathsf {SR} ) /*$ be the set of unordered pairs $\set{Q, Q^*}$ of monic irreducible polynomials reciprocal to each other with non-zero constant terms. Similarly, we denote by $\mathsf {Irr}_2^{\times}$ the set of monic irreducible polynomials in $\BF_{q^2} [\lambda]$ with non-zero constant terms. We let $\mathsf {SR}_2 \subset \mathsf {Irr}_2^{\times}$ be the subset of self-reciprocal irreducible polynomials, and let $\mathsf{NSR}_2: = (\mathsf {Irr}_2^{\times}  - \mathsf {SR}_2 ) /*$.   
\end{definition}	\begin{lemma}\label{lem:shape of factorization}
	If $f \in \BF_q[\lambda]$ is self-reciprocal, then its irreducible factorization is of the form 
	\begin{align}\label{eq:factorization}
	f = \prod _{Q \in \mathsf{SR}} Q ^{m_Q(f)} \prod _{\set{Q,Q^*} \in \mathsf {NSR} } (QQ^*) ^{m_{\set{Q, Q^*}}(f)}, 
	\end{align}
	for unique non-negative integers $m_Q(f), m_{\set{Q,Q^*}}(f)$.
	Similarly, if $f \in \BF_{q^2}[\lambda] $ is self-reciprocal, then we have 
	\begin{align}\label{eq:factorization 2}
	f = \prod _{Q \in \mathsf{SR}_2} Q ^{m_Q(f)} \prod _{\set{Q,Q^*} \in \mathsf {NSR}_2 } (QQ^*) ^{m_{\set{Q, Q^*}}(f)},
	\end{align}
	for unique non-negative integers $m_Q(f), m_{\set{Q,Q^*}}(f)$.
\end{lemma}
\begin{proof}This easily follows from unique factorization in $\BF_q[\lambda]$ and $\BF_{q^2}[\lambda]$.
\end{proof}
\begin{definition}\label{defn:MMM}
	Let $f \in \BF_q[\lambda]$ be self-reciprocal. Define $m_Q(f), m_{\set{Q,Q^*}}(f)$ as in (\ref{eq:factorization}). Define $$ \MMM(f) : = \prod _{ \set{Q, Q^*} \in \mathsf{NSR} } (1 + m _{\set{Q,Q^*}}(f)).$$
	Similarly, let $f \in \BF_{q^2}[\lambda]$ be self-reciprocal. Define $m_Q(f), m_{\set{Q,Q^*}}(f)$ as in (\ref{eq:factorization 2}). Define $$ \MMM_2(f) : = \prod _{ \set{Q, Q^*} \in \mathsf{NSR}_2} (1 + m _{\set{Q,Q^*}}(f)).$$
\end{definition}
\begin{lemma}\label{lem:MMM}
	Let $f \in \BF_q [\lambda]$ be self-reciprocal. Assume there is a unique element $Q_0 \in \mathsf{SR}$ such that $m_{Q_0}(f)$ is odd. Let $m$ be an odd integer such that $1\leq m \leq m_{Q_0}(f)$. Then 
	$$\# \set{U \in \BF_q [\lambda]^{\mathrm{monic}} ; UU^* = f/Q_0 ^{m}} = \MMM(f). $$ Similarly, let $f \in \BF_{q^2} [\lambda]$ be self-reciprocal. Assume there is a unique element $Q_0 \in \mathsf{SR}_2$ such that $m_{Q_0}(f)$ is odd. Let $m$ be an odd integer such that $1\leq m \leq m_{Q_0}(f)$. Then 
	$$\# \set{U \in \BF_{q^2} [\lambda]^{\mathrm{monic}} ; UU^* = f/Q_0 ^{m}} = \MMM_2(f). $$ 
\end{lemma}
\begin{proof}We only prove the statement about $\MMM(f)$, the other statement being similar. Write $h: = f/Q_0^m$. For any $Q \in \mathsf{SR}$, $m_Q(h)$ is even. For any $\set{Q, Q^*} \in \mathsf{NSR}$, $m_{\set{Q, Q^*}} (h ) = m _{\set{Q,Q^*}} (f)$. Now any $U\in \BF_q [\lambda] ^{\mathrm{monic}}$ with $UU^* = h$ is given by $$U = \prod _{Q \in \mathsf {SR}} Q ^{\frac{m_Q(h)}{2}} \prod _{\set{Q, Q^*} \in \mathsf {NSR} } U_{\set{Q, Q^*}},$$ where each $U_{\set{Q, Q^*}} = Q^i (Q^*)^j$, for any of the $1+ m_{\set{Q,Q^*}}(h)$ possible choices of pairs of non-negative integers $(i,j)$ satisfying $i+j = m _{\set{Q,Q^*}} (h)$.
\end{proof}
\begin{definition}\label{defn:even adm}
	Let $f \in \mathsf {SR}$ of even degree $d$. By an \emph{admissible enumeration} of the roots of $f$, we mean an enumeration of the $d$ distinct roots of $f$ in $k^{\times}$ of the form
	$ \lambda_1 ,\cdots, \lambda_{{\frac{d}{2}}}, \lambda_1 ^{-1},\cdots, \lambda_{{\frac{d}{2}}}^{-1}$ such that 
	$$\lambda_1 ^{\sigma} = \lambda_2, \lambda_2^{\s} = \lambda _3, \cdots, \lambda_{{\frac{d}{2}}-1} ^{\sigma} = \lambda_{{\frac{d}{2}}}, \lambda_{{\frac{d}{2}}}^\s = \lambda_1^{-1}.$$
\end{definition}

\begin{lemma}\label{lem:even deg}
	Let $f \in \mathsf {SR}$ of degree $d$. Then either $d$ is even or $f(\lambda) = \lambda  \pm 1$. When $d$ is even, there are precisely $d $ distinct admissible enumerations of the roots of $f$, all obtained from a given one by powers of a cyclic permutation of order $d$.
\end{lemma}
\begin{proof}
	The map $x \mapsto x^{-1}$ induces an involution on the set of all $d$ distinct roots of $f $. If $d$ is odd, this involution has a fixed point, which means $1$ or $-1$ is a root of $f$. 
	Hence $f = \lambda \pm 1$.
	
	We assume $d$ is even. We
	first prove the existence of one admissible enumeration. The $d$ distinct roots of $f$ are of the form
	$\lambda_1, \cdots, \lambda_{d/2}, \lambda_1 ^{-1},\cdots, \lambda_{d/2}^{-1}.$ Since they form precisely one $\sigma$-orbit, we may reorder the $\lambda_i$'s or switch the roles of $\lambda_i$ and $\lambda_i ^{-1}$, to arrange that 
	$
	\lambda_2 = \lambda_1 ^{\sigma}, \cdots, \lambda _{d/2} = \lambda_{d/2-1} ^{\sigma}.
	$
	We claim that we must then have $\lambda_{d/2} ^{\s} = \lambda_1 ^{-1}$. In fact, since the $d$ distinct roots form precisely one $\sigma$-orbit, we have $\lambda_{d/2} ^{\s} = \lambda_j ^{-1}$ for a unique $1 \leq j \leq d/2$. If $j \geq 2$, then 
	$$\lambda _{{\frac{d}{2}}}, \lambda_j^{-1}, \lambda_{j+1} ^{-1}, \cdots, \lambda_{{\frac{d}{2}}} ^{-1} , \lambda_j,\lambda_{j+1} ,\cdots, \lambda_{{\frac{d}{2}}-1}$$ already form one $\s$-orbit, which does not contain $\lambda_1$, a contradiction. Thus we have shown the existence of an admissible enumeration. The rest of the lemma is clear.
\end{proof}

\begin{definition}\label{defn:adm tuple} Let $d \geq 2$ be an even integer. Given a tuple 
	$\Lambda = (\lambda_1,\cdots, \lambda_{{\frac{d}{2}}}) \in (k^{\times})^{\oplus {\frac{d}{2}}},$ we define $$\Lambda^{-1}: = (\lambda_1\i,\cdots, \lambda_{{\frac{d}{2}}} \i),\quad \bar \Lambda := (\lambda_1,\cdots, \lambda_{{\frac{d}{2}}-1} ,\lambda_{{\frac{d}{2}}} \i),  \quad \Lambda[1] := (\lambda_{{\frac{d}{2}}}, \lambda_1,\cdots, \lambda_{{\frac{d}{2}}-1}).$$ By induction we also define $\Lambda[j]$ for all $j\in \ZZ$. Let $\Lambda$ be as above and let $f$ be an element of $\mathsf {SR}$ of degree $d$. We say that $\Lambda$ is \emph{admissible with respect to $f$}, if $(\Lambda, \Lambda ^{-1})$ is an admissible enumeration of the roots of $f$ in the sense of Definition \ref{defn:even adm}.
\end{definition}	

\begin{definition}\label{def:odd adm}
	Let $f \in \mathsf {SR}_2$ of odd degree $d$. By an \emph{admissible enumeration} of the roots of $f$, we mean an enumeration $\lambda_1 ,\cdots, \lambda_{d}$ of the $d$ distinct roots of $f$ such that $$
	\lambda_{1} ^{\s^2} = \lambda_{2}, \cdots, \lambda_{d-1}^{\s^2} = \lambda _d, \lambda _d^{\s^2} = \lambda_1.$$ 
\end{definition}
\begin{lemma}\label{lem:odd deg}
	Let $f \in \mathsf {SR}_2$ be of odd degree $d$. \begin{enumerate}
		\item There are precisely $d $ distinct admissible enumerations of the roots, all obtained from a given one by powers of a cyclic permutation of order $d$.
		\item Assume $d\geq 3$. Let $\lambda_1,\cdots, \lambda_d $ be an admissible enumeration of the roots of $f$. For any integer $j$ we define $\lambda_j$ to be $\lambda_{j'}$, for $1\leq j ' \leq d$ such that $j \equiv j' \mod d$. Then for all $j \in \ZZ $ we have 
		\begin{align}\label{eq:relation for all j}
		(\lambda_j\i)^{\s}=\lambda_{j+ \frac{d+1}{2}}.
		\end{align}
	\end{enumerate}
\end{lemma}

\begin{proof} Part (1) follows immediately from the fact that the $d$ distinct roots form precisely one $\sigma^2$-orbit. We prove part (2). Since for all $j$ we have $\lambda_j = \s ^{2(j-1)} (\lambda_1 )$, it suffices to prove (\ref{eq:relation for all j}) for $j=1$. Since the set of the roots is closed under the map $x\mapsto (x\i)^{\s}$, we have $(\lambda_1\i)^{\s} = \lambda_l$ for some $1\leq l \leq d$. We get $$ \lambda _1^{\s^2} = (((\lambda_1 \i) ^{\s} )\i ) ^{\s}= (\lambda_l \i)^\s = \s ^{2(l-1)}  [(\lambda_1 \i)^\s ]=  \s ^{2(l-1)}  (\lambda_l )= \lambda _{l+(l-1)}.$$ On the other hand $\lambda _1 ^{\s^2} = \lambda_2$, so $2l - 1 \equiv 2 \mod d .$ Since $1\leq l \leq d$ and $d \geq 3$ is odd, the only solution of this congruence is $ l = 	(d+3)/2,$ as desired. \end{proof}
\subsection{Eigenvalues $\pm 1$} \label{subsec:GL reg}
Fix a non-degenerate quadratic space $(V,[\cdot, \cdot])$ over $k$. We would like to control the multiplicities of the eigenvalues $\pm 1$, for elements $g \in \mathrm O (V) \cap \GL(V) ^{\mathrm{reg}}$. For $g \in \GL(V)$ and $\lambda \in k$, we write $V(g,\lambda)$ for the generalized eigenspace of $g$ belonging to $\lambda$, i.e., $V(g,\lambda) = \ker (g-\lambda)^{\dim V}$.
\begin{proposition}\label{prop:dealing with lambda-1} Let $g \in \mathrm{O}(V) \cap \GL(V) ^{\mathrm{reg}}$. Let $j =1$ or $-1$. Then $\dim V(g, j)$ is either zero or odd.
\end{proposition} 
\begin{proof} Firstly, it is easy to see that $V(g,j)$ is orthogonal to $V(g, \lambda)$ for any $\lambda \in k - \set{j}$. In particular, the quadratic form restricted to $V(g, j)$ is non-degenerate, and we obtain a quadratic space $(V(g,j), [\cdot,\cdot])$. By Proposition \ref{prop:GL reg}, $g|_{V(g,j)}$ is in $\GL(V(g,j)) ^{\mathrm{reg}}$. Thus we may and shall assume that $V= V(g, j)$.  
	
	Assume that $\dim V = \dim V(g,j) = 2n$, with $n \geq 1$, and we are to deduce a contradiction. Under this assumption we have $g\in \SO(V)$ (since $\det g = j^{2n} =1$). In particular $g$ lies in a Borel subgroup of $\SO(V)$, and so $g$ stabilizes a maximal totally isotropic subspace $M \subset V$. Let $N$ be a maximal totally isotropic subspace of $V$ such that $V= M \oplus N$. Since $g\in \GL(V) ^{\mathrm{reg}}$, the Jordan canonical form of $g|_M \in \GL(M)$ must be one Jordan block of eigenvalue $j$ (see Proposition \ref{prop:GL reg}). We thus find a $k$-basis $e_1,\cdots, e_n$ of $M$, such that $(g-j)$ sends each $e_{\alpha}$ to $e_{\alpha -1}$ (with $e_0 : = 0$). Let $f_1,\cdots, f_n$ be the basis of $N$ satisfying $[e_{\alpha}, f_\beta] = \delta _{\alpha,\beta}$. Using $g \in \SO(V)$ it is easy to see that $$gf_n = jf_n + \sum _{\alpha=1}^n \eta_{\alpha} e_{\alpha}$$ for some $\eta_{\alpha}\in k$. Then we have 
	$$0 = [ f_n, f_n] = [gf_n, gf_n] = 2j \eta _n. $$ Hence $\eta _n = 0 $. It follows that $(g-j)$ maps the $k$-span of $e_1,\cdots, e_n , f_n$ into the $k$-span of $e_1,\cdots, e_{n-1}$. Hence the nullity of $(g-j)$ is at least $2$, a contradiction (see Proposition \ref{prop:GL reg}).
\end{proof}
\subsection{The non-split even special orthogonal group}\label{subsec:even ortho} In this subsection we consider case (\ref{item:even ortho}) in \S \ref{sec:basic loci}. 

We fix a non-degenerate non-split $2n$-dimensional quadratic space $(\mathbb V,[\cdot, \cdot])$ over $\BF_q$, with $n\geq 1$ (the case $n=0$ being trivial). Let $\mathbb G = \SO (\mathbb V,[\cdot,\cdot])$. Let $V: = \mathbb V\otimes _{\BF_q} k$. By the classification of quadratic forms over $\BF_q$ (\cite[\S 1.3]{kitaoka}, also cf.~\cite[\S 15.3]{DigneMichel}) there exists a $k$-basis $\set{e_1,\cdots, e_n ,f_1,\cdots, f_n}$ of $V$, satisfying 
\begin{align*} &
[e_{\alpha} ,e_{\beta}] = [f_{\alpha}, f_{\beta}] = 0 ,\quad [e_{\alpha}, f_{\beta}] = \delta _{\alpha, \beta}, \quad \forall ~ 1\leq \alpha,\beta \leq n; \\ &
e_{\alpha} ^{\s} = e_{\alpha}, \quad f_{\alpha } ^{\s} = f_{\alpha}, \quad \forall ~ 1\leq \alpha \leq n-1; \\ & 
e_n ^{\s} = f_n, \quad f_n ^{\s} =e_n.
\end{align*}
For each $1\leq i \leq n$, we define $$V_i: =  \mathrm{span}_k(e_i , e_{i+1}, \cdots, e_{n}, f_i, f_{i+1},\cdots, f_n) \subset V,\qquad W_i : = \mathrm{span}_k ({e_1,\cdots, e_{i}}) \subset V .$$ For each $1\leq i \leq n-1$, we have $W_i= W_i^\s$, and we write $\mathbb W_i$ for the $\BF_q$-form of $W_i$. For $1\leq i \leq n$, we have $V_i = V_i ^{\s}$, and we write $\mathbb V_i$ for the $\BF_q$-form of $V_i$.

Let $G = \mathbb G_k$. Let $B \subset G$ be the common stabilizer of either of the following two flags in $V$:
$$W_1 \subset W_2\subset \cdots \subset W_{n-1}  \subset W_{n},$$
$$ W_1\subset W_2\subset \cdots \subset W_{n-1} \subset W_n ^{\s}.$$
Then $B$ is a $\s$-stable Borel subgroup of $G$.
Let $T$ be the intersection of $G$ with the diagonal torus in $\GL(V)$ under the basis $\set{e_1,\cdots, e_n, f_1,\cdots, f_n}$. Then $T$ is the maximal torus of $G$ contained in $B$. 

We number the simple roots of $(G,B,T)$ according to Bourbaki \cite{bourbaki}. We consider the $\s$-unbranched datum $(J=\BS-\{s_{n-1}\}, \mathscr L = (s_1, \cdots, s_{n-1}) ) $. Following the notation of \S \ref{subsec:setting} and \S \ref{subsec:parabolic induction}, we have $i _{\max} = n$, and for $1\leq i \leq n$ we have
$$\mathbb P_i = \mathrm{Stab}_{\mathbb G} ( \mathbb W_{i-1}) , \qquad \mathbb L_i = \mathbb L_i^{\natural} = \GL(\mathbb W_{i-1}) \times \SO (\mathbb V_i), $$
$$ \mathbb G_i = \SO(\mathbb V_i) = \SO _{2(n+1-i)} ~ \text{(non-split)}, \qquad \mathbb H_i = \GL(\mathbb W_{i-1}) = \GL _{i-1}.$$
Here by convention $\mathbb W_0 =0$ and $\GL_0 = \{1\}$. As in \S \ref{subsec:parabolic induction}, we have natural projections $\pi_i : \mathbb P_i \to \mathbb G_i$ and $\pi_i': \mathbb P_i \to \mathbb H_i$.

For any $h\in G_i(k)$, we denote by $ f_h \in k [\lambda]$ the characteristic polynomial of $h$ acting on $V_i$, which has degree $2(n+1-i)$. Thus if $h \in G_i(\BF_q)$, then $f_h$ is self-reciprocal in $\BF_q [\lambda]$. Similarly, for any $h \in H_i(k)$, we denote by $f_h (\lambda) \in k[\lambda]$ the characteristic polynomial of $h$ acting on $W_i$, which has degree $i-1$.

We fix $1\leq i \leq n$. Write $n' $ for $n+1 -i$. Thus $G_i = \SO_{2n'}$, with $n' \geq 1$. 
Let $B_i'$ (resp.~$T_i'$) be the intersection of $G_i$ with the upper triangular subgroup (resp.~diagonal subgroup) of $\GL(V_i)$, under the $k$-basis $\set{e_i,\cdots, e_n, f_n,\cdots, f_i}$ of $V_i$. Then $B_i'$ is a $\s$-stable Borel subgroup of $G_i$, and $T_i'$ is a $\s$-stable maximal torus of $G_i$ contained in $B_i'$. Thus $T_i'$ is a $\s$-stable maximal torus of type $1\in W_i$. For any $(\lambda_1,\cdots, \lambda_{n'}) \in (k^{\times})^{\oplus n'}$, let $\gamma'(\lambda_1,\cdots, \lambda_{n'})$ be the diagonal matrix $\diag(\lambda_1,\cdots, \lambda_{n'}, \lambda_{n'}^{-1},\cdots, \lambda_1^{-1})$ in $\GL(V_i)$ under the same basis. Then $\gamma'$ is an isomorphism $\GG_{m,k}^{n'} \isom T_i'$ (defined over $k$). The Weyl group $W_i$ can be identified with $(\set{\pm 1}^{\times n'})' \rtimes S_{n'}$, where $(\set{\pm 1}^{\times n'})'$ denotes the kernel of \begin{align*}
\set{\pm 1}^{\times n'} & \To \set{\pm 1} \\ (u_{\alpha})_{\alpha} &\longmapsto \prod_{\alpha} u_{\alpha}.
\end{align*} For $1\leq \alpha \leq n'$, the non-trivial element in the $\alpha$-th copy of $\set{\pm 1}$ sends $\gamma'(\lambda_1,\cdots, \lambda_{n'})$ to $$\gamma'(\lambda_1,\cdots, \lambda_{\alpha-1}, \lambda_{\alpha}^{-1}, \lambda_{\alpha+1}, \cdots, \lambda_{n'}). $$ For $\rho \in S_{n'}$, we have $\rho (\gamma'(\lambda_1,\cdots, \lambda_{n'})) = \gamma'(\lambda_{\rho^{-1}(1)},\cdots, \lambda_{\rho^{-1}(n')})$. We easily compute that $w_i$ acts on $T_i'$ in the following way: 
$$w_i:  \gamma'(\lambda_1,\cdots, \lambda_{n'}) \longmapsto \gamma' (\lambda_{n'},\lambda_1,\cdots, \lambda_{n'-1}).$$ 
Also, $\s$ acts on $T_i'$ in the following way:
$$\s: \gamma'(\lambda_1,\cdots, \lambda_{n'}) \longmapsto \gamma'(\lambda_1^{\s},\cdots,\lambda_{n'-1}^{\s}, (\lambda_{n'}^{\s})^{-1}). $$
Remember that $T_i$ is by definition a $\s$-stable maximal torus of $G_i$ of type $w_i$. From the above discussion, we see that on $T_i$ we have coordinates 
$$ (k^{\times})^{\oplus n'} \isom T_i ,\qquad 	 
(\lambda_1,\cdots, \lambda_{n'}) \mapsto \gamma (\lambda_1,\cdots, \lambda_{n'}), $$  such that the eigenvalues (with multiplicities) of $\gamma (\lambda_1,\cdots, \lambda _{n'})$ acting on $V_i\cong k^{2n'}$ are $$\lambda_1,\cdots ,\lambda_{n'}, \lambda_1 ^{-1},\cdots, \lambda_{n'} ^{-1},$$
and such that 
\begin{align}\label{eq:Galois action}
\gamma(\lambda_1,\cdots, \lambda_{n'})^{\sigma} = \gamma ((\lambda_{n'}^{-1})^{\s}, \lambda_1 ^{\s}, \lambda_2 ^{\s}, \cdots, \lambda_{n'-1} ^{\s}).
\end{align}
Moreover, the action of $W_i \cong (\set{\pm 1}^{\times n'})' \rtimes S_{n'}$ on $T_i$ (which is no longer defined over $\BF_q$) is described in terms of these coordinates similarly as before: The non-trivial element in the $\alpha$-th copy of $\set{\pm 1}$ sends $\gamma(\lambda_1,\cdots, \lambda_{n'})$ to $\gamma(\lambda_1,\cdots, \lambda_{\alpha}^{-1},\cdots, \lambda_{n'})$. For $\rho \in S_{n'}$, we have $\rho (\gamma(\lambda_1,\cdots, \lambda_{n'})) = \gamma(\lambda_{\rho^{-1}(1)},\cdots, \lambda_{\rho^{-1}(n')})$.	

\begin{theorem}\label{thm:about Ti}
	We have the following statements about $T_i (\BF_q)$.
	\begin{enumerate}
		\item If $\gamma \in T_i(\BF_q)$, then  $f_{\gamma} = Q^m$ for some $Q \in \mathsf {SR}$, and some positive integer $m$. Moreover, either $Q(\lambda) = \lambda \pm 1$, or $m$ is odd. 
		\item Let $Q \in \mathsf{SR}$. Assume $ m$ is an odd integer such that $m \deg Q = 2n'$. (In particular $Q(\lambda) \neq \lambda \pm 1$). Then there exists $\gamma \in T_i (\BF_q)$ with $f_{\gamma}= Q ^m$. 
		\item Let $Q$ and $m$ be as in part (2). Let $\gamma \in G_i(k)$ be a semi-simple element such that $f_{\gamma} = Q^m$. Then $\gamma$ is $G_i(k)$-conjugate to an element of $T_i (\BF_q)$.
		\item For any $\gamma \in T_i(\BF_q)$, the centralizer $G_{i,\gamma}$ is connected. 
		\item Let $\gamma \in T_i (\BF_q)$. Write $f_{\gamma} = Q^m$ as in part (1). Assume $Q(\lambda) \neq \lambda \pm 1$. Then $\TT(w_i,\gamma) = (\deg Q)/2$. Here $\TT(w_i,\gamma)$ is defined in Definition \ref{defn:TT}.
	\end{enumerate}
\end{theorem}
\begin{proof} 
	
	\textbf{(1)} Write $\gamma = \gamma(\lambda_1,\cdots, \lambda_{n'})$. Since $\gamma^{\sigma} = \gamma$, it follows from (\ref{eq:Galois action}) that we have the following equality between two $2n'$-tuples in $k^{\times}$:
	\begin{align}\label{eq:condition for rational}
	(\lambda_1, \lambda_1 ^\s, \cdots, \lambda _1 ^{\s^{2n'-1}}) = (\lambda_1,\cdots, \lambda _{n'}, \lambda_1 ^{-1} ,\cdots, \lambda _{n'} ^{-1}).
	\end{align}
	We remark that (\ref{eq:condition for rational}) is valid even for $i = i_{\max}=n$. In fact, in that case $T_i = G_i$ is the kernel of the norm map $\Res_{\BF_{q^2} / \BF_q} \GG_m \to \GG_m$, and (\ref{eq:condition for rational}) reads $\lambda_1 ^{\sigma} = \lambda_1 ^{-1}$.
	
	Therefore all eigenvalues of $\gamma$ are in one $\s$-orbit. It follows that $f_{\gamma}$ has a unique monic irreducible factor $Q$. Since $f_{\gamma}$ is self-reciprocal, so is $Q$.
	
	Now assume $m$ is even. Then $d: =\deg Q$ divides $n'$. Since (\ref{eq:condition for rational}) holds and since there are precisely $d$ distinct eigenvalues of $\gamma$, we know that $\lambda_1$ is fixed by $\s^d$. Since $d$ divides $n'$, it follows that $\lambda_1$ is fixed by $\s^{n'}$. By (\ref{eq:condition for rational}) $\lambda _1^{\s ^{n'}} = \lambda _1 \i$. Hence $\lambda_1  = \lambda _1 \i$, and so $\lambda_ 1 = \pm 1$. It follows that $Q(\lambda) = \lambda \pm 1$. 
	
	\textbf{(2)} Let $d= \deg Q$. Then $d$ is even since $dm$ is even. 
	We fix a tuple $\Lambda \in (k^{\times})^{\oplus {\frac{d}{2}}}$ admissible with respect to $Q$, see Definition \ref{defn:adm tuple}. Then
	$$
	\gamma : = \gamma (\underbrace{\Lambda,\Lambda^{-1},\cdots, \Lambda, \Lambda^{-1}, \Lambda}_m)$$
	is an element of $T_i (\BF_{q})$ satisfying $f_{\gamma} = Q^m$.
	
	\textbf{(3)} Let $d = \deg Q$. We know $d$ is even. We assume without loss of generality that $\gamma \in T_i(k)$. Since $f_{\gamma} = Q ^m$, the $n'$ coordinates of $\gamma$
	must contain elements $\lambda_1,\cdots, \lambda_{{\frac{d}{2}}}$ such that all roots of $Q$ are given by $\lambda_1,\cdots, \lambda_{{\frac{d}{2}}}, \lambda_1 \i ,\cdots, \lambda_{{\frac{d}{2}}} \i.$ We temporarily assume $m >1$. By Lemma \ref{lem:even deg}, there exists an admissible tuple $\Lambda$ with respect to $Q(\lambda)$, obtained by permuting $\lambda_1,\cdots, \lambda_{d/2}$ and replacing some of them with their inverses. Up to replacing $\gamma$ by  $^x\gamma$ for some $x\in W_i$, we may arbitrarily permute the coordinates of $\gamma$, and we may replace an arbitrary even number of coordinates of $\gamma$ by their inverses. As $m > 1$, we may therefore arrange that either
	$$\gamma = \gamma (\underbrace{\Lambda, \Lambda^{-1} , \cdots, \Lambda ,\Lambda ^{-1}, \Lambda}_m) $$
	or
	$$\gamma = \gamma (\underbrace{\Lambda, \Lambda^{-1} , \cdots, \Lambda ,\Lambda ^{-1}}_{m-1}, \bar \Lambda).$$  In the first case we already have $\gamma \in T_i (\BF_q)$. Assume we are in the second case. 
	Since $m$ is odd, we may simultaneously replace each of the first $m-1$ appearances of $\Lambda$ or $\Lambda \i$ by its bar, i.e., $\gamma$ is $W_i$-conjugate to
	$$ \gamma (\bar \Lambda, \overline {\Lambda ^{-1}}, \cdots, \bar \Lambda, \overline {\Lambda ^{-1}}, \bar \Lambda) = \gamma (\bar \Lambda, \bar\Lambda \i ,\cdots, \bar \Lambda, \bar \Lambda \i, \bar \Lambda).$$ But the above element is $W_i$-conjugate to
	$$  \gamma (\bar \Lambda [1], \bar\Lambda \i [1],\cdots, \bar \Lambda [1], \bar \Lambda \i [1], \bar \Lambda [1]) = \gamma (\Omega, \Omega \i,\cdots, \Omega, \Omega \i,  \Omega), $$ where $\Omega : = \bar \Lambda [1]$. Note that $\Omega$ is admissible with respect to $Q$, and using this fact it is easy to check that the above element is in $T_i (\BF_q)$.
	
	Now we treat the case $m=1$. In this case $\gamma$ is $W_i$-conjugate to either $\gamma(\Lambda)$ or $\gamma (\bar \Lambda)$, for a tuple $\Lambda$ admissible with respect to $Q$. The element $\gamma (\Lambda)$ is already in $T_i (\BF_q)$. The element $\gamma (\bar \Lambda)$ is $W_i$-conjugate to $\gamma (\bar \Lambda [1])$, which is in $T_i(\BF_q)$ since $\bar \Lambda[1]$ is admissible with respect to $Q$.

	\textbf{(4)} We claim that any element $x\in W_i$ fixing $\gamma$ is a certain product of reflections associated to roots that send $\gamma$ to $1$. Once the claim is proved, it will follow that $G_{i,\gamma}$ is connected, see \cite[Theorem 3.5.3]{Carter}. We now prove the claim.
	
	For each $1\leq \alpha \leq n'$, we let $\epsilon_{\alpha} \in X^*(T_i)$ be the character on $T_i$ sending $\gamma(\lambda_1,\cdots, \lambda_{n'})$ to $\lambda_{\alpha}$. Then $\set{\epsilon_1,\cdots, \epsilon_{n'}}$ is a $\ZZ$-basis $X^*(T_i)$, and the roots in $X^*(T_i)$ are $\set{\pm \epsilon_{\alpha } \pm \epsilon_{\beta};\alpha \neq \beta}$. For each $x\in W_i$, define $$A(x): = \set{\alpha; 1\leq \alpha \leq n', x(\epsilon _{\alpha}) \notin \set{\pm \epsilon_{\alpha}}}. $$
	Now assume that $x$ fixes $\gamma$, and assume that $A(x) \neq \emptyset$. Take $\alpha \in A(x)$. Then $x(\epsilon _{\alpha}) = \pm \epsilon _{\beta}$ for some $\beta \neq \alpha$. If $x(\epsilon _{\alpha}) =  \epsilon _{\beta}$, then we left multiply $x$ by the reflection $\epsilon _{\alpha} \mapsto \epsilon _{\beta}, \epsilon _{\beta} \mapsto \epsilon _{\alpha}$. If $x(\epsilon _{\alpha}) =  -\epsilon _{\beta}$, then we left multiply $x$ by the reflection $\epsilon _{\alpha} \mapsto - \epsilon _{\beta}, \epsilon _{\beta} \mapsto - \epsilon _{\alpha}$. In either case, we have left multiplied $x$ by a reflection associated to a root (i.e.~$\epsilon_{\alpha} - \epsilon _{\beta}$ in the first case and $\epsilon_{\alpha} + \epsilon _{\beta}$ in the second case) which sends $\gamma$ to $1$, and the product is an element $y \in W_i$ which also fixes $\gamma$ and which satisfies $\#  A(y) < \# A(x)$.  
	In this way, we reduce to the case where $A(x) = \emptyset$. Now assume $A(x) = \emptyset$, and let $$B(x) = \set{\alpha ; 1 \leq \alpha \leq n' , x(\epsilon_{\alpha}) \neq \epsilon_{\alpha } }. $$ Then $x \in (\set{\pm 1} ^{\times n'})' \subset W_i$, and if we write $x = (x_1,\cdots,x_{n'}) \in \set{\pm 1}^{\times n'}$, then $B(x) = \set{\alpha; x_{\alpha} = -1}$. In particular, $\# B(x)$ is even. Since $x$ fixes $\gamma$, we know $\epsilon _{\alpha } (\gamma) = \pm 1$ for each $\alpha \in B(x)$. By part (1) we know that $\pm 1$ cannot simultaneously be eigenvalues of $\gamma$, so these $\epsilon_{\alpha }(\gamma)$ must all be $1$ or all be $-1$. Write $\# B(x)$ as $2l$, and enumerate the elements of $B(x)$ arbitrarily as $\set{\alpha_1,\cdots, \alpha_l ,\beta_1,\cdots, \beta_l}$. Then for each $1\leq j \leq l$, the roots $\epsilon_{\alpha_j} + \epsilon _{\beta_j}$ and $\epsilon_{\alpha_j} - \epsilon_{\beta_j}$ both send $\gamma$ to $1$. We easily see that $$x = \prod_{j =1}^l s_{\epsilon_{\alpha_j} + \epsilon _{\beta_j}} \cdot  s_{\epsilon_{\alpha_j} - \epsilon _{\beta_j}}, $$ where $s_{\epsilon_{\alpha_j} \pm \epsilon _{\beta_j}}$ denotes the reflection associated to the root $\epsilon_{\alpha_j} \pm \epsilon _{\beta_j}$. The claim is proved.

	\textbf{(5)} Let $d = \deg Q$. By part (1) we know that $m$ is odd and $d$ is even. Write 
	$\gamma = \gamma (\lambda_1, \cdots, \lambda _{n'}).$ Since (\ref{eq:condition for rational}) holds, we know that $\lambda_1,\cdots, \lambda_d $ are the $d$ distinct roots of $Q(\lambda)$, and that $\lambda_1^{\s^d} = \lambda_1$. As $m$ is odd, we write $m = 2t+1$. Using $n' = md/2 = td + \frac{d}{2} $ and using (\ref{eq:condition for rational}), we see that $$\lambda_1^{-1} = \lambda_{n'}^{\s} = \lambda_1^{\s^{n'}} = \lambda_1^{\s^{td + \frac{d}{2}}} = \lambda_1^{\s^{\frac{d}{2}}} = \lambda_{d/2}^{\s}. $$ It then follows from (\ref{eq:condition for rational}) that
	$\Lambda: = (\lambda_1, \cdots, \lambda_{d/2})$ is an admissible tuple with respect to $Q$, and that we have 
	\begin{align}\label{eq:writing gamma}
	\gamma = \gamma(\underbrace{\Lambda, \Lambda ^{-1},\cdots,\Lambda}_{m}).
	\end{align} (Here if $d=2n'$ and $m=1$, the last equality is understood as $\gamma = \gamma(\Lambda)$.)
	
	By part (4) and Lemma \ref{lem:compute TT}, we have $$ \TT(w_i,\gamma)  = \# \{ \gamma' \in T_i (\BF_q) ; \gamma ' = {}^x \gamma \text{ for some } x\in W_i \}. $$
	By the above argument, any such $\gamma'$ must be of the form
	$\gamma ' = \gamma (\Lambda', (\Lambda')\i,\cdots, \Lambda'),$ for a tuple $\Lambda'$ which is admissible with respect to $Q$. Let $N$ be the number of admissible tuples $\Lambda'$ with respect to $Q$, such that $\gamma (\Lambda', (\Lambda')\i,\cdots, \Lambda')$ equals ${}^x \gamma$ for some $x \in W_i$. To finish the proof, it remains to show that $N = d/2$.

	We now compute $N$. By Lemma \ref{lem:even deg}, there are precisely $d$ distinct admissible tuples with respect to $Q$, and they are of the form $\Lambda_1, \Lambda_2,\cdots, \Lambda_{d}$, with $\Lambda_1 = \Lambda$, and $\Lambda_{j} = (\bar \Lambda_{j-1})[1]$ for $2\leq j \leq d$. See Definition \ref{defn:adm tuple} for the notation. For $1\leq j \leq d$, we let $$\gamma_j : = \gamma(\underbrace{\Lambda_j, \Lambda_j^{-1} ,\cdots, \Lambda_j }_{m}). $$ (If $m=1$, then $\gamma_j : = \gamma(\Lambda_j)$.) Thus $N$ is equal to the cardinality of 
	$$\set{j ; 1\leq j \leq d, \gamma_j =  {}^x \gamma \text{ for some } x\in W_i }. $$ 
	
	If $j \geq 3$, then we have $\Lambda_j = (\bar \Lambda_{j-1})[1] = \Lambda_{j-2}[2]$. 
	It easily follows that the Weyl orbit of $\gamma_j$ depends only on the parity of $j$, for any $1\leq j \leq d$. We claim that $\gamma_2$ is not in the same Weyl orbit as $\gamma= \gamma_1$. Once the claim is proved, it follows that $N$ is equal to the number of odd integers $j$ with $1\leq j \leq d$, i.e., $N = d/2$. 
	
	To prove the claim, remember that $m$ is odd. Hence $\gamma_2$ is $W_i$-conjugate to 
	\begin{align}\label{eq:gamma_2}
	\gamma(\underbrace{\Lambda, \Lambda^{-1},\cdots, \Lambda, \Lambda^{-1}}_{m-1}, \bar \Lambda). 
	\end{align}   
	Comparing with (\ref{eq:writing gamma}), and using the fact $\lambda_1, \cdots, \lambda_{d/2}, \lambda_1^{-1},\cdots, \lambda_{d/2}^{-1}$ are all distinct, we easily see that the element (\ref{eq:gamma_2}) is not conjugate to $\gamma$ by the group $W_i \cong (\set{\pm 1}^{\times n'})' \rtimes S_{n'}$.
\end{proof}

\begin{lemma}\label{lem:counting flags} Let $g \in G(\BF_q) \cap \GL(V)^{\mathrm{reg}}$. For each $1\leq i \leq n$, let $\mathcal M_i^g$ be as in \S \ref{subsec:combining results}. We have a bijection 
	\begin{align*}
	\mathcal M_i ^g  & \isom \set{ U\in \BF_q [\lambda]^{\mathrm{monic}} ; \deg U = i-1, UU^* \text{ divides } f_g \text{ in }\BF_q [\lambda]} \\  
	r P_i(\BF_q) & \mapsto f_{\pi_i' (r\i g r)} .
	\end{align*}
	
\end{lemma}
\begin{proof}
	Let
	$(\mathcal M_i^g)'$ be the set of $g$-stable $(i-1)$-dimensional totally isotropic $\BF_q$-subspaces of $\mathbb V$. We know that all $(i-1)$-dimensional totally isotropic $\BF_q$-subspaces of $\mathbb V$ are in the same $G(\BF_q)$-orbit, because $i-1 < n$. \footnote{In contrast, even over the algebraically closed field $k$, there are two $G(k)$-orbits of $n$-dimensional totally isotropic $k$-subspaces of $V$.} Thus we have a bijection 
	\begin{align*}
	\mathcal M_i^g  & \isom  (\mathcal M_i^g)' \\ 
	r P_i(\BF_q)  & \mapsto r W_i.
	\end{align*}
	
	Now given $W \in (\mathcal M_i^g)'$ corresponding to $rP_i(\BF_q) \in \mathcal M_i ^g$, the characteristic polynomial $f_{g|_W}$ of $g|_W$ is equal to $f_{\pi_i'(r\i g r )}$. Hence it suffices to show that the map 
	\begin{align}\label{eq:to show bij even ortho}
	(\mathcal M_i ^g)' \to \set{ U\in \BF_q [\lambda] ^{\mathrm{monic}} ; \deg U = i-1,  UU^* \text{ divides } f_g \text{ in }\BF_q [\lambda]}
	\end{align} sending $W$ to $f_{g|_W}$ (which is obviously well-defined) is a bijection. 
	
	Given any element $U(\lambda)$ of the right hand side of (\ref{eq:to show bij even ortho}), we obtain the $\BF_q$-subspace $\ker U (g) \subset \mathbb V$, which is $g$-stable. Let $S:  = f_g / (UU^*)\in \BF_q[\lambda] .$ We now claim that $\ker U(g)$ has dimension $i-1$ and is totally isotropic. To check this it suffices to replace $\ker U(g)$ by its base change to $k$. Since $g \in \GL(V)^{\mathrm{reg}}$, we know that the Jordan canonical form of $g$ over $k$ has only one Jordan block for each eigenvalue, by Proposition \ref{prop:GL reg}. Analyzing each Jordan block one by one, we see that $(\ker U(g))_k$ is equal to $(SU^*) (g) (V)$, and has dimension $i-1$. To check that $(\ker U(g))_k$ is totally isotropic, let $v \in (\ker U(g))_k$. Let $w \in V$ such that $v = (SU^*) (g) w $. Then 
	\begin{align*}
	[v,v] & = [v, (SU^*) (g) w] = [ v, U^*(g) S(g)w ] = [ U^*(g^{-1}) v , S(g)w ] \\  & = [U(0)\i g^{1-i} U(g) v, S(g) w] =0,  
	\end{align*}
	where the last equality holds because $U(g) v =0$. The claim is proved.
	
	By the claim, $\ker U(g)$ is an element of $(\mathcal M_i^g)'$. It then follows from the Cayley--Hamilton Theorem that $U \mapsto \ker U(g)$ is the inverse map of (\ref{eq:to show bij even ortho}). Hence (\ref{eq:to show bij even ortho}) is a bijection as desired.
\end{proof}
\begin{theorem}\label{thm:RZO case} Let $g \in G(\BF_q) \cap \GL(V)^{\mathrm{reg}}$. We use the notations in Definition \ref{defn:MMM}. For each $Q \in \mathsf{SR}$, we simply write $m_Q$ for $m_Q(f_g)$. The following statements hold.
	\begin{enumerate}
		\item We have $m_{(\lambda +1)} =0$, and $m_{(\lambda -1)}$ is zero or odd. 
		\item If $\tr (g,J ,\mathscr L) \neq 0$, then there is a unique element $Q_0 \in \mathsf {SR}$ such that $m_{Q_0}$ is odd. In this case we also know that $Q_0 \neq \lambda \pm 1$. (In particular, by part (1)  we have $m_{(\lambda +1)}  = m_{(\lambda-1)} =0$ in this case.)
		\item Assume there is a unique element $Q_0 \in \mathsf {SR}$ such that $m_{Q_0}$ is odd. Assume $Q_0 \neq \lambda\pm 1$. Then 
		$$\tr (g, J,\mathscr L ) = \frac{\deg Q_0}{2}  \frac{m_{Q_0}+1}{2}\MMM(f_g).$$
	\end{enumerate}
\end{theorem}
\begin{proof} Part (1) follows from Proposition \ref{prop:dealing with lambda-1} and the fact that $m_{(\lambda +1)}$ must be even in order for $\det g =1$.
	
	By Proposition \ref{prop:implication}, we have $g \in G(\BF_q) \cap G^{\mathrm{reg}}$, and so we may apply Theorem \ref{thm:char formula} to compute $\tr (g, J ,\mathscr L)$ in the following.
	
	Firstly, assume $1 \leq i \leq n$ and $\mathcal M_i^{g,\gamma} \neq \emptyset$ for some $\gamma \in Z_{G_i}(\BF_q)$. Here $Z_{G_i}$ denotes the center of $G_i$. Take $rP_i (\BF_q) \in \mathcal M_i^{g,\gamma}$. Then $f_{\pi_i (r\i g r)}= (\lambda - j) ^{2(n+1-i)}$ for $j=1$ or $-1$, and it follows from Lemma \ref{lem:counting flags} that $$f_g(\lambda) = (\lambda -j) ^{2(n+1-i)} U(\lambda)U^*(\lambda)$$ for some $U(\lambda) \in \BF_q [\lambda]$. Then $m_{(\lambda -j)}$ must be positive even, a contradiction with part (1). Hence $
	\mathcal M_i^{g,\gamma}  = \emptyset$ for all $1 \leq i \leq n$ and all $\gamma \in Z_{G_i}(\BF_q).$
	
	We now prove part (2) of the theorem. Assume $\tr(g, J,\mathscr L) \neq 0$. Then there exist $1 \leq i \leq n$ and $\gamma \in \Gamma_i$ such that  $\mathcal M_i^{g,\gamma} \neq \emptyset$. By the previous paragraph, we know that $\gamma \notin Z_{G_i}(\BF_q)$. Take $rP_i (\BF_q) \in \mathcal M_i^{g,\gamma}$. Then by Theorem \ref{thm:about Ti} (1), we have $f_{\pi_i (r\i g r)} = Q ^{m}$, for some $Q \in \mathsf {SR}- \set{\lambda \pm 1}$ and some odd $m$. Here $Q \neq \lambda \pm 1$ because $\gamma \notin Z_{G_i}$. By Lemma \ref{lem:counting flags} we have $f_g = Q ^m UU^*$ for some $U \in \BF_q [\lambda]^{\mathrm{monic}}$. It then follows that $Q$, which is not $\lambda \pm 1$, is the unique element of $\mathsf {SR}$ with $m_Q$ odd. Part (2) is proved.
	
	We now prove part (3). By Lemma \ref{lem:even deg} we know $\deg Q_0$ is even. Define $$\mathscr I : =\{i ; 1\leq i \leq n, {2(n+1-i)}/{\deg Q_0} \text{ is an odd integer }\leq m_{Q_0} \}. $$
	For $i \in \mathscr I$, define $m_i : = {2(n+1-i)}/{\deg Q_0} .$ Note that $i\mapsto m_i$ is a bijection $\mathscr I \to \{1,3,5,\cdots,m_{Q_0} \}.$
	In particular $|\mathscr I | = (m_{Q_0}+1)/2$. In the proof of part (2), we saw that if $r P_i(\BF_q) \in \mathcal M_i^{g,\gamma}$ for some $1\leq i \leq n$ and some $\gamma \in \Gamma_i$, then
	\begin{align}\label{eq:char poly forced}
	i \in \mathscr I, \text{ and }
	f_{\pi_i (r\i g r)} = Q_0 ^{m_i}.
	\end{align} Conversely, assume $i \in \mathscr I$ and assume $r P_i(\BF_q)\in \mathcal M_i^g$ is such that (\ref{eq:char poly forced}) holds. Then $\pi_i (r \i g r)_s$ is 
	$G_i(k)$-conjugate to an element of $T_i(\BF_q)$, by Theorem \ref{thm:about Ti} (3). By Theorem \ref{thm:about Ti} (4) and the Lang--Steinberg theorem, $\pi_i (r \i g r)_s$ is in fact
	$G_i(\BF_q)$-conjugate to an element of $T_i(\BF_q)$. Thus $rP_i(\BF_q) \in \mathcal M_i ^{g,\gamma}$ for a unique $\gamma \in \Gamma_i$. In conclusion, we have a bijection
	\begin{align}\label{eq:bij} & 
	\set {(i,\gamma, r P_i(\BF_q)) ; 1\leq i \leq n, \gamma \in \Gamma_i, r P_i(\BF_q) \in \mathcal M_i^{g,\gamma}}  \isom \\ \nonumber &   \set {(i, rP_i(\BF_q)) ; i \in \mathscr I,  r P_i(\BF_q)\in \mathcal M_i^g, f_{\pi_i (r\i g r )} = Q_0^{m_i} } \\ 
	\nonumber  & (i,\gamma, rP_i(\BF_q))  \mapsto (i,rP_i(\BF_q)). 
	\end{align}
	We also note that if $(i,\gamma, rP_i(\BF_q))$ is in the left hand side of (\ref{eq:bij}), then $f_{\gamma} = Q_0 ^{m_i}$, and so by Theorem \ref{thm:about Ti} (5) we have 
	\begin{align}\label{eq:d_0/2}
	\TT(w_i,\gamma) = \frac{\deg Q_0}  {2}.
	\end{align}
	Now we compute
	\begin{align*}
	& \tr (g, J,\mathscr L)  =  \sum _{i=1} ^{n}\sum _{\gamma \in \Gamma_i} \# \mathcal M_i^{g,\gamma} \cdot  \TT (w_i , \gamma) && \text{(by Thm. \ref{thm:char formula} )} \\
	& =  \sum _{i\in \mathscr I} \# \set{rP_i(\BF_q)\in \mathcal M_i^{g} ; f_{\pi_i (r\i g r)} = Q_0^{m_i}}  \cdot \frac{\deg Q_0}{2} && \text{(by (\ref{eq:bij}), (\ref{eq:d_0/2})) } \\
	&=   \frac {\deg Q_0}{2}
	\sum _{i\in \mathscr I} \# \set{U \in \BF_q [\lambda]^{\mathrm{monic}} ; UU^* = f_g/Q_0 ^{m_i}} && \text{(by Lem. \ref{lem:counting flags})} \\
	& = \frac {\deg Q_0}{2} \abs{\mathscr I} \MMM(f_g) && \text{(by Lem. \ref{lem:MMM})} \\ 
	& =  \frac{\deg Q_0}{2} \frac{m_{Q_0}+1}{2} \MMM(f_g). && \qedhere
	\end{align*} \end{proof}

\subsection{The odd special orthogonal group}
In this subsection we consider case (\ref{item:odd ortho}) in \S \ref{sec:basic loci}. 

We fix a non-degenerate $2n+1$-dimensional quadratic space $(\mathbb V,[\cdot, \cdot])$ over $\BF_q$, with $n\geq 0$. Let $\mathbb G = \SO (\mathbb V,[\cdot,\cdot])$. Let $V: = \mathbb V\otimes _{\BF_q} k$. By the classification of quadratic forms over $\BF_q$ (see \cite[\S 1.3]{kitaoka}), there exists an $\BF_{q}$-basis $\set{e_1,\cdots, e_{2n+1}}$ of $\mathbb V$, satisfying 
$$[e_{\alpha} ,e_{\beta}]= \delta_{2n+2, \alpha+\beta}, \quad \forall \alpha, \beta \neq n+1 ,$$$$[e_{n+1} ,e_{n+1}] \in \BF_q ^{\times}.$$
For each $1\leq i \leq n+1$, we define $$\mathbb V_i: =  \mathrm{span}_{\BF_{q}}(e_i , e_{i+1}, \cdots, e_{2n+2-i}) \subset \mathbb V, \qquad \mathbb W_i : = \mathrm{span}_{\BF_{q}} ({e_1,\cdots, e_{i}}) \subset \mathbb V. $$ 
We define $V : = \mathbb V \otimes k,$ $V_i: = \mathbb V_i \otimes k,$ $W_i: = \mathbb W_i \otimes k$.

Let $G = \mathbb G_k$. Let $B \subset G$ be the stabilizer of the flag
$W_1 \subset  W_2 \subset\cdots \subset W_{n}$ inside $V$. 
Then $B$ is a $\s$-stable Borel subgroup of $G$.
Let $T$ be the intersection of $G$ with the diagonal torus in $\GL(V)$ under the basis $e_1,\cdots, e_{2n+1}$. Then $T$ is the maximal torus of $G$ contained in $B$. 

We number the simple roots of $(G,B,T)$ according to Bourbaki \cite{bourbaki}. We consider the $\s$-unbranched datum $(J=\BS-\{s_{n}\}, \mathscr L = (s_1, \cdots, s_{n}) ) $. Following the notation of \S \ref{subsec:setting} and \S \ref{subsec:parabolic induction}, we have $i _{\max} = n+1$, and for $1\leq i \leq n+1$ we have
$$\mathbb P_i = \mathrm{Stab}_{\mathbb G} ( \mathbb W_{i-1}) , \qquad \mathbb L_i = \mathbb L_i^{\natural} = \GL(\mathbb W_{i-1}) \times \SO (\mathbb V_i),$$
$$ \mathbb G_i = \SO(\mathbb V_i) = \SO_{2(n+1-i)+1}, \qquad \mathbb H_i = \GL(\mathbb W_{i-1}) = \GL _{i-1}.$$
Here by convention $\mathbb W_0 = 0$ and $\GL_0 = \{1\}$. As in \S \ref{subsec:parabolic induction}, we have natural projections $\pi_i : \mathbb P_i \to \mathbb G_i$ and $\pi_i': \mathbb P_i \to \mathbb H_i$.

For any $h\in G_i(k)$, we denote by $ f_h \in k [\lambda]$ the characteristic polynomial of $h$ acting on $V_i$, which has degree $2(n+1-i)+1$. Thus if $h \in G_i(\BF_q)$, then $f_h$ is self-reciprocal in $\BF_q [\lambda]$. Similarly, for any $h \in H_i(k)$, we denote by $f_h (\lambda) \in k[\lambda]$ the characteristic polynomial of $h$ acting on $W_i$, which has degree $i-1$.

We fix $1\leq i \leq n$. Write $n' $ for $n+1 -i$. Thus $\mathbb G_i = \SO_{2n'+1}$, with $n' \geq 1$. Let $B_i'$ (resp.~$T_i'$) be the intersection of $G_i$ with the upper triangular subgroup (resp.~diagonal subgroup) of $\GL(V_i)$, under the $k$-basis $\set{e_i,\cdots, e_{2n+2-i}}$ of $V_i$. Then $B_i'$ is a $\s$-stable Borel subgroup of $G_i$, and $T_i'$ is a $\s$-stable maximal torus of $G_i$ contained in $B_i'$. For any $(\lambda_1,\cdots, \lambda_{n'}) \in (k^{\times})^{\oplus n'}$, let $\gamma'(\lambda_1,\cdots, \lambda_{n'})$ be the diagonal matrix $\diag(\lambda_1,\cdots, \lambda_{n'}, 1 ,  \lambda_{n'}^{-1},\cdots, \lambda_1^{-1})$ in $\GL(V_i)$ under the same basis. Then $\gamma'$ is an isomorphism $\GG_{m,k}^{n'} \isom T_i'$ (which is in fact defined over $\BF_q$). The Weyl group $W_i$ can be identified with $\set{\pm 1}^{\times n'} \rtimes S_{n'}$. We easily compute that $w_i$ acts on $T_i'$ in the following way: 
$$w_i:  \gamma'(\lambda_1,\cdots, \lambda_{n'}) \longmapsto \gamma' (\lambda_{n'}^{-1},\lambda_1,\cdots, \lambda_{n'-1}).$$ 
Also, $\s$ acts on $T_i'$ in the following way:
$$\s: \gamma'(\lambda_1,\cdots, \lambda_{n'}) \longmapsto \gamma'(\lambda_1^{\s},\cdots,\lambda_{n'-1}^{\s}, \lambda_{n'}^{\s}). $$
Remember that $T_i$ is by definition a $\s$-stable maximal torus of $G_i$ of type $w_i$. From the above discussion, we see that on $T_i$ we have coordinates 
$$ (k^{\times})^{\oplus n'} \isom T_i ,\qquad 	 
(\lambda_1,\cdots, \lambda_{n'}) \mapsto \gamma (\lambda_1,\cdots, \lambda_{n'}), $$  such that the eigenvalues (with multiplicities) of $\gamma (\lambda_1,\cdots, \lambda _{n'})$ acting on $V_i\cong k^{2n'+1}$ are $$\lambda_1,\cdots ,\lambda_{n'}, \lambda_1 ^{-1},\cdots, \lambda_{n'} ^{-1}, 1, $$
and such that 
\begin{align}\label{eq:Galois action, odd ortho}
\gamma(\lambda_1,\cdots, \lambda_{n'})^{\sigma} = \gamma ((\lambda_{n'}^{-1})^{\s}, \lambda_1 ^{\s}, \lambda_2 ^{\s}, \cdots, \lambda_{n'-1} ^{\s}).
\end{align}

\begin{theorem}\label{thm:about Ti, odd ortho}   We have the following statements about $T_i (\BF_q)$.
	\begin{enumerate}
		\item If $\gamma \in T_i(\BF_q)$, then  $f_{\gamma} (\lambda)= Q(\lambda)^m(\lambda-1)$ for some $Q\in \mathsf{SR}$, and some positive integer $m$. Moreover, either $Q(\lambda) = \lambda \pm 1$, or $m$ is odd.
		
		\item Let $Q \in \mathsf{SR}$. Assume $ m$ is an odd integer such that $m \deg Q = 2n'$. (In particular $Q(\lambda) \neq \lambda \pm 1$ for degree reasons). Then there exists $\gamma \in T_i (\BF_q)$ with $f_{\gamma}(\lambda)= Q (\lambda)^m (\lambda -1)$. 
		\item Let $Q$ and $m$ be as in part (2). Let $\gamma \in G_i(k)$ be a semi-simple element such that $f_{\gamma}(\lambda) = Q(\lambda)^m(\lambda -1)$. Then $\gamma$ is $G_i(k)$-conjugate to an element of $T_i (\BF_q)$.
		\item For any $\gamma \in T_i(\BF_q)$ such that $(\lambda+1)$ does not divide $f_{\gamma} (\lambda)$, the centralizer $G_{i,\gamma}$ is connected. 
		\item Let $\gamma \in T_i (\BF_q)$. Write $f_{\gamma} (\lambda)= Q(\lambda)^m (\lambda-1)$ as in part (1). Assume $Q(\lambda) \neq \lambda \pm 1$. Then $\TT(w_i,\gamma) = \deg Q$. 
	\end{enumerate}
\end{theorem}
\begin{proof}
	Observing that (\ref{eq:Galois action, odd ortho}) has the same form as (\ref{eq:Galois action}), one proves parts (1) (2) (3) in exactly the same way as parts (1) (2) (3) of Theorem \ref{thm:about Ti}. (In fact the proof of part (3) here is even easier, due to the fact that the Weyl group $W_i$ in the current case is larger.)
	
	The proof of part (4) is also similar to the proof of Theorem \ref{thm:about Ti} (4). In fact, using the same notation as the proof of Theorem \ref{thm:about Ti} (4), we can again reduce to the case $A(x) = \emptyset$. Then the new feature is that $\# B(x)$ need not be even. However, since $-1$ is not an eigenvalue by assumption, we know that $\epsilon _{\alpha} (\gamma) =1$ for all $\alpha \in B(x)$. Then $x$ is the product of the reflections associated to the roots $\epsilon _{\alpha}$, for $\alpha \in B(x)$. 
	
	The proof of part (5) is again similar to the proof of Theorem \ref{thm:about Ti} (5), the only difference being that here all $\deg Q$ admissible tuples $\Lambda'$ show up in the counting, as opposed to only $(\deg Q)/2$ of them. This is due to the fact that the Weyl group $W_i$ is larger in the current case. \end{proof}

\begin{lemma}\label{lem:counting flags, odd ortho} Let $g \in G(\BF_q) \cap \GL(V)^{\mathrm{reg}}$. For each $1\leq i \leq n+1$, let $\mathcal M_i^g$ be as in \S \ref{subsec:combining results}. We have a bijection 
	\begin{align*}
	\mathcal M_i ^g  & \isom  \set{ U\in \BF_q [\lambda] ^{\mathrm{monic}} ; \deg U = i-1, UU^* \text{ divides } f_g \text{ in }\BF_q [\lambda]} \\
	r P_i(\BF_q) & \mapsto f_{\pi_i' (r\i g r)} .
	\end{align*}
\end{lemma}
\begin{proof}
	The proof is identical to the proof of Lemma \ref{lem:counting flags}, based on the fact that all $(i-1)$-dimensional totally isotropic $\BF_q$-subspaces of $\mathbb V$ are in the same $G(\BF_q)$-orbit. 
\end{proof}
\begin{theorem}\label{thm:odd ortho}Let $g \in G(\BF_q) \cap \GL(V)^{\mathrm{reg}}$. We use the notations in Definition \ref{defn:MMM}. For each $Q \in \mathsf{SR}$, we simply write $m_Q$ for $m_Q(f_g)$. The following statements hold. 
	\begin{enumerate}
		\item We have $m _{(\lambda +1)} = 0$, and $m_{(\lambda -1)}$ is odd. 
		\item If $\tr (g,J ,\mathscr L) \neq 0$, then inside $ \mathsf {SR}- \set{\lambda -1}$ there is at most one element $Q_0$ with $m_{Q_0}$ odd. 
		\item Assume there exists a unique $Q_0 \in \mathsf {SR} - \set{\lambda -1}$ such that $m_{Q_0}$ is odd. Then 
		$$\tr (g, J,\mathscr L ) = \deg Q_0 \frac{m_{Q_0}+1}{2}\MMM(f_g).$$
		\item Assume there is no element $Q_0 \in \mathsf {SR} - \set{\lambda -1}$ such that $m_{Q_0}$ is odd. Then 
		$$\tr (g, J,\mathscr L ) = \frac{m_{(\lambda -1)}+1}{2}\MMM(f_g).$$ 
	\end{enumerate}
\end{theorem}
\begin{proof}
	Part (1) follows from Proposition \ref{prop:dealing with lambda-1}, the fact that $\lambda -1$ always divides $f_g$, and the fact that $m_{(\lambda +1)}$ must be even in order for $\det g =1$. 
	
	By Proposition \ref{prop:implication}, we have $g \in G(\BF_q) \cap G^{\mathrm{reg}}$, and so we may apply Theorem \ref{thm:char formula} to compute $\tr (g, J ,\mathscr L)$ in the following.
	
	We prove part (2). Assume $\tr(g, J,\mathscr L) \neq 0$. Then there exist $1 \leq i \leq n+1$ and $\gamma \in \Gamma_i$ such that  $\mathcal M_i^{g,\gamma} \neq \emptyset$.  Take $rP_i (\BF_q) \in \mathcal M_i^{g,\gamma}$. If $i =n+1$, then $f_{\pi_i (r\i g r)} = \lambda -1$. If $1 \leq i \leq n$, then by Theorem \ref{thm:about Ti, odd ortho} (1), we have $f_{\pi_i (r\i g r)}= Q(\lambda) ^{m} (\lambda -1)$, for some $Q \in \mathsf {SR}$ and some integer $m>0$. To simplify notation we set $Q: = 1$ and $m: =0$ when $i = n+1$. Then in all cases $f_{\pi_i (r\i g r)} = Q(\lambda) ^{m} (\lambda -1)$. By Lemma \ref{lem:counting flags, odd ortho} we have 
	\begin{align}\label{eq:analyze parity}
	f_g (\lambda) = Q (\lambda)^m  (\lambda -1)U(\lambda)U^*(\lambda)
	\end{align}
	for some $U \in \BF_q [\lambda]^{\mathrm{monic}}$. Now if $Q(\lambda) = \lambda-1$ or $m =0$, then it follows from (\ref{eq:analyze parity}) that $\lambda-1$ is the only element of $\mathsf {SR}$ whose multiplicity in $f$ is odd. On the other hand, if $Q(\lambda) \neq \lambda-1$ and $m> 0 $, then $Q(\lambda) \neq \lambda\pm 1$ by part (1), and we know that $m$ is odd by Theorem \ref{thm:about Ti, odd ortho} (1). In this case, we conclude from (\ref{eq:analyze parity}) that $m_Q$ is odd, and that $Q$ is the unique element of $\mathsf {SR} -\set{\lambda -1}$ whose multiplicity in $f$ is odd. Part (2) is proved.
	
	We remark that the above analysis shows that under the sole assumption that $\mathsf {SR} -\set{\lambda -1}$ has an element $Q$ with $m_{Q}$ odd, we have	  \begin{align} \label{eq:n+1 does not contribute}
	\mathcal M_{n+1} ^{g, \gamma} = \emptyset,\quad \forall \gamma \in \Gamma _{n+1}
	\end{align} (where $\Gamma_{n+1}$ in fact has only one element, the identity).
	
	We now prove part (3). Under the hypothesis of part (3), the assertion (\ref{eq:n+1 does not contribute}) holds. Since $Q_0 \neq \lambda \pm 1$, by Lemma \ref{lem:even deg} we know that $\deg Q_0$ is even. 
	Define $$\mathscr I : =\{i ; 1\leq i \leq n, {2(n+1-i)}/{\deg Q_0} \text{ is an odd integer }\leq m_{Q_0} \}. $$
	For $i \in \mathscr I$, define $m_i : = {2(n+1-i)}/{\deg Q_0} .$ Note that $i\mapsto m_i$ is a bijection $\mathscr I \to \{1,3,5,\cdots,m_{Q_0} \}.$
	In particular $|\mathscr I | = (m_{Q_0}+1)/2$. Similar to the bijection (\ref{eq:bij}), we obtain a bijection
	\begin{align}\label{eq:bij, odd ortho} & 
	\set {(i,\gamma, r P_i(\BF_q)) ; 1\leq i \leq n, \gamma \in \Gamma_i, r P_i(\BF_q) \in \mathcal M_i^{g,\gamma}} \isom  \\ \nonumber & \set {(i, rP_i(\BF_q)) ; i \in \mathscr I,  r P_i(\BF_q)\in \mathcal M_i^g, f_{\pi_i (r\i g r )} = Q_0^{m_i}\cdot (\lambda -1) }  \\ \nonumber & (i,\gamma, r P_i(\BF_q)) \mapsto (i,rP_i(\BF_q)),
	\end{align}
	based on parts (3) (4) of Theorem \ref{thm:about Ti, odd ortho} (part (4) being applicable because $m_{(\lambda+1)} =0$). We also note that if $(i,\gamma, rP_i(\BF_q))$ is in the left hand side of (\ref{eq:bij, odd ortho}), then $f_{\gamma}(\lambda) = Q_0(\lambda) ^{m_i}(\lambda-1)$, and so by Theorem \ref{thm:about Ti, odd ortho} (5) we have 
	\begin{align}\label{eq:d_0}
	\TT(w_i,\gamma) = \deg Q_0.
	\end{align}
	Now we compute
	\begin{align*}
	& \tr (g, J,\mathscr L)    = \sum _{i=1} ^{n}\sum _{\gamma \in \Gamma_i} \# \mathcal M_i^{g,\gamma} \cdot  \TT (w_i , \gamma)   && \text{(by Thm. \ref{thm:char formula}, and (\ref{eq:n+1 does not contribute}) )} \\
	& = \sum _{i\in \mathscr I} \# \set{rP_i(\BF_q)\in \mathcal M_i^{g} ; f_{\pi_i (r\i g r)} = Q_0^{m_i} \cdot (\lambda -1)}  \cdot  \deg Q_0    && \text{(by (\ref{eq:bij, odd ortho}), (\ref{eq:d_0}))}\\
	& =   \deg Q_0\sum _{i\in \mathscr I} \# \set{U \in \BF_q [\lambda]^{\mathrm{monic}} ; UU^* = \frac{f_g} {Q_0 ^{m_i}(\lambda-1)}}   && \text{(by Lem. \ref{lem:counting flags, odd ortho})}  \\
	& = \deg Q_0 \abs{\mathscr I} \MMM(\frac{f_g}{\lambda -1})  & & \text{(by Lem. \ref{lem:MMM} \ignore{applied to $\frac{f_g}{\lambda-1}$ and $Q_0$})}\\
	& = \deg Q_0 \frac{m_{Q_0}+1}{2}\MMM(f_g).
	\end{align*}
	In the second last step Lemma \ref{lem:MMM} is applicable because $Q_0$ is the unique element of $\mathsf {SR}$ such that $m_{Q_0} (f_g/ (\lambda-1))$ is odd, which follows from the definition of $Q_0$ and part (1). Part (3) is proved.
	
	Finally we prove part (4). By the proof of part (2), we know that for any $1\leq i \leq n+1$, we have $\mathcal M_i ^{g,\gamma} \neq \emptyset$ only if $f_{\gamma} (\lambda) = (\lambda -1) ^{2(n+1-i) +1}$. The last condition is equivalent to $\gamma = \id \in T_i$. 
	
	Define $$\mathscr I = \set{i \in \ZZ; n +1 - \frac{m_{(\lambda -1)} -1}{2} \leq i \leq n+1}.$$
	Now assume $rP_i(\BF_q) \in \mathcal M_i ^{g, \id }$ for some $1\leq i\leq n+1$. Then we have 
	\begin{align}\label{eq:forced}
	f_{\pi_i (r\i g r)} (\lambda)= (\lambda -1 ) ^{2(n+1-i) +1}.
	\end{align} In particular, $2(n+1-i) +1 \leq m_{\lambda -1}$, and so $i \in \mathscr I$. Conversely, assume $i \in \mathscr I$, and $r P_i(\BF_q) \in \mathcal M_i ^g$ such that (\ref{eq:forced}) holds. Then $r P_i(\BF_q) \in \mathcal M_i ^{g, \id}$ because the only semi-simple element of $G_i$ whose characteristic polynomial equals $(\lambda -1 ) ^{2(n+1-i) +1}$ is the identity. Therefore similar to the proof of part (3), we have 
	\begin{align*}
	\tr(g, J ,\mathscr L) & = \sum _{i \in \mathscr I} \TT(w_i , \id) \cdot \# \set{ U \in \BF_q [\lambda] ^{\mathrm{monic}} ; UU^* = f_g /(\lambda -1)^{2(n+1-i) +1}}\\ & = \sum _{i\in \mathscr I} \TT (w_i, \id) \MMM(f_g).
	\end{align*}
	By Definition \ref{defn:TT}, we have $\TT(w_i,\id) =1$ for each $i \in \mathscr I$. Hence
	\[ \pushQED{\qed} \tr(g, J ,\mathscr L) =  \abs{\mathscr I}\MMM(f_g)  = \frac{m_{(\lambda -1)} +1}{2} \MMM(f_g) .  \qedhere \]	
\end{proof}

\subsection{The symplectic group} In this subsection we consider case (\ref{item:symp}) in \S \ref{sec:basic loci}. 

We fix a $2n$-dimensional symplectic space $(\mathbb V,[\cdot, \cdot])$ over $\BF_{q}$, with $n \geq 0$. Let $\mathbb G = \Sp (\mathbb V, [\cdot, \cdot])$. We fix an $\BF_{q}$-basis $\set{e_1,\cdots, e_{2n}}$ of $\mathbb V$, satisfying 
$$[e_{\alpha} ,e_{\beta}]= \delta_{2n+1, \alpha+\beta}, ~\forall 1\leq \alpha \leq \beta \leq 2n. $$
For each $1\leq i \leq n+1$, we define $$\mathbb V_i: =  \mathrm{span}_{\BF_{q}}(e_i , e_{i+1}, \cdots, e_{2n+1-i}) \subset \mathbb V, \qquad \mathbb W_i : = \mathrm{span}_{\BF_{q}} ({e_1,\cdots, e_{i}}) \subset \mathbb V. $$ 
We define $V : = \mathbb V \otimes k,$ $V_i: = \mathbb V_i \otimes k,$ $W_i: = \mathbb W_i \otimes k$. 

Let $G = \mathbb G_k$. Let $B \subset G$ be the stabilizer of the flag
$W_1 \subset  W_2 \subset\cdots \subset W_{n}$ inside $V$. Then $B$ is a $\s$-stable Borel subgroup of $G$. Let $T$ be the intersection of $G$ with the diagonal torus in $\GL(V)$ under the basis $e_1,\cdots, e_{2n}$. Then $T$ is the maximal torus of $G$ contained in $B$.

We number the simple roots of $(G,B,T)$ according to Bourbaki \cite{bourbaki}. We consider the $\s$-unbranched datum $(J=\BS-\{s_{n}\}, \mathscr L = (s_1, \cdots, s_{n}) ) $. Following the notation of \S \ref{subsec:setting} and \S \ref{subsec:parabolic induction}, we have $i _{\max} = n+1$, and for $1\leq i \leq n+1$ we have
$$\mathbb P_i = \mathrm{Stab}_{\mathbb G} ( \mathbb W_{i-1}), \qquad \mathbb L_i = \mathbb L_i^{\natural} = \GL(\mathbb W_{i-1}) \times \Sp (\mathbb V_i),$$
$$ \mathbb G_i = \Sp(\mathbb V_i) = \Sp_{2(n+1-i)}, \qquad \mathbb H_i = \GL(\mathbb W_{i-1}) = \GL _{i-1}.$$
Here by convention $\mathbb W_0 = 0$ and $\GL_0 = \{1\}$. As in \S \ref{subsec:parabolic induction}, we have natural projections $\pi_i : \mathbb P_i \to \mathbb G_i$ and $\pi_i': \mathbb P_i \to \mathbb H_i$.

For any $h\in G_i(k)$, we denote by $ f_h \in k [\lambda]$ the characteristic polynomial of $h$ acting on $V_i$, which has degree $2(n+1-i)$. Thus if $h \in G_i(\BF_q)$, then $f_h$ is self-reciprocal in $\BF_q [\lambda]$. Similarly, for any $h \in H_i(k)$, we denote by $f_h (\lambda) \in k[\lambda]$ the characteristic polynomial of $h$ acting on $W_i$, which has degree $i-1$.

\begin{theorem}\label{thm:about Ti, symp} We fix $1\leq i \leq n$. Write $n' $ for $n+1 -i$. Thus $\mathbb G_i = \Sp_{2n'}$, with $n' \geq 1$. We have the following statements about $T_i (\BF_q)$.
	\begin{enumerate}
		\item If $\gamma \in T_i(\BF_q)$, then $f_{\gamma}= Q^m$ for some irreducible, self-reciprocal $Q \in \BF_q [\lambda]$, and some positive integer $m$. Moreover, either $Q(\lambda) = \lambda \pm 1$, or $m$ is odd.
		
		\item Let $Q \in \BF_q [\lambda]$ be an irreducible, self-reciprocal polynomial. Assume $ m$ is an odd integer such that $m \deg Q = 2n'$. (In particular $Q(\lambda) \neq \lambda \pm 1$). Then there exists $\gamma \in T_i (\BF_q)$ with $f_{\gamma}= Q ^m$. 
		\item Let $Q$ and $m$ be as in part (2). Let $\gamma \in G_i(k)$ be a semi-simple element such that $f_{\gamma}=Q^m$. Then $\gamma$ is $G_i(k)$-conjugate to an element of $T_i (\BF_q)$.
		\item Let $\gamma \in T_i (\BF_q)$. Write $f_{\gamma} =Q^m$ as in part (1). Assume $Q(\lambda) \neq \lambda \pm 1$. Then $\TT(w_i,\gamma) = \deg Q$. 
	\end{enumerate}
\end{theorem}
\begin{proof} Since the root datum of $G_i$ is dual to that of an odd special orthogonal group, the torus $T_i$ has a similar description as the torus $T_i$ in Theorem \ref{thm:about Ti, odd ortho}. Thus the proof of the theorem is identical to the proof of Theorem \ref{thm:about Ti, odd ortho}.\end{proof}
\begin{remark}\label{rem:symp auto conn}
	In Theorem \ref{thm:about Ti, symp} we do not state the analogue of Theorem \ref{thm:about Ti} (4) and Theorem \ref{thm:about Ti, odd ortho} (4). This is because $G$ being simply connected, the centralizer in $G$ of any semi-simple element is automatically connected, see \S \ref{subsubsec:conn setting}. 
\end{remark}
\begin{lemma}\label{lem:counting flags, symp} Let $g \in G(\BF_q) \cap \GL(V)^{\mathrm{reg}}$. For each $1\leq i \leq n+1$, let $\mathcal M_i^g$ be as in \S \ref{subsec:combining results}. We have a bijection 
	\begin{align*}
	\mathcal M_i ^g  & \isom  \set{ U\in \BF_q [\lambda] ^{\mathrm{monic}} ; \deg U = i-1, UU^* \text{ divides } f_g \text{ in }\BF_q [\lambda]} \\  r P_i(\BF_q) &  \mapsto f_{\pi_i' (r\i g r)} .
	\end{align*}
\end{lemma}
\begin{proof}
	The proof is identical to the proof of Lemma \ref{lem:counting flags}, based on the fact that all $(i-1)$-dimensional totally isotropic $\BF_q$-subspaces of $\mathbb V$ are in the same $G(\BF_q)$-orbit. 
\end{proof}
\begin{theorem}\label{thm:symp} Let $g \in G(\BF_q) \cap \GL(V)^{\mathrm{reg}}$.We use the notations in Definition \ref{defn:MMM}. For each $Q \in \mathsf{SR}$, we simply write $m_Q$ for $m_Q(f_g)$. The following statements hold.
	\begin{enumerate}
		\item Assume $\tr (g,J ,\mathscr L) \neq 0$. Then inside $ \mathsf {SR}$ there is at most one element $Q_0$ with $m_{Q_0}$ odd. Moreover, if such $Q_0$ exists, then $Q_0 \neq \lambda \pm 1$.
		\item Assume there exists a unique $Q_0 \in \mathsf {SR}$ such that $m_{Q_0}$ is odd. Assume $Q_0 \neq \lambda \pm 1$. Then 
		$$\tr (g, J,\mathscr L ) = \deg Q_0 \frac{m_{Q_0}+1}{2}\MMM(f_g). $$
		\item Assume there is no element $Q_0 \in \mathsf {SR}$ such that $m_{Q_0}$ is odd. Then 
		$$\tr (g, J,\mathscr L ) = \left( \frac{m_{(\lambda -1)}}{2} + 1 +  \frac{m_{(\lambda +1)}}{2}  \right)\MMM(f_g).$$ 
	\end{enumerate}
\end{theorem}
\begin{proof}
	By Proposition \ref{prop:implication}, we have $g \in G(\BF_q) \cap G^{\mathrm{reg}}$, and so we may apply Theorem \ref{thm:char formula} to compute $\tr (g, J ,\mathscr L)$ in the following.
	
	We prove part (1). Assume $\tr(g, J,\mathscr L) \neq 0$. Then there exist $1 \leq i \leq n+1$ and $\gamma \in \Gamma_i$ such that  $\mathcal M_i^{g,\gamma} \neq \emptyset$.  Take $rP_i (\BF_q) \in \mathcal M_i^{g,\gamma}$. If $i =n +1$, then $f_{\pi_i (r\i g r)} = 1$. If $1\leq i\leq n$, then by Theorem \ref{thm:about Ti, symp} (1), we have $f_{\pi_i (r\i g r)} = Q^m$, for some $Q \in \mathsf{SR}$ and some integer $m>0$. To simplify notation we set $Q:=1$ and $m:=0$ when $i=n+1$. Then in all cases $f_{\pi_i (r\i g r)} = Q^m$. By Lemma \ref{lem:counting flags, symp} we have $
	f_g = Q ^m UU^* $
	for some $U \in \BF_q [\lambda]^{\mathrm{monic}}$. It immediately follows that inside $\mathsf {SR}$ there is at most one element whose multiplicity in $f_g$ is odd. Moreover, if such an element exists, denoted by $Q_0$, then $Q$ in the current discussion must equal to $Q_0$, and $m$ must be odd. (In particular, $i \leq n$.) In this case, we show that $Q_0 \neq \lambda \pm 1$. In fact, if $Q_0 = \lambda\pm 1$, then $m$ is even because $Q^m = Q_0 ^m$ has even degree. This contradicts with our previous assertion that $m$ must be odd. Part (1) is proved. 
	
	We remark that the above analysis also shows that under the sole assumption that $\mathsf {SR}$ has an element $Q$ with $m_{Q}$ odd, we have	\begin{align} \label{eq:n+1 does not contribute, symp}
	\mathcal M_{n+1} ^{g, \gamma} = \emptyset,\quad \forall \gamma \in \Gamma _{n+1}
	\end{align} (where $\Gamma_{n+1}$ in fact has only one element, the identity).
	
	We now prove part (2). Under the hypothesis of part (2), the assertion (\ref{eq:n+1 does not contribute, symp}) holds. Since $Q_0 \neq \lambda \pm 1$, by Lemma \ref{lem:even deg} we know that $\deg Q_0$ is even. 
	Define $$\mathscr I : =\{i ; 1\leq i \leq n, {2(n+1-i)}/{\deg Q_0} \text{ is an odd integer }\leq m_{Q_0} \}. $$
	For $i \in \mathscr I$, define $m_i : = {2(n+1-i)}/{\deg Q_0} .$ Note that $i\mapsto m_i$ is a bijection $\mathscr I \to \{1,3,5,\cdots,m_{Q_0} \}.$
	In particular $|\mathscr I | = (m_{Q_0}+1)/2$. Similar to the bijection (\ref{eq:bij}), we obtain a bijection
	\begin{align}\label{eq:bij, symp} & \set {(i,\gamma, r P_i(\BF_q)) ; 1\leq i \leq n, \gamma \in \Gamma_i, r P_i(\BF_q) \in \mathcal M_i^{g,\gamma}} \isom  \\ \nonumber &  \set {(i, rP_i(\BF_q)) ; i \in \mathscr I,  r P_i(\BF_q)\in \mathcal M_i^g, f_{\pi_i (r\i g r )} = Q_0^{m_i} } \\
	\nonumber & (i,\gamma, r P_i(\BF_q))  \mapsto (i,rP_i(\BF_q))
	\end{align}
	based on Theorem \ref{thm:about Ti, symp} (3) and Remark \ref{rem:symp auto conn}.
	We also note that if $(i,\gamma, rP_i(\BF_q))$ is in the left hand side of (\ref{eq:bij, symp}), then $f_{\gamma} = Q_0 ^{m_i}$, and so by Theorem \ref{thm:about Ti, symp} (4) we have 
	\begin{align}\label{eq:d_0, symp}
	\TT(w_i,\gamma) = \deg Q_0.
	\end{align}
	Now we compute
	\begin{align*}
	& \tr (g, J,\mathscr L)    = \sum _{i=1} ^{n}\sum _{\gamma \in \Gamma_i} \# \mathcal M_i^{g,\gamma} \cdot  \TT (w_i , \gamma)  && \text{(by Thm. \ref{thm:char formula}, and (\ref{eq:n+1 does not contribute, symp}) )} \\
	& = \sum _{i\in \mathscr I} \# \set{rP_i(\BF_q)\in \mathcal M_i^{g} ; f_{\pi_i (r\i g r)} = Q_0^{m_i}}  \cdot  \deg Q_0  && \text{(by (\ref{eq:bij, symp}), (\ref{eq:d_0, symp}))}\\
	& =   \deg Q_0\sum _{i\in \mathscr I} \# \set{U \in \BF_q [\lambda]^{\mathrm{monic}} ; UU^* = f_g/ Q_0 ^{m_i}}. && \text{(by Lem. \ref{lem:counting flags, symp} )} 
	\\ & = \deg Q_0 \abs{\mathscr I} \MMM(f_g) && \text{(by Lem. \ref{lem:MMM})}\\
	& = \deg Q_0 \frac{m_{Q_0}+1}{2}\MMM(f_g).
	\end{align*}
	Part (2) is proved.
	
	Finally we prove part (3). We claim that for each $1\leq i \leq n+1$, we have $\mathcal M_i ^{g,\gamma} \neq \emptyset$ for some $\gamma \in \Gamma _i$ only if $f_{\gamma} (\lambda) = (\lambda  \pm 1) ^{2(n+1-i)}$. In fact, assume this is not the case. Take $rP_i(\BF_q) \in \mathcal M_i ^{g,\gamma} $. Then by Theorem \ref{thm:about Ti, symp} (1), we have $f_{\pi_i (r\i g r)} = Q^m$, for some $Q \in \mathsf {SR}$ and some odd integer $m$. By Lemma \ref{lem:counting flags, symp} we have $
	f_g = Q ^m UU^*
	$ for some $U \in \BF_q [\lambda]^{\mathrm{monic}}$, contradicting with the assumption that there is no element in $\mathsf{SR}$ with odd multiplicity in $f_g$. The claim is proved. 
	
	Define $$\mathscr I = \set{i \in \ZZ; n +1 - \frac{m_{(\lambda -1)}}{2} \leq i \leq n+1},\quad \mathscr J = \set{i \in \ZZ; n +1 - \frac{m_{(\lambda + 1)}}{2} \leq i \leq n}.$$
	Now assume $rP_i(\BF_q) \in \mathcal M_i ^{g, \gamma }$ for some $1\leq i\leq n+1$ and some $\gamma \in \Gamma _i$. Then by the previous claim one of the following two statements holds:
	\begin{itemize}
		\item $i \in \mathscr I,$ and $f_{\pi_i (r\i g r)} (\lambda)= (\lambda -1 ) ^{2(n+1-i)}$.
		\item $i \in \mathscr J,$ and $f_{\pi_i (r\i g r)} (\lambda)= (\lambda +1 ) ^{2(n+1-i)}$.
	\end{itemize}
	Moreover, in the above two cases, the image of $\gamma$ in $\GL(V_i)$ is $\id$ and $-\id $ respectively. Conversely, if $i \in \mathscr I$ and if $r P_i(\BF_q) \in \mathcal M_i ^g$ is such that $f_{\pi_i (r\i g r)} (\lambda)= (\lambda -1 ) ^{2(n+1-i)}$, then $r P_i(\BF_q) \in \mathcal M_i ^{g, \id}$. Similarly, if $i\in \mathscr J$  and if $r P_i(\BF_q) \in \mathcal M_i ^g$ is such that $f_{\pi_i (r\i g r)} (\lambda)= (\lambda +1 ) ^{2(n+1-i)}$, then $r P_i(\BF_q) \in \mathcal M_i ^{g, -\id}$. Therefore as in the proof of part (2), we have 
	\begin{align*}
	\tr(g, J ,\mathscr L) &  = \sum _{i \in \mathscr I} \TT(w_i , \id) \cdot \# \set{ U \in \BF_q [\lambda] ^{\mathrm{monic}} ; UU^* = f_g /(\lambda -1)^{2(n+1-i)}} \\ &  + \sum _{i \in \mathscr J} \TT(w_i , -\id) \cdot \# \set{ U \in \BF_q [\lambda] ^{\mathrm{monic}} ; UU^* = f_g /(\lambda +1)^{2(n+1-i)}} .
	\end{align*}
	Let $ i \in \mathscr I$. By the obvious analogue of Lemma \ref{lem:MMM} applied to $f_g/ (\lambda -1) ^{2(n+1-i)}$ and $Q_0 =1$, we have 
	$$\# \set{ U \in \BF_q [\lambda] ^{\mathrm{monic}} ; UU^* = f_g /(\lambda -1)^{2(n+1-i)}} = \MMM(f_g/ (\lambda -1) ^{2(n+1-i)}),$$ which is equal to $\MMM(f_g)$. Similarly, for $ i \in \mathscr J$, we have 
	$$\# \set{ U \in \BF_q [\lambda] ^{\mathrm{monic}} ; UU^* = f_g /(\lambda + 1)^{2(n+1-i)}} = \MMM(f_g).$$ On the other hand by Definition \ref{defn:TT} we have $\TT(w_i,\id) =1$ for all $i \in \mathscr I$ and  $\TT(w_i,-\id) =1$ for all $i \in \mathscr J$. Therefore 
	\[ \pushQED{\qed} \tr(g, J ,\mathscr L) =  (\abs{\mathscr I}+ \abs{\mathscr J} ) \MMM(f_g) = \left( \frac{m_{(\lambda -1)}}{2} + 1 +  \frac{m_{(\lambda +1)}}{2}  \right) \MMM(f_g). \qedhere \]
\end{proof}
\subsection{The odd unitary group}\label{subsec:odd unitary} In this subsection we consider case (\ref{item:unitary}) in \S \ref{sec:basic loci}. 

We fix a $(2n+1)$-dimensional Hermitian space $(\mathbb V,[\cdot, \cdot])$ over $\BF_{q^2}$ (for the quadratic extension $\BF_{q^2}/ \BF_q$), with $n \geq 0$. Let $\mathbb G = \UU (\mathbb V, [\cdot,\cdot])$. By \cite[Proposition 2.15]{PR}, the Witt index of $(\mathbb V,[\cdot,\cdot])$ is equal to the $\BF_q$-rank of $\mathbb G$, which we know is $n$. Also the norm map $\BF_{q^2}^{\times} \to \BF_q ^{\times}$ is surjective. Hence there exists an $\BF_{q^2}$-basis $\set{e_1,\cdots, e_{2n+1}}$ of $\mathbb V$, satisfying 
$$[e_{\alpha} ,e_{\beta}]= \delta_{2n+2, \alpha+\beta}. $$
For each $1\leq i \leq n+1$, we define $$\mathbb V_i: =  \mathrm{span}_{\BF_{q^2}}(e_i , e_{i+1}, \cdots, e_{2n+2-i}) \subset \mathbb V, \qquad \mathbb W_i : = \mathrm{span}_{\BF_{q^2}} ({e_1,\cdots, e_{i}}) \subset \mathbb V. $$ We fix an embedding $\BF_{q^2} \to k$, viewed as the identity, and we let $V: =\mathbb  V\otimes _{\BF_{q^2}} k$. For each $1\leq i\leq n+1$ we also let $V_i : = \mathbb V_i \otimes _{\BF_{q^2}} k \subset V$, and $W_i: = \mathbb W_i \otimes _{\BF_{q^2}} k \subset V$.

Let $G = \mathbb G_k$. The action of $G$ on $\mathbb V\otimes _{\BF_q} k \cong V \oplus (\mathbb V\otimes_{ \BF_{q^2},\s} k )$ preserves the subspace $V$, and this induces a $k$-isomorphism $G \cong \GL(V)$.  Let $B \subset G$ (resp.~$T \subset G$) be the upper triangular subgroup (resp.~diagonal subgroup) under the basis $\set{e_1,\cdots, e_{2n+1}}$. Then $B$ is a $\s$-stable Borel subgroup of $G$, and $T$ is the maximal torus of $G$ contained in $B$. 

We number the simple roots of $(G,B,T)$ according to Bourbaki \cite{bourbaki}. We consider the $\s$-unbranched datum $(J=\BS-\{s_{n}\}, \mathscr L = (s_1, \cdots, s_{n}) ) $. Following the notation of \S \ref{subsec:setting} and \S \ref{subsec:parabolic induction}, we have $i _{\max} = n+1$, and for $1\leq i \leq n+1$ we have
$$\mathbb P_i = \mathrm{Stab}_{\mathbb G} ( \mathbb W_{i-1}) , \qquad \mathbb L_i = \mathbb L_i^{\natural} = \GL_{\BF_{q^2}}(\mathbb W_{i-1}) \times \UU (\mathbb V_i),$$ $$ \mathbb G_i = \UU(\mathbb V_i) = \UU_{2(n+1-i)+1} , \qquad \mathbb H_i = \GL_{\BF_{q^2}}(\mathbb W_{i-1}) = \Res_{\BF_{q^2} / \BF_q} \GL _{i-1}.$$
Here by convention $\mathbb W_0 = 0$ and $\GL_0 = \{1\}$. As in \S \ref{subsec:parabolic induction}, we have natural projections 
$\pi_i: \mathbb L_i \to \mathbb G_i$ and $\pi_i ' : \mathbb L_i \to \mathbb H_i.$

The action of $G_i$ on $\mathbb V_i\otimes _{\BF_q} k \cong V_i \oplus (\mathbb V_i \otimes_{ \BF_{q^2},\s} k )$ preserves the subspace $V_i$, and this induces a $k$-isomorphism $G_i \cong \GL(V_i)\cong \GL_{2(n+1-i) +1}$.
For any $h\in G_i(k)$, we denote by $ f_h \in k[\lambda]$ the characteristic polynomial of $h$ acting on $V_i$, of degree $2(n+1-i) +1$. When $h \in G_i (\BF_q)$, we know that $f_h$ is self-reciprocal in $\BF_{q^2} [\lambda]$. Similarly, for any $h \in H_i(\BF_q) = \GL_{\BF_{q^2}} (\mathbb W_{i-1})$, we denote by $f_h (\lambda) \in \BF_{q^2}[\lambda]$ the characteristic polynomial of $h$ acting on $\mathbb W_{i-1}$, which has degree $i-1$.

We fix $1\leq i \leq n+1$. Write $n' $ for $n+1 -i$. Thus $\mathbb G_i = \UU_{2n'+1}$. It is easy to show that in $G_i$ there is a $\s$-stable maximal torus $T_i'$ of type $1$, with coordinates
$$ (k^{\times})^{\oplus 2n'+ 1} \isom T_i',  \qquad 	  
(\lambda_1,\cdots,\lambda_{2n'+1} ) \mapsto \gamma ' (\lambda_1,\cdots,\lambda_{2n'+1}), $$ satisfying the following conditions: 
\begin{itemize}
	\item The eigenvalues (with multiplicities) of $\gamma'(\lambda_1,\cdots, \lambda _{2n'+1})$ acting on $V_i$ are $\lambda_1,\cdots, \lambda_{2n'+1}.$
	\item The action of $\s$ on $T_i'$ sends $\gamma'(\lambda_1,\cdots, \lambda_{2n'+1})$ to $\gamma'(\lambda_{2n'+1}^{-q}, \cdots, \lambda_1^{-q})$.
	\item The action of $w_i$ on $T_i'$ sends $\gamma'(\lambda_1,\cdots, \lambda_{2n'+1})$ to $\gamma'(\lambda_{n'+1}, \lambda_1,  \cdots, \lambda_{n'}, \lambda_{n'+2}, \cdots, \lambda_{2n'+1})$. 
\end{itemize}
Then it easily follows that on $T_i$ we have coordinates 
$$ (k^{\times})^{\oplus 2n'+ 1} \isom T_i,  \qquad 	  
(\lambda_1,\cdots,\lambda_{2n'+1} ) \mapsto \gamma_0 (\lambda_1,\cdots,\lambda_{2n'+1}), $$ such that the eigenvalues (with multiplicities) of $\gamma_0 (\lambda_1,\cdots, \lambda _{2n'+1})$ acting on $V_i$ are $\lambda_1,\cdots, \lambda_{2n'+1},$ and such that 
$$\gamma_0 (\lambda_1,\cdots, \lambda _{2n'+1}) ^{\s} = \gamma_0(\lambda_{n'+1}^{-q}, \lambda_{2n'+1}^{-q}, \cdots, \lambda_{n'+2}^{-q}, \lambda_{n'}^{-q},\cdots, \lambda_1^{-q}). $$
We define new coordinates on $T_i$ $$ (k^{\times})^{\oplus 2n'+ 1} \isom T_i,  \qquad 	  
(\lambda_1,\cdots,\lambda_{2n'+1} ) \mapsto \gamma (\lambda_1,\cdots,\lambda_{2n'+1}), $$ by setting $$\gamma(\lambda_1,\cdots,\lambda_{2n'+1}) := \gamma_0(\lambda_1, \cdots, \lambda_{n'+1}, \lambda_{2n'+1}, \cdots, \lambda_{n'+2}).$$ Then we have 
\begin{align}\label{eq:Galois action, unitary original}
\gamma (\lambda_1,\cdots,\lambda_{2n'+1}) ^{\sigma} = \gamma ( \lambda_{n'+1}  ^{-q},  \cdots, \lambda _{2n'+1}^{-q}, \lambda _1 ^{-q}, \cdots, \lambda _{n'}^{-q}).
\end{align}
In particular, we have 
\begin{align}\label{eq:Galois action, unitary}
\gamma (\lambda_1,\cdots,\lambda_{2n'+1}) ^{\sigma^2} = \gamma ( \lambda _{2n'+1}^{\s^2}, \lambda _1 ^{\s^2}, \cdots, \lambda _{2n'}^{\s^2}).
\end{align}
\begin{theorem}\label{thm:about Ti, unitary} 
	We have the following statements about $T_i (\BF_q)$.
	\begin{enumerate}
		\item If $\gamma \in T_i(\BF_q)$, then $f_{\gamma}=Q ^m$ for some $Q \in \mathsf{SR}_2$, and some positive integer $m$.
		\item Let $Q \in \mathsf{SR}_2$. Assume $ m$ is an integer such that $m \deg Q = 2n'+1$. Then there exists $\gamma \in T_i (\BF_q)$ with $f_{\gamma}= Q ^m$. 
		\item Let $Q$ and $m$ be as in part (2). Let $\gamma \in G_i(\BF_q)$ be a semi-simple element such that $f_{\gamma} = Q^m$. Then $\gamma$ is $G_i(\BF_q)$-conjugate to an element of $T_i (\BF_q)$. 
		\item Let $\gamma \in T_i (\BF_q)$. Write $f_{\gamma} = Q^m$ as in part (1). Then $\TT(w_i,\gamma)=\deg Q$.  
	\end{enumerate}
\end{theorem}
\begin{proof} 
	
	\textbf{(1)} Write $\gamma = \gamma(\lambda_1,\cdots,\lambda_{2n'+1})$. Since $\gamma^{\sigma} = \gamma$, it follows from (\ref{eq:Galois action, unitary}) that all eigenvalues of $\gamma$ are in one $\s^2$-orbit. Hence $f_{\gamma}$ has a unique monic irreducible factor $Q \in \BF_{q^2}[\lambda]$. Since $f_{\gamma}$ is self-reciprocal, so is $Q$.
	
	\textbf{(2)} Let $d = \deg Q$. Then $d$ is odd by hypothesis. Let $\Lambda = (\lambda_1,\cdots, \lambda_{d})$ be an admissible enumeration of the roots of $Q$, in the sense of Definition \ref{def:odd adm}. Then
	$
	\gamma : = \gamma (\Lambda,\cdots, \Lambda)$
	(with $m$ appearances of $\Lambda$) is an element of $T_i (k)$. We now show that $\gamma\in T_i (\BF_q)$. 
	
	If $d =1$, then $\lambda_1 ^{-q} = \lambda_1$, and it is clear that $\gamma \in T_i (\BF_q)$ by (\ref{eq:Galois action, unitary original}). Now assume $d \geq 3$. By (\ref{eq:Galois action, unitary original}), we need only show that 
	$\lambda_{\alpha} ^{-q} = \lambda_{\alpha +n'+1}$, where the subscripts are in $\ZZ/d\ZZ$, for all $\alpha \in \ZZ / d \ZZ$. By Lemma \ref{lem:odd deg} (2), it suffices to show that $n'+1 \equiv  (d+1)/2 \mod d$. Since $d$ is odd, the last congruence is equivalent to $2n'+2 \equiv d+1 \mod d$. But the last congruence is true because $2n'+1 = md$. We have proved that $\gamma\in T_i (\BF_q)$. By construction, $f_{\gamma} = Q^m$. Part (2) is proved. 
	
	\textbf{(3)} Firstly, as $G_i$ is isomorphic to $\GL(V_i) = \GL_{2n'+1}$ over $k$, we know that two semi-simple elements in $G_i(k)$ are conjugate if and only if they have the same characteristic polynomial. Secondly, since $G_i$ has simply connected derived subgroup, by the Lang--Steinberg theorem we know that any two semi-simple elements in $G_i(\BF_q)$ are $G_i(\BF_q)$-conjugate if and only if they are $G_i(k)$-conjugate (cf.~\S \ref{subsubsec:conn setting} and the proof of Lemma \ref{lem:compute TT}). The assertion now follows from part (2). 
	
	\textbf{(4)} Let $d = \deg Q$. Since $G_i$ has simply connected derived subgroup, we may use Lemma \ref{lem:compute TT} to compute $\TT (w_i, \gamma)$. We have $$ \TT(w_i,\gamma)  = \# \{ \gamma' \in T_i (\BF_q) ; \gamma ' = {}^x \gamma \text{ for some } x\in W_i \}. $$
	By (\ref{eq:Galois action, unitary}), it is clear that any $\gamma' \in T_i (\BF_q)$ with characteristic polynomial $Q^m$ must be of the form $ \gamma ' = \gamma (\Lambda ', \cdots, \Lambda ')$, for some admissible enumeration $\Lambda'$ of the $d$ roots of $Q$. There are $d$ such admissible enumerations (Lemma \ref{lem:odd deg}), and all of them correspond to elements in $T_i (\BF_q)$ by the proof of part (2). Moreover, it is clear that these $d$ resulting elements of $T_i(\BF_q)$ are in the same $W_i$-orbit. Hence $\TT(w_i,\gamma) = d$. 
\end{proof}

\begin{lemma}\label{lem:counting flags, unitary} Let $g\in G(\BF_q) \cap G^{\mathrm{reg}}$. For each $1\leq i \leq n+1$, let $\mathcal M_i^g$ be as in \S \ref{subsec:combining results}. We have a bijection 
	\begin{align*}
	\mathcal M_i ^g  & \isom  \set{ U\in \BF_{q^2} [\lambda] ^{\mathrm{monic}} ; \deg U = i-1, UU^* \text{ divides } f_g \text{ in }\BF_{q^2} [\lambda]} \\ 
	r P_i(\BF_q) & \mapsto f_{\pi_i' (r\i g r)} .
	\end{align*}
\end{lemma}
\begin{proof} The proof is completely analogous to Lemma \ref{lem:counting flags}, based on the fact that all $(i-1)$-dimensional totally isotropic $\BF_{q^2}$-subspaces of $\mathbb V$ are in the same $G(\BF_q)$-orbit. \end{proof}
\begin{theorem}\label{thm:unitary} Let $g \in G(\BF_q) \cap G^{\mathrm{reg}}$. We use the notations in Definition \ref{defn:MMM}. For each $Q \in \mathsf{SR}_2$, we simply write $m_Q$ for $m_Q(f_g)$. The following statements hold.
	\begin{enumerate}
		\item If $\tr (g,J ,\mathscr L) \neq 0$, then there is a unique element $Q_0 \in \mathsf {SR}_2 $ such that $m_{Q_0}$ is odd.
		\item Assume there is a unique element $Q_0 \in \mathsf {SR}_2$ such that $m_{Q_0}$ is odd. Then 
		$$\tr (g, J,\mathscr L ) = \deg Q_0 \frac{m_{Q_0}+1}{2}\MMM_2(f_g).$$
	\end{enumerate}
\end{theorem}
\begin{proof}
	We apply Theorem \ref{thm:char formula} to compute $\tr (g, J ,\mathscr L)$ in the following.
	
	We prove part (1). Assume $\tr(g, J,\mathscr L) \neq 0$. Then there exist $1 \leq i \leq n+1$ and $\gamma \in \Gamma_i$ such that  $\mathcal M_i^{g,\gamma} \neq \emptyset$. Take $rP_i (\BF_q) \in \mathcal M_i^{g,\gamma}$. Then by Theorem \ref{thm:about Ti, unitary} (1), we have $f_{\pi_i (r\i g r)}= Q ^{m}$, for some $Q \in \mathsf {SR}_2$ and some positive integer $m$. In particular $m$ is odd because $Q^m$ has odd degree. By Lemma \ref{lem:counting flags, unitary}, we have $f_g = Q ^m UU^*$ for some $U \in \BF_{q^2} [\lambda]^{\mathrm{monic}}$. It then follows that $Q$ is the unique element of $\mathsf {SR}_2$ such that $m_Q$ is odd. Part (1) is proved.
	
	We now prove part (2). Since $f_g$ has odd degree, it immediately follows from the hypothesis that $\deg Q_0$ is odd.  
	Define $$\mathscr I : =\{i ; 1\leq i \leq n+1, \frac{2(n+1-i) +1 }{\deg Q_0} \text{ is a (necessarily odd) integer }\leq m_{Q_0} \}. $$
	For $i \in \mathscr I$, define $m_i : = [2(n+1-i) +1]/\deg Q_0.$ Note that $i\mapsto m_i$ is a bijection $\mathscr I \to \{1,3,5,\cdots,m_{Q_0} \}.$
	In particular $|\mathscr I | = (m_{Q_0}+1)/2$. Similar to the bijection (\ref{eq:bij}), we obtain a bijection 
	\begin{align}\label{eq:bij, unitary}
	&	\set {(i,\gamma, r P_i(\BF_q)) ; 1\leq i \leq n+1, \gamma \in \Gamma_i, r P_i(\BF_q) \in \mathcal M_i^{g,\gamma}}  \isom  \\ & \nonumber \set {(i, rP_i(\BF_q)) ; i \in \mathscr I,  r P_i(\BF_q)\in \mathcal M_i^g, f_{\pi_i (r\i g r )} = Q_0^{m_i} } \\ & \nonumber (i,\gamma, rP_i(\BF_q))   \mapsto (i,rP_i(\BF_q))
	\end{align}
	based on Theorem \ref{thm:about Ti, unitary} (3). We also note that if $(i,\gamma, rP_i(\BF_q))$ is in the left hand side of (\ref{eq:bij, unitary}), then $f_{\gamma} = Q_0 ^{m_i}$, and so by Theorem \ref{thm:about Ti, unitary} (4) we have 
	\begin{align}\label{eq:d_0, unitary}
	\TT(w_i,\gamma) = \deg Q_0.
	\end{align} The rest of the proof is identical to the proof of Theorem \ref{thm:RZO case} (3), based on (\ref{eq:bij, unitary}), (\ref{eq:d_0, unitary}), and Lemma \ref{lem:counting flags, unitary}. \end{proof}
\section{Application to arithmetic intersection}\label{sec:application}
In this section we apply Theorem \ref{thm:unitary} to prove the arithmetic fundamental lemma in the minuscule case, generalizing the main result of \cite{RTZ} and \cite{AFL}. We also apply Theorem \ref{thm:RZO case} to compute certain arithmetic intersection in GSpin Rapoport--Zink spaces, generalizing the main result of \cite{RZO}.  
\subsection{The arithmetic fundamental lemma in the minuscule case}\label{subsec:AFL}
We follow the notation of \cite{RTZ} and \cite{AFL}. Let $p$ be an odd prime. Let $F$ be a finite extension of $\mathbb{Q}_p$ with residue field $\mathbb{F}_q$ and a uniformizer $\pi$. As usual we denote $k: = \overline{ \BF}_q$. Let $E/F$ be a quadratic unramified extension. Let $\breve E$ be the completion of the maximal unramified extension of $E$. Let $S =\Spf \mathcal{O}_{\breve E}$. Fix an integer $n\geq 2$. Let $\RZ_{n}$ be the \emph{unitary Rapoport--Zink space of signature $(1,n-1)$}, which is a formal scheme over $S$ parameterizing deformations up to quasi-isogeny of height $0$ of unitary $\mathcal O_F$-modules of signature $(1,n-1)$. For details on $\RZ_n$ see \cite{Kudla2011}, \cite{Mih}, and \cite{Cho2018}. 

Let $C_{n}$ be a non-split Hermitian space of dimension $n$, for the quadratic extension $E/F$. Here non-split means that the discriminant has odd valuation. We identify $C_n$ with the space of special quasi-homomorphisms for the framing object in the moduli problem of $\RZ_n$, see \cite{Kudla2011} for $F=\QQ_p$ (cf.~\cite[\S 2.2, \S 2.3]{AFL}), and \cite{Cho2018} for general $F$. Similarly, we form $\RZ_{n-1}$ and $C_{n-1}$. We identify $C_{n-1}$ with the orthogonal complement in $C_n$ of a fixed vector $u\in C_n$ of norm $1$, thus $C_n = C_{n-1}\oplus  E u$. We have a natural closed immersion $$\delta : \RZ_{n-1} \hookrightarrow \RZ_n.$$ In fact $\delta$ identifies $\RZ_{n-1}$ with the special divisor in $\RZ_n$ associated to $u$, see \cite{Kudla2011} for $F=\QQ_p$, and see \cite{Cho2018} for general $F$. 

The unitary group $J(F):  = \UU (C_n)(F)$ acts on $\RZ_n$. Let $g \in J(F)$. Define 
$$L(g): = \mathcal O_E \cdot u + \mathcal O_E \cdot gu + \cdots +\mathcal O_E \cdot g^{n-1} u \subset  C_n. $$ Throughout we make two assumptions on $g$. Firstly, we assume that $g$ is \emph{regular semi-simple minuscule}, in the sense that $L(g)$ is a full-rank $\mathcal O_E$-lattice in $C_n$ satisfying 
$$\pi L(g) ^{\vee} \subset L(g) \subset L(g) ^{\vee}. $$ Secondly, we assume that $g$ has non-empty fixed points in $\RZ_n(k)$. By \cite[\S 5]{RTZ}, our second assumption implies that both $L(g)$ and $L(g)^{\vee}$ are stable under $g$.

Define $\mathbb V : = L(g) ^{\vee}/L(g)$. This is an odd-dimensional vector space over $\BF_{q^2}$, with a natural structure of a Hermitian space, see \cite[\S 2.4]{AFL}. Let $\mathcal V: = \mathcal V(L(g)^\vee)$ be the smooth projective generalized Deligne--Lusztig variety associated to the vertex lattice $L(g)^\vee$ as in \cite{Vollaard2010} and \cite{Vollaard2011}. (These references assume $F =\QQ_p$, but see \cite{Cho2018} for general $F$.) The finite group $\UU (\mathbb V)(\BF_q)$ naturally acts on $\mathcal V$. 	Let $\mathbb G = \UU(\mathbb V)$, $G=\mathbb G_k$, and let $(J,\mathscr L)$ be the $\s$-unbranched datum for $\mathbb G$ specified in \S \ref{subsec:odd unitary}.
\begin{lemma}\label{lem:identify V}
	The variety $\mathcal V$ is $\mathbb G(\BF_q)$-equivariantly isomorphic to $\overline{X_{J, w_1}}$. 
\end{lemma}
\begin{proof}
	Since $ G_1= P_1=G$, by Proposition \ref{prop:inductive} we have an isomorphism 
	\begin{align*}
	X_{w_1}& \isom X_{J,w_1} \subset G/P_J \\  gB & \mapsto gP_J,
	\end{align*}
	where $X_{w_1}$ is the classical Deligne--Lusztig variety associated to $w_1$ in the full flag variety $G/B$. The lemma then follows from \cite[Theorem 2.15]{Vollaard2010}, which asserts that $\mathcal V$ is also the closure in $G/P_J$ of the image of $X_{w_1}$. (Again, the reference \cite{Vollaard2010} assumes $F=\QQ_p$ and $\BF_q = \BF_p$, but the result \cite[Theorem 2.15]{Vollaard2010} easily generalizes.)
\end{proof}
The action of $g$ on $\mathbb V$ defines an element $\bar g\in \mathbb G (\BF_q)$. We also know that $\bar g$ is regular, because $\mathbb V$ is a cyclic $\BF_{q^2}[\bar g]$-module. Let $f = f_{\bar g}\in \BF_{q^2}[\lambda]$ be the characteristic polynomial of $\bar g$. Thus $f$ is self-reciprocal. We use the notations in Definition \ref{defn:MMM}.

\begin{theorem}\label{thm:app to AFL} As before, assume $g\in J(F)$ is  regular semi-simple minuscule, such that $\RZ_{n}^g\neq \emptyset$. The following statements hold.
	\begin{enumerate}
		\item The formal scheme $\delta (\RZ_{n-1}) \cap \RZ_n ^{g}$ over $S$ is a $k$-scheme.
		\item The $k$-scheme $\delta (\RZ_{n-1}) \cap \RZ_n ^{g}$ is non-empty if and only if there is a unique element $Q_0 \in \mathsf{SR}_2$ with $m_{Q_0}(f)$ odd. In this case, $\delta (\RZ_{n-1}) \cap \RZ_n ^{g}$ has finitely many $k$-points, and is in particular Artinian, and moreover $\Int(g)$ is equal to the total $k$-length of $\delta (\RZ_{n-1}) \cap \RZ_n ^{g}$.
		\item Assume there is a unique element $Q_0 \in \mathsf{SR}_2$ with $m_{Q_0}(f)$ odd. Then the total $k$-length of $\delta (\RZ_{n-1})\cap \RZ_n ^{g}$ is equal to 
		$$ \deg (Q_0) \frac{m_{Q_0}(f)+1}{2} \MMM_2 (f).$$
	\end{enumerate}
\end{theorem} 
\begin{proof} We temporarily assume that $F=\QQ_p$.
	Then part (1) follows from \cite[Proposition 4.1.2]{AFL} (cf.~\cite[(9.6), Theorem 9.4]{RTZ}). Part (2) is proved in \cite[Proposition 8.1 (i)]{RTZ} and \cite[Proposition 4.2 (iii)]{RTZ}.
	
	For part (3), we first apply \cite[Proposition 4.1.2]{AFL} to identify $ \delta (\RZ_{n-1}) \cap \RZ_n ^{g}$ with $\mathcal V^{\bar g}$, the scheme theoretic fixed points of $\mathcal V$ under $\bar g\in \mathbb G (\BF_q)$.
	By part (2), $\mathcal V^{\bar g}$ is an Artinian scheme. Since $\mathcal V$ is smooth over $k$ and since $\mathcal V^{\bar g}$ is Artinian, it is well known (see for instance \cite[p.~111]{SerreLA}) that the intersection multiplicities of the graph of identity and the graph of $\bar g$ in $\mathcal V \times_k \mathcal V$ are simply given by the lengths of the local rings of $\mathcal V^{\bar g}$, as the higher Tor terms vanish. 
	It then follows from the Lefschetz fixed point formula \cite[Corollaire 3.7]{CYCLE} that the $k$-length of $\mathcal V^{\bar g}$ is equal to $\tr(\bar g, \coh^*(\mathcal V))$.
	By Lemma \ref{lem:identify V}, the last number is equal to $\tr(\bar g, J, \mathscr L)$. Hence part (3) follows from Theorem \ref{thm:unitary} and the fact that $\bar g$ is regular.
	We have proved the theorem assuming $F=\QQ_p$.
	
	We now explain the proof when $F$ is an arbitrary finite extension of $\QQ_p$. In fact, the reason that the references \cite{RTZ} and \cite{AFL} assumed $F=\QQ_p$ was because two ingredients needed in the arguments depended on this assumption. The first is the theory of special cycles considered in \cite{Kudla2011}, and the second is the Bruhat--Tits stratification of the reduced subscheme of $\RZ_n$ into generalized Deligne--Lusztig varieties, worked out in \cite{Vollaard2010} and \cite{Vollaard2011}. Both of these ingredients have now been generalized to arbitrary $F$ in \cite{Cho2018}. Based on this, all the previous arguments carry over.\footnote{It should be pointed out that in \cite[\S 2.6]{AFL}, for a vertex lattice $\Lambda$ the notation $\RZ_{\Lambda}$ denotes the special cycle in $\RZ_n$ associated to $\Lambda ^{\vee}$. Thus a priori $\RZ_{\Lambda}$ is a formal scheme over $S$, but it is a theorem (\cite[Theorems 9.4, 10.1]{RTZ}, see also \cite[Corollary 3.2.3]{AFL}) that $\RZ_{\Lambda}$ is in fact a reduced scheme over $k$. This result plays a key role in \cite{RTZ} and \cite{AFL}, and its proof depends on Grothendieck--Messing theory. In contrast, in \cite{Vollaard2011} and \cite{Cho2018} the notation $\RZ_{\Lambda}$ is by definition a scheme over characteristic $p$. Thus the two notations agree only a posteriori.}
\end{proof}
\begin{remark}
	Theorem \ref{thm:app to AFL} (3) was previously proved in \cite{RTZ} and \cite{AFL}, under the assumption that $F= \QQ_p$ with $p> (m_{Q_0}+1)/2$. This assumption is removed in Theorem \ref{thm:app to AFL}. On the other hand, under the same assumption on $p$ the papers \cite{RTZ} and \cite{AFL} determine each local ring of $\delta (\RZ_{n-1}) \cap \RZ^{g}_n$. This is a result not revealed by the methods of the current paper. 
\end{remark}
\begin{corollary}\label{cor:AFL}
	The minuscule case of the arithmetic fundamental lemma conjecture \cite[Conjecture 7.4]{RTZ} (cf.~\cite[\S 1]{Rapoport2017}) holds.
\end{corollary}
\begin{proof} This follows from the formula for the arithmetic intersection number $\Int(g)$ in Theorem \ref{thm:app to AFL} (2--3) and the explicit computation of the analytic side in \cite[Proposition 8.2]{RTZ}.
\end{proof}

\subsection{Arithmetic intersection on GSpin Rapoport--Zink spaces.}
We follow the notation of \cite{RZO}. Let $p$ be an odd prime, and fix an integer $n\geq 4$. Let $\RZO$ (resp.~$\RZO^\flat$) be the $\GSpin $ Rapoport--Zink space associated to a self-dual quadratic $\ZZ_p$-lattice of rank $n$ (resp.~$n-1$). We have a natural closed immersion $$\delta : \RZO^{\flat} \to \RZO $$
of formal schemes over $\Spf W(k)$. These are specific Hodge-type Rapoport--Zink spaces introduced by Howard--Pappas \cite{Howard2015}. Associated to the precise data used to define $\RZO^{\flat}$ and $\RZO$, we have a pair of quadratic spaces $V_K^{\flat,\Phi}$ and $V_K^{\Phi}$ over $\QQ_p$, and $V_K^{\flat,\Phi}$ can be identified with the orthogonal complement in $V_K^{\Phi}$ of a fixed vector $x_n\in  V_K^{\Phi}$ whose norm is $1$. (The triple $(V_K^{\flat,\Phi},V_K^{\Phi},x_n)$ is analogous to the triple $(C_{n-1}, C_n, u)$ in \S \ref{subsec:AFL}.) 

The group $J(\QQ_p) = \GSpin(V_K^{\Phi})(\QQ_p) $ acts on $\RZO$. As in \cite[\S 4.3]{Howard2015}, $\RZO$ is the disjoint union of open and closed formal subschemes $\RZO^{(l)}$, indexed by $l\in\ZZ$. The action of any $g\in J(\QQ_p)$ on $\RZO$ maps each $\RZO ^{(l)}$ isomorphically to $\RZO^{(l+l_g)}$, where $l_g$ is the $p$-adic valuation of the spinor norm of $g$ in $\QQ_p^\times$. We view $p$ as an element of $J(\QQ_p)$ by viewing it as an scalar in the $\GSpin$ group. Thus $p$ maps each $\RZO^{(l)}$ isomorphically to $\RZO^{(l+2)}$.

Let $g\in J(\QQ_p)$. Define 
$$L(g):= \ZZ_p \cdot x_n + \ZZ_p \cdot g x_n +\cdots \ZZ_p \cdot g^{n-1} x_n \subset V_{K}^{\Phi}.$$ Here $g$ acts on $V_K^{\Phi}$ via the natural map $\GSpin(V_K^{\Phi}) \to \SO(V_K ^{\Phi})$. Throughout we make two assumptions on $g$. Firstly, we assume that $g$ is \emph{regular semi-simple minuscule}, in the sense that $L(g)$ is a full-rank $\ZZ_p$-lattice in $V_K^{\Phi}$ satisfying 
$$p L(g) ^{\vee} \subset L(g) \subset L(g) ^{\vee}. $$ Secondly, we assume that $g$ has non-empty fixed points in $\RZO(k)$. By \cite[\S 3.6]{RZO}, our second assumption implies that both $L(g)$ and $L(g)^{\vee}$ are stable under $g$. It also directly follows from our second assumption that $l_g=0$. In particular $g$ stabilizes each $\RZO^{(l)}$.

Define $\mathbb V : = L(g) ^{\vee}/L(g)$. This is an even-dimensional, non-zero vector space over $\BF_{p}$, with a natural structure of a non-split quadratic space, see \cite[\S 2.7]{RZO}. Let $S = S_{L(g)^\vee}$ be the smooth projective $k$-variety associated to the vertex lattice $L(g)^\vee$ as in \cite[\S 5.3]{Howard2015}. The finite group $\mathrm O (\mathbb V)(\BF_p)$ naturally acts on $S$.
By \cite[Proposition 5.3.2]{Howard2015} and its proof, we know that $S$ has two connected components $S^+, S^-$, that the action of $\SO(\mathbb V)(\BF_p)$ on $S$ stabilizes each of $S^+, S^-$, and that any element of $\mathrm{O}(\mathbb V) (\BF_p)-\SO (\mathbb V)(\BF_p)$ interchanges $S^+,S^-$. Let $\mathbb G = \SO(\mathbb V)$, $G=\mathbb G_k$, and let $(J,\mathscr L)$ be the $\s$-unbranched datum for $\mathbb G$ specified in \S \ref{subsec:even ortho}. For definiteness, we fix the convention so that our $w_1$ corresponds to the Weyl group element $w^-$ in \cite[\S 3.2]{HPGU22}.\footnote{This is harmless because up to outer automorphism of $G$, our $w_1$ corresponds to either $w^-$ or $w^+$ in \cite[\S 3.2]{HPGU22}. All the arguments below are the same in the two cases.} 

\begin{lemma}\label{lem:identify S} The variety $S^-$ is $\mathbb G(\BF_q)$-equivariantly isomorphic to $\overline{X_{J, w_1}}$.
\end{lemma}
\begin{proof} Since $ G_1= P_1=G$, by Proposition \ref{prop:inductive} we have an isomorphism 
	\begin{align*}
	X_{w_1}& \isom X_{J,w_1} \subset G/P_J \\ gB & \mapsto gP_J,
	\end{align*}
	where $X_{w_1}$ is the classical Deligne--Lusztig variety associated to $w_1$ in the full flag variety $G/B$. The claim then follows from \cite[Proposition 3.8]{HPGU22}, which asserts that $S^-$ (denoted by $\mathscr X^-$ in \textit{loc.~cit.}) is also the closure of the image of $X_{w_1}$ in $G/P_J$. 
\end{proof}

The action of $g$ on $\mathbb V$ defines an element $\bar g\in \mathrm{O} (\mathbb V)(\BF_p)$. The following result is implicitly assumed in \cite{RZO}, but is not explicitly stated and proved there. We give two proofs here, for the sake of completeness.
\begin{lemma}\label{lem:in SO}
	The element $\bar g \in \mathrm{O} (\mathbb V)(\BF_p)$ lies in $\SO (\mathbb V) (\BF_p)$.
\end{lemma}
\begin{proof}
	\textbf{First proof.} Let $S = S_{ L(g)^{\vee}}$ be as before. By \cite[Theorem 6.3.1]{Howard2015}, we have an isomorphism $p^\ZZ\backslash \RZO_{L(g) ^{\vee}}^{\mathrm{red}} \isom  S$, where $\RZO_{L(g) ^{\vee}}^{\mathrm{red}}$ is a certain $g$-stable subscheme of $\RZO$. It is easy to see that this isomorphism intertwines the action of $g$ on the left and the action of $\bar g$ on the right, for example by checking the statement on $k $-points. Since $g$ stabilizes each $\RZO^{(l)}$, by \cite[Corollary 6.3.2]{Howard2015} we know that $g$ stabilizes each of the two connected components of $p^{\ZZ}\backslash \RZO_{\Lambda}^{\mathrm{red}}$. Therefore $\bar g$ stabilizes each of the two connected components of $S$. By the proof of \cite[Proposition 5.3.2]{Howard2015}, any element of $\mathrm{O}(\mathbb V) (\BF_p)-\SO (\mathbb V)(\BF_p)$ interchanges the two connected components of $S$. It then follows that $\bar g \in \SO (\mathbb V)$.
	
	\textbf{Second proof.} The result follows from Lemma \ref{lem:spinor norm} in the following, applied to $W:= V_K^{\Phi}$, $L:= L(g)$, and $h:= $ the image of $g$ under $\GSpin (V_K^{\Phi})(\QQ_p)\to \SO (V_K^{\Phi})(\QQ_p)$. The hypothesis on the spinor norm of $h$ is satisfied because $l_g=0$.
\end{proof}

\begin{lemma}\label{lem:spinor norm}Let $(W, [\cdot, \cdot])$ be a quadratic space over $\QQ_p$. Let $h \in \mathrm O(W)(\QQ_p)$ be an element whose spinor norm (see \cite[\S 1.6]{kitaoka}) in $\QQ_p^{\times}/\QQ_p^{\times,2}$ has even valuation. Let $L$ be a full-rank lattice in $W$ satisfying $pL^\vee \subset L \subset L^\vee$. Assume $L$ is stable under $h$. Then the induced action $\bar h$ of $h$ on the $\BF_p$-vector space $L^{\vee}/L$ has determinant $1$.
\end{lemma}
\begin{proof}
	Since $h$ stabilizes $L$, by \cite[Theorem 5.3.3]{kitaoka} we have $h=\tau_1\cdots \tau_m$, where each $\tau_j \in \mathrm{O}(W)(\QQ_p)$ is the reflection associated to an anisotropic vector $v_j\in L$ (namely $\tau_j(x)= x- 2[x,v_j][v_j,v_j]^{-1}v_j, \forall x\in W$), such that $\tau_j$ also stabilizes $L$. By rescaling, we may and shall assume that each $v_j\in L-pL$. We now fix $1\leq j\leq m$. 
	
	Since $\tau_j$ stabilizes $L$, we have $[x,v_j]\in [v_j,v_j]\ZZ_p$ for all $x\in L,$ or equivalently that \begin{align}\label{eq:vj}
	v_j\in [v_j,v_j]L^{\vee}.
	\end{align} Since $pL^\vee \subset L\subset L^{\vee}$ and $v_j\in L-pL$, it follows from (\ref{eq:vj}) that $[v_j,v_j]$ has valuation $0$ or $1$. If $[v_j,v_j]$ has valuation $0$, then $\tau_j$ maps each $x\in L^{\vee}$ into $x +\ZZ_p v_j \subset x+ L$, and so the image of $\tau_j$ in $\GL(L^\vee/L)$ is trivial. Assume $[v_j,v_j]$ has valuation $1$. Then $v_j\in pL^{\vee}$ by (\ref{eq:vj}), and so $v_j=pw_j$ for some $w_j\in L^\vee-L$. In this case we have \begin{align}
	\label{eq:vj val 1} \tau_j(x)=x-2\frac{p[x,w_j]}{p[w_j, w_j]} w_j ,\quad \forall x\in L.
	\end{align} Now the map 
	\begin{align*}
	L^\vee\times L^\vee & \to \BF_p \\ (x,y) & \mapsto p[x,y]\mod p\end{align*}
	is well defined and descends to a non-degenerate bi-linear pairing on the $\BF_p$-vector space $L^{\vee}/L$ (cf.~\cite[\S 5.3.1]{Howard2015}). Noting that $p[w_j, w_j]=p\i [v_j,v_j]$ is by assumption in $\ZZ_p^{\times}$, we see from (\ref{eq:vj val 1}) that the image of $\tau_j$ in $\GL(L^\vee/L)$ is given by the reflection associated to an anisotropic vector in $L^\vee/L$, namely the image of $w_j$. 
	
	In conclusion, the image of $h$ in $\GL(L^\vee/L)$ is the product of $m'$ reflections, where $m'$ is the number of the $v_j$'s such that $[v_j,v_j]\in p \ZZ_p^{\times}$, whereas the $m-m'$ other $v_j$'s satisfy $[v_j,v_j]\in \ZZ_p^{\times}$. Since the spinor norm of $h$ has even valuation, we know that $m'$ is even. The lemma follows.
\end{proof}

By Lemma \ref{lem:in SO} we have $\bar g \in \SO(\mathbb{V})(\BF_p)$. We also know that the image of $\bar g$ in $\GL(\mathbb V)$ is regular, because $\mathbb V$ is a cyclic $\BF_{p}[\bar g]$-module. Let $f = f_{\bar g}\in  \BF_{p}[\lambda]$ be the characteristic polynomial of $\bar g$. Thus $f$ is self-reciprocal. We use the notations in Definition \ref{defn:MMM}.
\begin{theorem}\label{thm:app to RZO} As before, assume $g\in J(\QQ_p)$ is regular semi-simple minuscule, such that $\RZO^g\neq \emptyset$. The following statements hold.
	\begin{enumerate}
		\item The formal scheme $\delta (\RZO^{\flat}) \cap \RZO ^{g}$ over $\Spf W(k)$ is a $k$-scheme.
		\item The $k$-scheme $\delta (\RZO^{\flat}) \cap \RZO ^{g}$ is non-empty if and only if there is a unique element $Q_0 \in \mathsf{SR}$ with $m_{Q_0}(f)$ odd. Moreover, when this is the case $p^{\ZZ} \backslash (\delta (\RZO^{\flat}) \cap \RZO ^{g})$ has finitely many $k$-points, and is in particular Artinian.
		\item Assume there is a unique element $Q_0 \in \mathsf{SR}$ with $m_{Q_0}(f)$ odd. Then the total $k$-length of $p^{\ZZ} \backslash (\delta (\RZO^{\flat})\cap \RZO ^{g})$ is equal to 
		$$ \deg (Q_0) \frac{m_{Q_0}(f)+1}{2}\MMM(f).$$
	\end{enumerate}
\end{theorem} 
\begin{proof} Part (1) follows from \cite[Corollary 5.1.2]{RZO}, and part (2) is proved in \cite[Theorem 3.6.4]{RZO}.
	
	For part (3), we first apply \cite{RZO} to identify $p^{\ZZ} \backslash (\delta (\RZO^{\flat})\cap \RZO ^{g}) $ with $S^{\bar g}$, the scheme theoretic fixed points of $S$ under $\bar g$. Since $\bar g$ is in $\SO(\mathbb V)(\BF_p)$ (Lemma \ref{lem:in SO}), it stabilizes $S^+$ and $S^-$. Hence $S^{\bar g} = (S^+)^{\bar g} \sqcup (S^-)^{\bar g}$. By the same arguments as in the proof of Theorem \ref{thm:app to AFL} (3), the $k$-length of $S^{\bar g}$ is equal to $\tr(\bar g, \coh^*(S)) = \tr (\bar g , \coh ^*(S^+)) + \tr (\bar g, \coh ^*(S^-))$.

	By Lemma \ref{lem:identify S} and by the fact that $\bar g$ is regular in $\GL(\mathbb V)$, we know that $\tr(\bar g, \coh^*(S^-))$ is given by the formula in Theorem \ref{thm:RZO case} (3). Fix $g_0\in \mathrm{O}(\mathbb V)(\BF_p)-\SO(\mathbb V)(\BF_p)$. Then under the natural action of $\mathrm{O} (\mathbb V) (\BF_q)$ on $S$, the element $g_0$ interchanges $S^+$ and $S^-$, by the proof of \cite[Proposition 5.3.2]{Howard2015}. Hence we have $\tr(\bar g, \coh^*(S^+)) = \tr(g_0\bar gg_0\i, \coh^*(S^-))$. Since the formula in Theorem \ref{thm:RZO case} (3) only depends on the characteristic polynomial, and since $\bar g$ and $g_0 \bar g g_0\i$ are elements of $\SO(\mathbb V)(\BF_p)$ which are both regular in $\GL(\mathbb V)$ and have the same characteristic polynomial, we have $\tr(\bar g, \coh^*(S^+))= \tr(\bar g, \coh^*(S^-))$. It follows that $\tr(\bar g, \coh^*(S))$ is equal to twice the formula in Theorem \ref{thm:RZO case} (3). The proof of part (3) is finished.
\end{proof}
\begin{remark}
	Theorem \ref{thm:app to RZO} (3) was previously proved in \cite{RZO}, under the assumption that $p>(m_{Q_0}+1)/2$. This assumption is removed in Theorem \ref{thm:app to RZO}. On the other hand, under the same assumption on $p$ the paper \cite{RZO} determines each local ring of $\delta (\RZO^{\flat}) \cap \RZO ^{g}$. This is a result not revealed by the methods of the current paper. 
\end{remark}
\begin{remark}
	We correct two mistakes in \cite{AFL} and \cite{RZO}. Firstly, in both the papers the definition of the reciprocal of a polynomial should be normalized so that the reciprocal is monic, as in \S \ref{subsec:rec}. This mistake does not affect the correctness of any of the proofs. Secondly, in \cite[Theorem A (2), Theorem 3.6.4]{RZO}, the product should be over pairs of non-self-reciprocal irreducible monic factors, as in Theorem \ref{thm:app to RZO} and Definition \ref{defn:MMM}, as opposed to over single non-self-reciprocal irreducible monic factors. To correct the proof of \cite[Theorem 3.6.4]{RZO}, one interprets the symbol $\prod _{R(T)\neq R^*(T)}$ in the proof as the product over such pairs $\set {R(T), R^*(T)}$ rather than over such $R(T)$'s.
\end{remark}

\bibliographystyle{hep}
\bibliography{myrefchar}

\begin{thebibliography}{{Cho}18}

\bibitem[Bou68]{bourbaki}
N.~Bourbaki,
\newblock \textsl{ \'El\'ements de math\'ematique. {F}asc. {XXXIV}. {G}roupes
  et alg\`ebres de {L}ie. {C}hapitre {IV}: {G}roupes de {C}oxeter et syst\`emes
  de {T}its. {C}hapitre {V}: {G}roupes engendr\'es par des r\'eflexions.
  {C}hapitre {VI}: syst\`emes de racines},
\newblock Actualit\'es Scientifiques et Industrielles, No. 1337, Hermann,
  Paris, 1968.

\bibitem[Car93]{Carter}
R.~W. Carter,
\newblock \textsl{ Finite groups of {L}ie type},
\newblock Wiley Classics Library, John Wiley \& Sons, Ltd., Chichester, 1993,
\newblock Conjugacy classes and complex characters, Reprint of the 1985
  original, A Wiley-Interscience Publication.

\bibitem[{Cho}18]{Cho2018}
S.~{Cho}, \textsl{ {The basic locus of the unitary Shimura variety with
  parahoric level structure, and special cycles}},
\newblock arXiv e-prints , arXiv:1807.09997 (Jul 2018), {1807.09997}.

\bibitem[DL76]{DL}
P.~Deligne and G.~Lusztig, \textsl{ Representations of reductive groups over
  finite fields},
\newblock Ann. of Math. (2) \textbf{ 103}(1), 103--161 (1976).

\bibitem[DM91]{DigneMichel}
F.~Digne and J.~Michel,
\newblock \textsl{ Representations of finite groups of {L}ie type}, volume~21
  of \textsl{ London Mathematical Society Student Texts},
\newblock Cambridge University Press, Cambridge, 1991.

\bibitem[GD77]{CYCLE}
A.~Grothendieck and P.~Deligne,
\newblock La classe de cohomologie associ\'{e}e \`a un cycle,
\newblock in \textsl{ Cohomologie \'{e}tale}, volume 569 of \textsl{ Lecture
  Notes in Math.}, pages 129--153, Springer, Berlin, 1977.

\bibitem[GGP12]{Gan2012}
W.~T. Gan, B.~H. Gross and D.~Prasad, \textsl{ Symplectic local root numbers,
  central critical {$L$} values, and restriction problems in the representation
  theory of classical groups},
\newblock Ast\'erisque (346), 1--109 (2012),
\newblock Sur les conjectures de Gross et Prasad. I.

\bibitem[GH15]{GH}
U.~G\"{o}rtz and X.~He, \textsl{ Basic loci of {C}oxeter type in {S}himura
  varieties},
\newblock Camb. J. Math. \textbf{ 3}(3), 323--353 (2015).

\bibitem[He07a]{He-piece}
X.~He, \textsl{ The {$G$}-stable pieces of the wonderful compactification},
\newblock Trans. Amer. Math. Soc. \textbf{ 359}(7), 3005--3024 (2007).

\bibitem[He07b]{He-min}
X.~He, \textsl{ Minimal length elements in some double cosets of {C}oxeter
  groups},
\newblock Adv. Math. \textbf{ 215}(2), 469--503 (2007).

\bibitem[He09]{He-par}
X.~He,
\newblock {$G$}-stable pieces and partial flag varieties,
\newblock in \textsl{ Representation theory}, volume 478 of \textsl{ Contemp.
  Math.}, pages 61--70, Amer. Math. Soc., Providence, RI, 2009.

\bibitem[HP14]{HPGU22}
B.~Howard and G.~Pappas, \textsl{ On the supersingular locus of the {${\rm
  GU}(2,2)$} {S}himura variety},
\newblock Algebra Number Theory \textbf{ 8}(7), 1659--1699 (2014).

\bibitem[HP17]{Howard2015}
B.~Howard and G.~Pappas, \textsl{ Rapoport--{Z}ink spaces for spinor groups},
\newblock Compos. Math. \textbf{ 153}(5), 1050--1118 (2017).

\bibitem[Kit93]{kitaoka}
Y.~Kitaoka,
\newblock \textsl{ Arithmetic of quadratic forms}, volume 106 of \textsl{
  Cambridge Tracts in Mathematics},
\newblock Cambridge University Press, Cambridge, 1993.

\bibitem[Kot82]{kottwitzrational}
R.~E. Kottwitz, \textsl{ Rational conjugacy classes in reductive groups},
\newblock Duke Math. J. \textbf{ 49}(4), 785--806 (1982).

\bibitem[KR11]{Kudla2011}
S.~Kudla and M.~Rapoport, \textsl{ Special cycles on unitary {S}himura
  varieties {I}. {U}nramified local theory},
\newblock Invent. Math. \textbf{ 184}(3), 629--682 (2011).

\bibitem[Liu18]{Liu2018}
Y.~Liu,
\newblock {Fourier-Jacobi cycles and arithmetic relative trace formula},
\newblock \url{https://gauss.math.yale.edu/~yl2269/FJcycle.pdf}, August 2018.

\bibitem[Lus07]{Lu07}
G.~Lusztig, \textsl{ A class of perverse sheaves on a partial flag manifold},
\newblock Represent. Theory \textbf{ 11}, 122--171 (2007).

\bibitem[LZ17]{AFL}
C.~Li and Y.~Zhu, \textsl{ Remarks on the arithmetic fundamental lemma},
\newblock Algebra Number Theory \textbf{ 11}(10), 2425--2445 (2017).

\bibitem[LZ18]{RZO}
C.~Li and Y.~Zhu, \textsl{ Arithmetic intersection on {GS}pin {R}apoport-{Z}ink
  spaces},
\newblock Compos. Math. \textbf{ 154}(7), 1407--1440 (2018).

\bibitem[{Mih}16]{Mih}
A.~{Mihatsch}, \textsl{ {Relative unitary RZ-spaces and the Arithmetic
  Fundamental Lemma}},
\newblock arXiv e-prints , arXiv:1611.06520 (Nov 2016), {1611.06520}.

\bibitem[PR94]{PR}
V.~Platonov and A.~Rapinchuk,
\newblock \textsl{ Algebraic groups and number theory}, volume 139 of \textsl{
  Pure and Applied Mathematics},
\newblock Academic Press, Inc., Boston, MA, 1994,
\newblock Translated from the 1991 Russian original by Rachel Rowen.

\bibitem[RSZ17a]{Rapoport2017}
M.~Rapoport, B.~Smithling and W.~Zhang, \textsl{ On the arithmetic transfer
  conjecture for exotic smooth formal moduli spaces},
\newblock Duke Math. J. \textbf{ 166}(12), 2183--2336 (2017).

\bibitem[RSZ17b]{Rapoport2017a}
M.~{Rapoport}, B.~{Smithling} and W.~{Zhang}, \textsl{ {Arithmetic diagonal
  cycles on unitary Shimura varieties}},
\newblock arXiv e-prints , arXiv:1710.06962 (Oct 2017), {1710.06962}.

\bibitem[RTZ13]{RTZ}
M.~Rapoport, U.~Terstiege and W.~Zhang, \textsl{ On the arithmetic fundamental
  lemma in the minuscule case},
\newblock Compos. Math. \textbf{ 149}(10), 1631--1666 (2013).

\bibitem[Ser00]{SerreLA}
J.-P. Serre,
\newblock \textsl{ Local algebra},
\newblock Springer Monographs in Mathematics, Springer-Verlag, Berlin, 2000,
\newblock Translated from the French by CheeWhye Chin and revised by the
  author.

\bibitem[Spr74]{Sp}
T.~A. Springer, \textsl{ Regular elements of finite reflection groups},
\newblock Invent. Math. \textbf{ 25}, 159--198 (1974).

\bibitem[Ste65]{Steinberg-reg}
R.~Steinberg, \textsl{ Regular elements of semisimple algebraic groups},
\newblock Inst. Hautes \'{E}tudes Sci. Publ. Math. (25), 49--80 (1965).

\bibitem[Ste68]{Steinberg-End}
R.~Steinberg,
\newblock \textsl{ Endomorphisms of linear algebraic groups},
\newblock Memoirs of the American Mathematical Society, No. 80, American
  Mathematical Society, Providence, R.I., 1968.

\bibitem[Vol10]{Vollaard2010}
I.~Vollaard, \textsl{ The supersingular locus of the {S}himura variety for
  {${\rm GU}(1,s)$}},
\newblock Canad. J. Math. \textbf{ 62}(3), 668--720 (2010).

\bibitem[VW11]{Vollaard2011}
I.~Vollaard and T.~Wedhorn, \textsl{ The supersingular locus of the {S}himura
  variety of {${\rm GU}(1,n-1)$} {II}},
\newblock Invent. Math. \textbf{ 184}(3), 591--627 (2011).

\bibitem[Zha12]{Zhang2012}
W.~Zhang, \textsl{ On arithmetic fundamental lemmas},
\newblock Invent. Math. \textbf{ 188}(1), 197--252 (2012).

\bibitem[{Zha}19]{Zhang2019}
W.~{Zhang}, \textsl{ {Weil representation and Arithmetic Fundamental Lemma}},
\newblock arXiv e-prints , arXiv:1909.02697 (Sep 2019), {1909.02697}.

\end{thebibliography}
\end{document}